\documentclass[11pt]{article}

\usepackage[all,2cell]{xy} \UseAllTwocells \SilentMatrices
\usepackage{amsfonts,amsmath}
\usepackage{amssymb,latexsym}
\usepackage{amscd}
\usepackage{amsthm}
\usepackage[hmargin=3cm,vmargin=3.5cm]{geometry}
\usepackage{hyperref}

\normalfont\upshape

\usepackage{fancyheadings}
\title{{\bf Hopfological Algebra}}
\author{You Qi}
\date{\today}

\theoremstyle{plain}
\newtheorem{thm}{Theorem}[section]
\newtheorem{prop}[thm]{Proposition}
\newtheorem{lemma}[thm]{Lemma}
\newtheorem{cor}[thm]{Corollary}

\theoremstyle{definition}
\newtheorem{defn}[thm]{Definition}
\newtheorem{eg}[thm]{Example}
\newtheorem*{define}{Definition} 
\newtheorem{rmk}[thm]{Remark}

\begin{document}

\maketitle \baselineskip 14pt

\def\R{\mathbb R}
\def\Q{\mathbb Q}
\def\Z{\mathbb Z}
\def\N{\mathbb N}
\def\C{\mathbb C}
\def\l{\lbrace}
\def\r{\rbrace}
\def\o{\otimes}
\def\Lo{\otimes^\mathbf{L}}
\def\lra{\longrightarrow}
\def\HOM{\mathrm{HOM}} 
\def\RHom{\mathbf{R}\mathrm{Hom}} 
\def\RHOM{\mathbf{R}\mathrm{HOM}} 
\def\Id{\mathrm{Id}}
\def\mc{\mathcal}
\def\mf{\mathfrak}
\def\Ext{\mathrm{Ext}}
\def\End{\mathrm{End}}
\def\yesnocases#1#2#3#4{\left\{
\begin{array}{ll} #1 & #2 \\ #3 & #4
\end{array} \right. }

\def\gdim{{\mathrm{gdim}}}
\def\dmod{{\mathrm{-mod}}}
\def\vect{\Bbbk\!-\!\mathrm{vect}}
\def\dgmod{{\!\mathrm{-gmod}}}
\def\udmod{\!-\!\underline{\mathrm{mod}}}
\def\udgmod{\!-\!\underline{\mathrm{gmod}}}
\def\Pf#1{\begin{proof}#1
\end{proof}} 
\def\kom#1{\mathcal{M}_{#1}}
\def\com#1{\mathcal{K}_{#1}}
\def\D#1{\mathcal{D}(#1)}
\newcommand{\Hom}{{\rm Hom}}

\newcommand{\Ind}{{\rm Ind}}
\newcommand{\Res}{{\rm Res}}

\newcommand{\Sone}{\mc{S}'} 
\newcommand{\Stwo}{\mc{S}} 

\begin{abstract}
We develop some basic homological theory of hopfological algebra
as defined by Khovanov \cite{Kh}. Several homological properties in
hopfological algebra analogous to those of usual homological theory
of DG algebras are obtained.
\end{abstract}

\tableofcontents


\section{Introduction}

Since its birth, homological algebra has commonly been regarded as
being centered around the equation $d^2=0$. Such a view can be best
seen through the famous quote of Henri Cartan:

\begin{quotation}
If I could only understand the beautiful consequence following from
the concise proposition $d^2=0$.

\hfill -Henri Cartan.\footnote{ See the foreword of \cite{GM}.}
\end{quotation}

Thus it is a natural question to ask whether and how we could deform
this equation while maintaining an equally beautiful and useful
theory. Indeed, in \cite{May1,May2}, Mayer defined a ``new
simplicial homology'' theory over a field of characteristic $p>0$ by
forgetting the usual alternating signs in the definition of boundary
maps. The boundary maps satisfy $\partial^p=0$, and associated with
this kind of ``$p$-chain complex'' one obtains the ``$p$-cohomology
groups'' $\text{Ker}(\partial^q)/\text{Im}(\partial^{p-q})$, for any
$1\leq q \leq p-1$. Furthermore, when applied to singular chains on
topological spaces, this construction results in a ``new homology
theory'' which is a topological invariant of the underlying space!
Exciting as it might seem, however, Spanier \cite{Spa} soon found
out that these homology groups can be recovered from the usual
singular homology groups, due to the restrictions placed on any
topological homology theory by the Eilenberg-Steenrod axioms. This
immediately extinguished most of the interest in Mayer's invariant,
and people paid little attention to these pioneering works on
$p$-complexes; they remained buried among historical documents
until several decades later. In 1996, Kapranov \cite{Ka}, and independently Sarkaria \cite{Sar2}, studied a
``quantum'' analogue of the equation $d^p=0$, working over a field
of characteristic zero with $n$-th roots of unity (e.g. the $n$-th
cyclotomic field $\Q[\zeta_n]$). The analogous construction yields
$n$-complexes where the boundary maps satisfy $d^n=0$ for some $n
\in \N$. Similar homology groups of these complexes as in
\cite{May1,May2} are defined. This construction, as a purely
algebraic object, rekindled more interest this time and found
applications in theoretic physics. Nowadays there is a vast
collection of literature on the subject. See, for instance, Angel-D\'{\i}az \cite{AD},
Bichon \cite{Bi}, Cibils-Solotar-Wisbauer \cite{CSW},
Dubois-Violette \cite{Dub}, Sitarz \cite{Si}, Kassel-Wambst \cite{KW},
and many of the references therein. It is worth mentioning
that \cite{KW} put both $d^p=0$ and $d^n=0$ on equal footing, and
developed some general homological theory for both cases.

Meanwhile, Pareigis \cite{Pa} reinterpreted the usual homological
algebra over a base ring $K$ as (co)modules over a non-commutative,
non-cocommutative Hopf algebra. In fact, using Majid's
``bosonisation process'' \cite{Maj}, one can understand this Hopf
algebra as a graded Hopf-algebra object $K[d]/(d^2)$ in the category
of graded super modules over the ground ring $K$. Similar
reformulations for the deformations $d^n=0$ were given by Bichon
\cite{Bi}. One crucial feature of such Hopf algebras used by these
authors is that their (co)module categories are Frobenius. Indeed,
finite dimensional Hopf algebras or objects bearing enough similar
properties are well-known to have a left (co)integral, which in turn
can be used to define non-degenerate associative bilinear forms on
the algebras. See for instance \cite{Ku} for an arrow-diagrammatic
proof of this result.

To this end, the work of Khovanov \cite{Kh} can be
regarded as a general framework to unify both points of view about
the homological algebra of $d^n=0$. There he considers (co)module
algebras over any finite dimensional Hopf algebra (or a finite
dimensional Hopf-algebra object in some category). In this
framework, Mayer's original $p$-complexes can be identified with
(co)modules over the $\Z$-graded finite dimensional Hopf algebra
$\Bbbk[\partial]/(\partial^p)$, where $\Bbbk$ is a field of
characteristic $p>0$. Moreover the usual notion of a differential
graded algebra (DGA) can be reinterpreted as a module-algebra over
the graded Hopf super algebra $K[d]/(d^2)$, and therefore affords a
generalization to arbitrary module-algebras over finite dimensional
Hopf algebras, among which the Hopf algebra $\Bbbk[\partial^p]/(\partial^p)$
over a field of characteristic $p>0$
is the simplest example. Nonetheless, one question dating back to
Mayer-Spanier should still be addressed: why should we care about
this construction if its homology gives us nothing new?

One answer to this question was given by Khovanov in \cite{Kh}.
Instead of homology, the Grothendieck groups $K_0$ of the
triangulated (stable) categories $H\udgmod$ are isomorphic to the
$p$-th (equivalently the $2p$-th) cyclotomic integers
$\Z[\zeta]/(1+\zeta+ \cdots+\zeta^{p-1})\cong \Z[\zeta_p]$.
Furthermore, the (triangulated) module category over such a Hopf
module-algebra inherits a (triangulated) module category structure.
Therefore the Grothendieck group of such a module category will be a
module over the ring of cyclotomic integers. Finding interesting
such module-algebras could potentially realize the dreams dating
back to Crane-Frenkel on categorification of quantum three-manifold
invariants at certain roots of unity and extend them into 4d
topological quantum field theories \cite{CF}. With this motivation,
Khovanov coined the terminology ``hopfological algebra'' since this
new framework is a mixture of Hopf algebra and homological algebra.
We follow  his suggestion and use this term vaguely to refer to the
general homological theory of Hopf module-algebras and their module
categories.

In the present work, we develop some general homological properties
of hopfological algebra (or following \cite{Kh}, we should say
``hopfological properties'') in analogy with the usual homological theory
of DG algebras. The
strategy is rather straightforward since there are now beautiful
structural expositions on DG algebras to mimic, such as the book by
Bernstein and Lunts \cite[Section 10]{BL}, the less formal and very
readable online lecture notes by Kaledin \cite{Kal}, or the papers
of Keller \cite{Ke1,Ke2}. We will mainly follow Keller's approach in
\cite{Ke1}.

Now we give a rough summary of the content of this paper. We start
by briefly reviewing Khovanov's original constructions in the first
three sections and giving ways to construct distinguished triangles
in the ``homotopy'' and ``derived'' categories of hopfological
modules, in analogy with DG algebras. Then we analyze more closely
the morphism spaces in the homotopy category, which is needed to
define the notion of cofibrant hopfological modules. As in the DG
case, we show that any hopfological module has a cofibrant replacement
(Theorem \ref{thm-bar-resolution}), and the morphism spaces between
cofibrant objects in the derived category coincide with their
morphism spaces in the homotopy category. Such cofibrant
replacements are also needed to define derived functors and to
construct derived equivalences of different hopfological module
categories. Next, we show that the derived categories of
hopfological modules are compactly generated, and this allows us to
use the formidable machinery of Ravenel-Neeman \cite{Ra, Nee, Nee2}
to give a characterization of compact objects in the derived category
(Corollary \ref{cor-characterizing-compact-objects}),
as well as to make precise the definition of Grothendieck groups of
hopfological module categories. Finally, a restrictive version of Morita
equivalence between derived categories is given (Corollary
\ref{cor-morita-equivalent-hopfological-algebra}).
Throughout, the
general theory is illustrated by three specific examples in parallel comparison, namely
the usual DG algebra, Kapranov-Sarkaria's $n$-DG algebra, and
Mayer's $p$-DG algebra.

As this paper will mainly serve as a tool kit for our work in
progress on categorification at roots of unity, there are some
important caveats we have to make clear. The first remark to make is
that we do not attempt to develop hopfological theory for Kapranov's
characteristic zero ``$n$-differential graded algebra'' in full
generality. In Section 8, we need to assume that the underlying Hopf
algebra be \emph{(co)commutative}. One reason is that, given a left
$H$-module algebra $A$, we could not find a natural way to define a
left $H$-module algebra structure on $A^{op}$ for arbitrary $H$. Another
problem is that, given two module-algebras equipped with $n$-differentials
(i.e. $d(ab)=d(a) b + \zeta^{\text{deg}(a)}ad(b)$, and $d^n=0$ for
any elements $a,~b\in A$), there does not seem to be a natural way
to define a module-algebra structure on the tensor product algebra. This problem
was already pointed out in \cite{Si}.
Such a monoidal structure plays a very important role in
many existing examples of categorification, for instance \cite{KL}.
Secondly, we will not develop in this paper the full analogue of DG
Morita theory (as in Keller \cite{Ke1}), as we wish to control the
length of the paper. Such a theory might be better treated in a more
categorical setting than the one we use here. In subsequent works we
will investigate this question in parallel with To\"{e}n's framework
\cite{To} on DG categories, as well as more related $K$-theoretical
questions.

\paragraph{Acknowledgements.} The author would like to express his
deep gratitude towards his advisor Professor Mikhail Khovanov for
many discussions and encouragements on the project, and more
importantly, a great deal of inspirational education throughout the
past years. He would also like to thank Professors A. Johan de
Jong, Bernhard Keller and Dmitri Orlov for many oral and email
communications on DG categories and some geometric examples, and his
friends Ben Elias, Alex P. Ellis, Zach Maddock, Kris Putyra and Josh Sussan for
reading a first draft of the manuscript and suggesting numerous
helpful corrections. The author was partially supported by the Columbia
topology RTG grant DMS-0739392.


\section{Module categories}
In this and the next two sections we review the basic constructions
of hopfological algebra, following~\cite[Sections 2.1-2.3]{Kh}. Then
we will develop some basic properties of hopfological algebra,
adapting the framework for DG-categories (algebras) in \cite{Ke1}.
Our goal is to show that, as predicted in \cite{Kh}, a fair amount of
the general theory of DG-algebra generalizes to hopfological
algebra.

\subsection{The base category}
Let $H$ be a finite dimensional Hopf algebra over a field
$\Bbbk$. We denote by $\Delta$ the comultiplication, by $\epsilon$
the counit, and by $S$ the antipode of $H$. It is well-known that
$S$ is an invertible algebra anti-automorphism. We will fix a
non-zero left integral $\Lambda$ of $H$ once and for all, which is
uniquely determined (see, for instance, Corollary 3.5
of~\cite[Section 3]{Ku}), up to a non-zero constant in the ground
field $\Bbbk$ by the property that, for any $h\in H$,
$$h\Lambda=\epsilon(h)\Lambda.$$
The category $H\dmod$ of left $H-$modules is monoidal, with $H$
acting on the tensor product $M \o N$ of two $H$-modules $M$ and $N$
via the comultiplication $\Delta$. In what follows, we will
constantly use the Sweedler notation: for any $h\in H$,
$\Delta(h)=\sum_{(h)} h_{(1)}\o h_{(2)}\in H\o H$, and we will omit
the summation symbol if no confusion can arise. Moreover, we will
freely use the fact that, for any $h\in H$,
$h_{(2)}S^{-1}(h_{(1)})=\epsilon(h)=S^{-1}(h_{(2)})h_{(1)}$, which
follows by applying the anti-automorphism $S^{-1}$ to the axiom
$h_{(1)}S(h_{(2)})=\epsilon(h)=S(h_{(1)})h_{(2)}$.

By convention, when a tensor product sign $\o$ is undecorated, we
always mean that it is over the base field $\Bbbk$. Moreover, when
tensor products ``$\o$'' and direct sums ``$\oplus$'' appear
together without brackets, tensor products always take precedence
over direct sums. By modules over an algebra we will always mean
left modules over the algebra unless otherwise stated.

\begin{prop} \label{prop-basic-property-H-mod}
\begin{enumerate}
\item For any $H$-module $M$, we have a canonical isomorphism of
$H$-modules $M\o H\cong M_0\o H$, where $M_0$ denotes $M$ as a
$\Bbbk$-vector space equipped with the trivial $H$-module structure.
\item $H$ is a Frobenius algebra, so that it is self-injective.  The
associated stable module category $H\udmod$ is triangulated
monoidal.
\item The shift functor $T$ on $H\udmod$ is given as follows:
for any $H$-module $M$, let $M \subset I$ be the inclusion of $M$
into the injective $H$-module $I=M \o H$,
given by $\Id_M \o \Lambda :M \lra M \o H$. Then $T(M)$ is defined
to be the cokernel of this inclusion:
$$
T:H\udmod  \lra  H\udmod, \quad M  \mapsto M \o (H/\Bbbk \Lambda).
$$
\item The tensor product of $H$-modules descends to an exact bifunctor on
$H\udmod$ $$\o :H\udmod\times H\udmod \lra H\udmod,$$ which is
compatible with the shift functor above. $H\udmod$ is symmetric
monoidal if $H$ is cocommutative. Here compatibility means that, for
any $M,~N\in H\dmod$,
$$T(M)\o N \cong T(M\o N) \cong M\o T(N).$$
\end{enumerate}
\end{prop}
\begin{proof} We give the proof of part 1 here. The rest of the
statements are proved in \cite[Section 1]{Kh}. We define a map of
$H$-modules: $f_M:M\o H\lra M_0\o H$ by sending $m\o l\mapsto
S^{-1}(l_{(1)})m\o l_{(2)}$, for any $l \in H$, $m\in M$. Then we
check that it's an $H$-module map: for any $h\in H$,
$$
\begin{array}{rcl}
f_{M}(h(m\o l))& = & f_{M}(h_{(1)}m \o h_{(2)}l)  =
S^{-1}((h_{(2)}l)_{(1)}) h_{(1)}m \o (h_{(2)}l)_{(2)}\\
& = &S^{-1}(l_{(1)})S^{-1}(h_{(2)})h_{(1)}m \o h_{(3)}l_{(2)}
 =  S^{-1}(l_{(1)})\epsilon(h_{(1)})m \o h_{(2)}l_{(2)}\\
& = & S^{-1}(l_{(1)}) m \o hl_{(2)}= h f_{M}(m \o l),
\end{array} $$
where we used that $S^{-1}(h_{(2)})h_{(1)}=\epsilon(h)$ and
$h_{(1)}\epsilon(h_{(2)})=h$. Notice that in the second to the last
equality, $h$ only acts on the second factor. Finally, $f_{M}$ is
invertible whose two sided inverse is given by $f_M^{-1}: M_0\o H
\lra M\o H$, $m\o h \mapsto h_{(1)}m\o h_{(2)}$. We leave this
verification to the reader.
\end{proof}

We briefly remind the reader of the notion of a stable category
associated with a Frobenius category (e.g. modules over a Frobenius
algebra), and this will explain some of the notations we used in the
above proposition. For more details, see \cite[Section 2, Chapter
1]{Ha}. An abelian category $\mc{C}$ (e.g. $H\dmod$) is called
Frobenius if it has enough injectives and enough projectives, and
moreover the class of injectives coincides with that of the
projectives. If $\mc{C}$ is such a category, we denote by
$\underline{\mc{C}}$ the stable category associated with it, whose
objects are the same as that of $\mc{C}$, and the morphism space
between any two objects $X,Y\in \textrm{Ob}(\mc{C})$ are constructed
as the quotient
$$\Hom_{\underline{\mc{C}}}(X,Y):=\Hom_\mc{C}(X, Y)/I(X,Y),$$
where $I(X,Y)$ stands for the space of morphisms between $X$ and $Y$
in $\mc{C}$ that factor through an injective($=$ projective) object
in $\mc{C}$. Theorem 2.6 of \cite[Section 2, Chapter 1]{Ha} shows
that $\underline{\mc{C}}$ is triangulated. The translation
endo-functor of $T:\underline{\mc{C}}\lra \underline{\mc{C}}$ is
given as follows. For any $X\in \textrm{Ob}(\mc{C})$, choose a
monomorphism $\lambda_X: X \lra I(X)$ of $X$ into an injective
object $I(X)$. We define $T(X):=I(X)/\textrm{Im}(\lambda_X)$,
considered as an object of $\underline{\mc{C}}$. It can be checked
that the isomorphism class of $T(X)$ in $\underline{\mc{C}}$ is
independent of choices of $I(X)$, and this leads to
a well-defined functor on $\underline{\mc{C}}$. Happel also shows that
$T$ is an automorphism of $\underline{\mc{C}}$ (Proposition 2.2
of \cite[Chapter 1]{Ha}), and it is readily checked that its inverse is
given as follows: for any $X \in \textrm{Ob}(\mc{C})$, take an
epimorphism from a projective object $\mu_X:P(X)\lra X$, then
$T^{-1}(X):=\textrm{ker}(\mu_X)$, regarded as an object in
$\underline{\mc{C}}$. Finally, every short exact sequence of objects
in $\mc{C}$ descends to a distinguished triangle in
$\underline{\mc{C}}$, and conversely any distinguished triangle in
$\underline{\mc{C}}$ is isomorphic to one that arises in this way
(Lemma 2.7 \cite[Chapter 1]{Ha}).

\begin{eg} \label{eg-hopf-algebras} We give some simple examples
of finite dimensional (graded, super) Hopf algebras and their left
integrals.
\begin{itemize}
\item Let $G$ be a finite group and $H=\Bbbk G$ be its group ring over a
field $\Bbbk$. Then $H$ is a Hopf algebra with $\Delta(g)=g \o g$,
$S(g)=g^{-1}$ and $\epsilon(g)=1$, for any $g\in G$. The element
$\sum_{g\in G}g$ spans the space of (left and right) integrals.
\item Let $V$ be an $(n+1)$-dimensional vector space over a field
$\Bbbk$, and let $H=\mathbf{\Lambda}^{*}V$ be the exterior algebra
over $V$. Then $H$ becomes a graded super Hopf algebra if we define
any non-zero element $v \in V$ to be of degree one; $\Delta(v)=v\o
1+ 1\o v$; $S(v)=-v$; $\epsilon(v)=0$. The space spanned by a
non-zero (left and right) integral can be canonically identified
with $\mathbf{\Lambda}^{n+1}(V)\cong \Bbbk v_0 \wedge \cdots \wedge
v_n$, where $\{v_0,\cdots, v_n\}$ forms a basis of $V$.
\item Let $\Bbbk$ be a field of positive characteristic $p$. Let $H=
\Bbbk[\partial]/(\partial^p)$, with $\Delta(\partial)=\partial\o
1+1\o \partial$, $S(\partial)=-\partial$, and
$\epsilon(\partial)=0$. $H$ will be graded if we fix a degree for
$\partial$. The space of (left and right) integrals in $H$ is
spanned by $\partial^{p-1}$.
\item Let $H_n$ be the Taft algebra (see \cite{Bi} or \cite[Section
4, Characteristic 0 case]{Kh}) over the $n$-th cyclotomic field
$\Bbbk=\mathbb{Q}[\zeta]$, where $\zeta$ is a primitive n-th root of
unity. As a $\Bbbk$-algebra, $H_n$ is generated by $K$,$K^{-1}$ and
$d$, subject to the relations $K^{-1}K=KK^{-1}=1$, $K^n=1$,
$Kd=\zeta dK$, and $d^n=0$. $H_n$ is an $n^2$-dimensional Hopf
algebra with $\Delta(K)=K\o K$, $\Delta(d)=d\o 1+K\o d$,
$S(K)=K^{-1}$, $S(d)=-K^{-1}d$, $\epsilon(K)=1$, $\epsilon(d)=0$.
It is easily checked using the commutator relations that a non-zero
left integral is given by
$\Lambda_l=\frac{1}{n}(\sum_{i=0}^{n-1}K^i)d^{n-1}$, while a
non-zero right integral is given by
$\Lambda_r=\frac{1}{n}d^{n-1}(\sum_{i=0}^{n-1}K^i)$.
\end{itemize}
\end{eg}

The following lemma is a slight generalization of Proposition 2 of
\cite[Section 1]{Kh}, which will be needed for technical reasons
later.

\begin{lemma}\label{lemma-infinite-projective-H-modules}
 Let $M$ be an arbitrary $H-$module and $N$ be a projective
$H-$module. Then $M\o_\Bbbk N$, $\Hom_\Bbbk(M,N)$ and
$\Hom_\Bbbk(N,M)$ are projective as $H-$modules. The $H$-module
structures are defined in the usual way: for any $h\in H$, $m\in M$,
$n\in N$, $f \in \Hom_\Bbbk(M,N)$,
$$h\cdot(m\o n) :=\sum h_{(1)}\cdot m \o h_{(2)}\cdot n,$$
$$(h\cdot f)(m):=\sum h_{(2)}\cdot f(S^{-1}(h_{(1)})\cdot m).$$
\end{lemma}
\begin{proof} The case when either one of $M$ or $N$ is finite
dimensional follows from Proposition 2 of \cite{Kh}. When both $M$
and $N$ are infinite dimensional, we can write $M$ as a union of its
finite dimensional submodules $M=\cup_{i\in I}M_i$ where $I$ is some filtered
partially ordered set, with $i \leq j$ in $I$ if and only if $M_i\subset M_j$. In
other words, we regard $I$ as a small filtered category in which there is an arrow $i\lra j$ if and
only if $M_i\subset M_j$, and then $M$ is the colimit of $I$.
We also write $N$ as a direct sum of finite dimensional injective (= projective) modules
$N=\oplus_{j\in J} P_j$. Now the tensor product is injective since we can write it
as
$$M\o N \cong M\o (\bigoplus_{j\in J}P_j)\cong \bigoplus_{j\in J}M\o P_j.$$
which is a direct sum of injectives \footnote{Any product of
injectives over a ring is injective; an infinite direct sum of
injectives is injective if and only if the ring is
noetherian \cite[Theorem 3.46]{Lam}.}, where each term $M\o P_j$ is
injective by Proposition 2 of \cite{Kh}.

Next, $\Hom_\Bbbk(N,M)$ can be rewritten as
$$\Hom_\Bbbk(\bigoplus_{j\in J}P_j,M)\cong \prod_{j\in J}\Hom_{\Bbbk}
(P_j,M)\cong \prod_{j\in J}P_j^* \o M .$$ Each $P_j^*$ is injective
since $P_j$ is also finite dimensional projective, and we
are again reduced to the case of Proposition 2 of \cite{Kh}.

Finally, for $\Hom_\Bbbk(M,N)$, we use the short exact sequence of
vector spaces
$$0\lra \bigoplus_{(i \lra j)\in I} M_{i}\stackrel{\Psi}{\lra} \bigoplus_{k
\in I}M_k\lra M \lra 0,$$
where the first direct sum is over all arrows in $I$, the second direct sum
is over all objects of $I$, and $\Psi$ restricted on each summand $M_i$ labeled by $i\lra j$ is given by
composing
$$M_i \lra M_i\oplus M_j;~m_i\mapsto (m_i, -m_i)$$
with the natural inclusion map
$$ M_i \oplus M_j \hookrightarrow \bigoplus_{i\in I}M_i.$$
Applying $\Hom_\Bbbk(-,N)$ to this exact sequence, we get a short
exact sequence of $H-$modules:
$$0\lra\Hom_\Bbbk(M,N)\lra \prod_{k\in
I}\Hom_\Bbbk(M_k,N)\stackrel{\Psi^*}{\lra} \prod_{(i\lra j) \in
I}\Hom_\Bbbk(M_i,N)\lra 0.$$
Notice that
$$\prod_{k\in I} \Hom_\Bbbk(M_k, N)\cong \prod_{k\in I}(\bigoplus_{j\in J}
\Hom_\Bbbk(M_i,P_j)) \cong \prod_{i\in I}(\bigoplus_{j\in J}P_j\o
M_i^*),$$
so that it is injective once again by the finite
dimensional case \cite[Proposition 2]{Kh}. Likewise for the last
term in the short exact sequence. It
follows that the above sequence of H-modules splits, and
$\Hom_\Bbbk(M,N)$ is injective.
\end{proof}

\subsection{Comodule algebras and stable module categories}
Now we recall the notion of a (right) comodule-algebra over $H$. We
slightly modify the convention used in~\cite{Kh} to better suit the
special case of DG-algebras over the base field $\Bbbk$. In
particular we will be mainly using the notion of right
H-comodule-algebras as opposed to left comodule-algebras. The proofs
of~\cite{Kh} go through almost unchanged with appropriate ``left''
notions switched to the ``right'' ones.

\begin{defn} \label{def-right-comodule-algebra}A \emph{right
$H$-comodule-algebra}
$B$ is a unital, associative $\Bbbk-$algebra equipped with a map
$$\Delta_B: B \lra B \o H$$
making $B$ into a right $H$-comodule and such that $\Delta_B$ is a
map of algebras. Equivalently, we have the following identities:

\begin{eqnarray*}
(\Id_B \o \epsilon)\Delta_B=\Id_B,& (\Id_B \o \Delta)\Delta_B=
(\Delta_B \o \Id_H)\Delta_B,\\
\Delta_B(1)=1\o 1,& \Delta_B(ab)=\Delta_B(a)\Delta_B(b).
\end{eqnarray*}
Here $B \o H$ is equipped with the product algebra structure.
\end{defn}

Let $V$ be an $H$-module, and $M$ be a $B$-module. The tensor product
$M \o V$ is naturally a $B$-module, via $\Delta_B$. The tensor
product gives rise to a bifunctor
$$B\dmod \times H \dmod \lra B \dmod$$
compatible with the monoidal structure of $H\dmod$, and in turn this
makes $B\dmod$ into a (right) module-category over $H\dmod$.

\begin{defn}\label{def-null-homotopy-category-of-B}Let $B_H\udmod$ be
the quotient category of $B\dmod$ by the ideal of morphisms that
factor through a $B-$module of the form $N \o H$, where $N$ is some
$B-$module.
\end{defn}

More precisely, we call a morphism of $B$-modules $f:M_1 \lra M_2$
\emph{null-homotopic} if there exists a $B$-module $N$ such that $f$
factors as
$$M_1 \lra N \o H \lra M_2.$$
The space of null-homotopic morphisms forms an ideal in $B\dmod$.
The quotient category $B_H\udmod$ by this ideal by definition has
the same objects as $B\dmod$, while the  $\Bbbk-$vector space of
morphisms in $B_H\udmod$ between any two objects $M_1$, $M_2$ is the
quotient of $\Hom_{B}(M_1,M_2)$ by the subspace of
null-homotopic morphisms.

We also recall the following useful lemma, which gives an
alternative characterization of the ideal of null-homotopic
homomorphism.

\begin{lemma}\label{lemma-characterizing-null-homotopy}
A map $f:M \lra N$ of $B$-modules is null-homotopic if and only if
it factors through the map $M \xrightarrow{\Id_M \o \Lambda} M\o
H$.
\end{lemma}
\Pf{This is Lemma 1 of \cite[Section 1]{Kh}.}

As a matter of notation, we will denote the canonical $B$-module map
in the lemma by $\lambda_M:M \xrightarrow{\Id_M \o \Lambda} M\o H$
for any $B$-module $M$, as such maps will appear repeatedly in what
follows.

\begin{prop} \label{module-category}
$B_H\udmod$ is a (right) module-category over $H\udmod$.
\end{prop}
\Pf{The tensor product $B \dmod\times H\dmod \lra B\dmod$ descends
to a bifunctor
$$B_H \udmod \times H \udmod \lra B_H \udmod,$$
compatible with the monoidal structure of $H\udmod$.}

We will be mainly interested in the following class of examples. See
example (g) of \cite[Section 1]{Kh}, or \cite[Chapter 4]{Mo}.

\begin{eg}[The main example] \label{the-main-example}
Let $A$ be a left $H$-module algebra. This means that $A$ is a left
$H$-module, and the multiplication and unit maps of $A$ are left
$H$-module maps. An excellent treatise for such algebras is
\cite{Mo}, which gives a detailed survey of recent research on such
module-algebras and their ring theoretical properties.

\begin{define}
The \emph{smash product algebra} $B=A\#H$ is the $\Bbbk-$vector
space $A\o H$ with the multiplication:
$$(a\o h)(b\o l)=\sum_{(h)} a (h_{(1)}\cdot b)\o h_{(2)}l .$$
Here ``$\cdot$'' denotes the left $H$ action of $h_{(1)}$ on $b$.

$B$ has the structure of a right $H$-module algebra by setting
$\Delta_B:B\lra B\o H$, $\Delta_B(a\o h):=a\o \Delta(h)$ for any
$a\o h \in B$. We will loosely refer to the class of modules over
this kind of smash product ring $B$ as \emph{hopfological modules}.
\end{define}

As special cases of this main example, we have:
\begin{enumerate}
\item If $A=\Bbbk$ with the trivial module structure over $H$, then
$A=\Bbbk \# H=H$. We recover the usual stable category of $H$:
$B_H\udmod=H\udmod$.
\item Slightly more generally, let $A$ be any $\Bbbk$-algebra with the
trivial $H-$module structure. Then $B= A \o H$. We will see later
that the usual notion of chain complexes of modules over the algebra
$B$, or their ``n-complex'' analogs \cite{Ka, Dub}, are examples of
this particular case. We will deal with a more specific class of
examples of this kind in the last section.
\end{enumerate}
\end{eg}


\section{Triangular structure} Now let us recall the shift functor,
the cone construction, and the triangles in $B_H\udmod$. See
\cite[Section 1]{Kh}. We refer the reader to \cite[Chapter IV]{GM}
and \cite[Chapter I]{Ha} for more information about triangulated
categories.

\subsection{The shift functor}
The shift (or translation) functor $T$ on $B_H\udmod$ is the functor
that $B_H\udmod$ inherits from $T$ of $H\udmod$, where we regard
$B_H\udmod$ as a module category over $H\udmod$ (see Proposition
\ref{module-category} above). More precisely, we define:

\begin{defn}\label{def-shift-functor}
For any left $B$-module $M$, let $T(M)$ be
$$T(M):=M \o (H/(\Bbbk \Lambda)).$$
This defines a functor on $B\dmod$ and it descends to be the shift
endo-functor on $B_H \udmod$.
\end{defn}

The above definition is justified thanks to the following.

\begin{prop}\label{prop-invertible-shift-functor}
$T$ is an invertible functor on $B_H \udmod$, whose inverse $T^{-1}$
is given by
$$T^{-1}(M):=M \o \mathrm{ker}(\epsilon).$$
\end{prop}
\Pf{Omitted. This is Proposition 3 of \cite[Section 1]{Kh}.}

\subsection{Distinguished triangles}
For any $B-$module morphism $u:X \lra Y$ denote by $\underline{u}$
its residue class in the stable category $B_H\udmod$ (this and the
following $\overline{u}$ notation etc. are taken from \cite{Ha}).

\begin{defn}\label{def-cone-construction}
The cone $C_u$ is defined as the pushout of $u$ and $\lambda_X$ in
$B\dmod$, so that it fits into the following Cartesian diagram
$$\xymatrix{X \ar[r]^{u} \ar[d]_{\lambda_X} & Y\ar[d]_{v}\\
X \o H \ar[r]^-{\overline{u}} & C_{u}.\\
}$$
\end{defn}

Now, let $u:X \lra Y$ be a morphism of $B-$modules. We denote by
$\overline{\lambda_X}$ the quotient map from $X \o H$ to $TX$, so
that there is the following diagram of short exact sequences in
$B\dmod$:

$$\xymatrix{0\ar[r]\ar[d] & X\ar[r]^-{\lambda_X}\ar[d]_u & X \o H \ar[r]^{\overline{\lambda_X}}
\ar[d]_{\overline{u}} & TX\ar[r]\ar@{=}[d]&0\ar[d]\\
0\ar[r]& Y \ar[r]^v & C_u \ar[r]^w & TX\ar[r] & 0.}$$

\begin{defn}\label{def-standard-dt}A \emph{standard distinguished
triangle}
in $B_H\udmod$ is defined to be the sextuple:
$$\xymatrix{X\ar[r]^{\underline{u}} & Y\ar[r]^{\underline{v}} &
 C_u\ar[r]^{\underline{w}} & TX}$$
associated with some morphism $u$ of $B$-modules. A sextuple
$\xymatrix{X\ar[r]^{\underline{u}} & Y\ar[r]^{\underline{v}} &
Z\ar[r]^{\underline{w}} & TX}$ in $B_H\udmod$ of objects and
morphisms in $B_H\udmod$ is called a \emph{distinguished triangle}
if it is isomorphic in $B_H\udmod$ to a standard distinguished
triangle.
\end{defn}

\begin{thm}The category $B_H\udmod$ is triangulated, with the shift
functor $T$ and the class of distinguished triangles defined as
above.
\end{thm}
\Pf{Omitted. This is Theorem 1 of \cite[Section 1]{Kh}.}

\subsection{Triangulated module category}
Recall that an additive functor $F:\mc{C}\lra\mc{D}$ between
triangulated categories is called \emph{exact} if it commutes with
the respective shift functors and takes distinguished triangles to
distinguished triangles. The lemma below implies that, if  $V$ is an
$H-$module, then tensoring a distinguished triangle
$\xymatrix{X\ar[r] & Y\ar[r] & Z \ar[r] & TX}$ with $V$ gives a
distinguished triangle in $B_H\udmod$:
$$\xymatrix{X \o V\ar[r] & Y \o V\ar[r] & Z \ar[r]\o V & T(X\o V)},$$
so that tensoring with any $H-$module $V$ is an exact functor on
$B_H\udmod$. We say informally that $B_H\udmod$ is a ``triangulated
(right) module-category'' over $H\dmod$.

\begin{lemma}
There exists a functorial-in-$V$ isomorphism of $H$-modules $$r: H
\o V \lra V\o H$$ intertwining the $H$-module inclusions $\Lambda \o
\Id_V :V \lra H\o V$, and $\Id_V \o \Lambda : V \lra V \o H$.
\end{lemma}
\Pf{Omitted. See Lemma 2 of \cite[Section 1]{Kh}. We take $r$ to be
the inverse of the functorial intertwiner in the lemma there.}

\begin{rmk}[Graded versions] Before proceeding to other hopfological
constructions, we remark here that all of our constructions above
apply without much change to finite dimensional graded Hopf
algebras, finite dimensional graded Hopf super-algebras, or more
generally, any finite dimensional Hopf-algebra object in a symmetric
monoidal category which admits integrals (see \cite[Section 3]{Ku}
where a diagrammatic construction of integrals in these cases are
exhibited). A good example to keep in mind is when
$H=\Bbbk[d]/(d^2)$ is the $\Z-$graded Hopf super algebra where
$\textrm{deg}(d)=1$. As we will see, a $\Z$-graded algebra $A$ being
an $H$-module algebra means that $A$ is a differential graded (DG)
algebra over the field $\Bbbk$, as defined in \cite[Section 10]{BL}.
The categories $A\# H\dmod$, $\mc{C}(A,H)$, and $\mc{D}(A, H)$
correspond respectively to the abelian category of complexes of DG
modules over $A$, the homotopy category of complexes of DG modules
over $A$, and the derived category of DG modules over $A$, with the
latter two being triangulated. The morphism spaces in these cases are
slightly different: as we will see later, the morphism spaces are
given by the usual $\RHom$ of complexes in $\mc{C}(A,H)$ and
$\mc{D}(A,H)$, at least between ``nice'' complexes. See the first
example of \cite[Section 2]{Kh} for more details.
\end{rmk}

\subsection{Examples}
We now describe the objects of $B_H\udmod$ more explicitly for some
particular smash product algebras $B=A\# H$ (see the main example
\ref{the-main-example}). By regarding the usual notion of DG modules
over a DG algebra as a special example, we will see that examples of
this kind are naturally generalizations of the DG case.

\begin{itemize}
\item Let $H=\Bbbk[d]/(d^2)$ be the graded Hopf super algebra over
$\Bbbk$, where $\textrm{deg}(d)=1$. For a graded $\Bbbk$-algebra $A$
to carry an $H$-module structure, it is equivalent to have a degree
one differential $d:A\lra A$ satisfying the following conditions: for
any $a, b\in A$,
$$d(ab)=d(a)b+(-1)^{|a|}ad(b),\ \ \ \ d^2(a)=0,$$
i.e. $A$ is a DG algebra over $\Bbbk$. Notice that $d(1)=0$ follows
automatically from the first equation. A (left)
$A\#H$-module $M$ is an $A$-module equipped with a compatible
$H$-action. Since $H$ is generated by $d$, it suffices to specify
the $d$-action on $M$ and require it to be compatible with the
$A$-module structure on $M$ and $d$-action on $A$. This amounts to
saying that, for any $a\in A$, $m \in M$, we have
$$d(am)=d(a)m+(-1)^{|a|}ad(m),\ \ \ \ d^2(m)=0,$$
i.e. $M$ is a (left) DG module over the DG algebra $A$. We refer the
reader to \cite[Section 10]{BL} for details about the homological
properties of DG modules.

\item Let $H=H_n$ be the Taft algebra over $\Q[\zeta]$ (see
\ref{eg-hopf-algebras}), and let $A$ be an $H_n$-module algebra.
Since $K$ generate a subalgebra of $H_n$ isomorphic to the group
algebra of $\Z/n\Z$, $A$ must be $\Z/n\Z$-graded and the
multiplication on $A$ must respect this grading. For any homogeneous
element $a\in A$ of degree $|a|$, $K$ acts on $a$ by $K\cdot
a=\zeta^{|a|}a$. Furthermore, the relation $Kd=\zeta dK$ applied to
$a$ gives us $Kd(a)=\zeta^{|a|+1}d(a)$, i.e. $d(a)$ is homogeneous of
degree $|a|+1$. Equivalently, $d$ has to increase the degree by one.
Thirdly, $\Delta(d)=d\o 1+K\o d$, when applied to any product of
homogeneous elements $a_1,~a_2\in A$, imposes the differential
condition that $d(a_1a_2)=d(a_1)a_2+\zeta^{|a_1|}a_1d(a_2)$. Finally
$d^n=0$ just says that $d^n(a)=0$ for all $a\in A$. Thus we conclude
that an $H_n$-module algebra is just a $\Z/n\Z$-graded algebra
equipped with a degree one differential such that
$$d(a_1a_2)=d(a_1)a_2+\zeta^{|a_1|}a_1d(a_2), \ \ \ \ d^n(a)=0.$$
Following \cite{Bi, Dub, Ka, KW}, we say that $A$ is an
\emph{$n$-differential graded (n-DG) algebra} over $\Q[\zeta]$.
Notice that $A$ could have a $\Z$-grading since any such grading
collapses into a $\Z/n\Z$-grading. Similar as
in the DG case, an $A\#H_n$-module is equivalent to a
$\Z/n\Z$-graded $A$-module, equipped with a degree one differential
$d$, such that for any homogeneous $a\in A$, $m \in M$,
$$d(am)=d(a)m+\zeta^{|a|}ad(m),\ \ \ \ d^n(m)=0.$$
Likewise, we will call such a module an \emph{$n$-DG module}.

\item Let $\Bbbk$ be a field of positive characteristic $p$, and
$H=\Bbbk[\partial]/(\partial^p)$. This case is entirely analogous to
the above n-DG algebra case, and we just state the results. An
$H$-module algebra $A$ comes with differential $\partial$ such that
for all $a, a_1, a_2\in A$,
$$\partial(a_1a_2)=\partial (a_1)a_2+a_1\partial (a_2),\ \ \ \ \partial^p(a)=0.$$
Notice the lack of coefficients before $a_1$ on the right hand side
of the first equation. Similarly, an $A\#H$-module $M$ is an
$A$-module equipped with a differential $\partial$ on it compatible
with the $A$-module differential, i.e. for all $a\in A$, $m\in M$,
$$\partial(am)=\partial (a)m+a\partial(m)\ \ \ \ \partial^p(m)=0.$$
Algebras and modules of this kind will be refereed to as
\emph{$p$-DG algebras} and \emph{$p$-DG modules}. We can also
require some compatible grading on $\partial$, $A$ and $M$, but the
formulas remain unchanged. We leave the details to the reader.
\end{itemize}


\section{Derived categories}
From now on, we will focus on the case of the main example
\ref{the-main-example} above, where derived categories can
be defined.

\subsection{Quasi-isomorphisms}
Suppose $B=A\# H$ is the smash product of $H$ and a left $H$-module
algebra $A$. Since $H\cong \Bbbk \o H$ is a subalgebra of $B$, we
have the restriction functor from $B \dmod$ to $H\dmod$:
$$\mathrm{Res}: B\dmod\lra H\dmod.$$
This descends to an exact functor on the quotient categories
$$\underline{\mathrm{Res}}: B_H\udmod \lra H\udmod.$$
In what follows, we will introduce a new notation for the
triangulated category $B_H\udmod$ for the special case of the main
example \ref{the-main-example}:
$$\mc{C}(A,H):=B_H\udmod.$$
The notation stands informally for ``the category of chain complexes
of $A$-modules up to homotopy''. The reason for using this term will
be clear once we understand the $\Hom$ spaces better, and realize
the category $\mc{C}(A,H)$ as an analogue of the homotopy category
of DG-modules in the next section.

\begin{defn}\label{def-cohomology-and-quasi-isomorphism}
\begin{enumerate}
\item[(i).] We define the \emph{total cohomology functor} to be the restriction functor:
$$\underline{\mathrm{Res}}: \mc{C}(A, H)\udmod \lra H\udmod.$$

\item[(ii).] A morphism $f:M\lra N$ in $\mc{C}(A,H)$ is a called a
\emph{quasi-isomorphism} if its restriction
$\underline{\mathrm{Res}}(f)$ is an isomorphism in $H\udmod$.
\item[(iii).] A $B-$module $M$ is called \emph{acyclic} if $0\lra M$ is a
quasi-isomorphism.
\end{enumerate}
\end{defn}

\begin{thm}\label{derived-stable-categories}
\begin{enumerate}
\item Quasi-isomorphisms in $\mc{C}(A,H)$ constitute a localizing class.
\item The localization of $\mc{C}(A,H)$ with respect to the quasi-isomorphisms,
denoted $\mc{D}(A,H)$, is triangulated. Tensoring with any
$H-$module (on the right) is an exact functor in $\mc{D}(A,H)$.
\end{enumerate}
\end{thm}

We will call $\mc{D}(A,H)$ the derived category of $B\dmod$.

\Pf{Omitted. See Proposition 4 and Corollary 2 of \cite[Section
1]{Kh}.}

\subsection{Constructing distinguished triangles}
Now we describe how short exact sequences in the abelian category
$B\dmod$ lead to distinguished triangles in $\mc{C}(A,H)$ and
$\mc{D}(A,H)$. We start with the construction in $\mc{C}(A,H)$.

\begin{lemma}\label{lemma-split-ses-lead-to-dt}Let
$$0 \lra X \stackrel{u}{\lra} Y \stackrel{v}{\lra} Z \lra 0$$
be a short exact sequence in $B\dmod$, which is split exact as a
sequence of $A-$modules. Then associated to it there is a
distinguished triangle in $\mc{C}(A,H)$:
$$X\stackrel{\underline{u}}{\lra} Y\stackrel{\underline{v}}{\lra} Z\lra TX$$
(the connecting homomorphism on the third arrow is described in the
proof below). Conversely, any distinguished triangle in
$\mc{C}(A,H)$ is isomorphic to one that arises in this way.
\end{lemma}
\begin{proof} The converse part holds by construction,
since $\lambda_X:X\lra X\o H$ is always a split injection of $A$-modules.

Now, according to the definition (\ref{def-standard-dt}), the map
$u:X\lra Y$ gives rise to a commutative diagram in $B\dmod$:
$$\begin{gathered}
\xymatrix{& 0\ar[d] & 0\ar[d] & & \\
0\ar[r] & X\ar[r]^-{\lambda_X}\ar[d]_u & X\o H \ar[r]
\ar[d]_{\overline{u}} & TX\ar[r]\ar@{=}[d]&0\\
0\ar[r]& Y\ar[r]\ar[d]_{v} & C_u \ar[r]\ar[d]_{\overline{v}} & TX\ar[r] & 0\\
& Z\ar[d]\ar@{=}[r] & Z\ar[d] & & \\
& 0 & 0 & &. }
\end{gathered}\quad \quad \quad (\star)
$$ Therefore the cone $C_u$ fits into a short exact
sequence of $B-$modules:
$$0\lra X\o H \stackrel{\overline{u}}{\lra} C_u \stackrel{\overline{v}}{\lra}
Z \lra 0,$$
which is split exact as a sequence of $A$-modules.
Thus, we will be done
with the first half of the lemma once we establish it in the
following special case: in the short exact sequence as above,
$\overline{v}$ becomes an isomorphism in $\mc{C}(A,H)$. The
connecting homomorphism is then taken to be the composition of the
inverse of $\overline{v}$ and $C_u\lra TX$.

To prove the last claim, consider the cone of $\overline{v}$, which
fits into the commutative diagram:
\begin{equation*}
\begin{gathered}
\xymatrix{0\ar[r]\ar[d] &X \o H\ar[r]^-{\overline{u}}\ar@{=}[d] & C_u \ar[r]^{\overline{v}}
\ar[d]_{\lambda_{C_u}} & Z\ar[r]\ar[d]&0\ar[d]\\
0\ar[r]& X \o H \ar[r]^{u^\prime} & C_u\o H \ar[r]^{v^\prime} & C_{\overline{v}}\ar[r] &
0.}
\end{gathered}\quad \quad \quad (\star\star)
\end{equation*}
By assumption, the top short exact sequence splits in $A\dmod$, so does the bottom one since
the third square in $(\star\star)$ is a push-out. We will show that $C_{\overline{v}}\cong 0$ in $\mc{C}(A,H)$,
and the special case will follow since, by construction,
$$C_u\stackrel{\overline{v}}{\lra} Z \lra C_{\overline{v}} \lra T(C_u)$$
is a distinguished triangle in $\mc{C}(A, H)$.

Now we examine the $B$-module structure of $C_{u}\o H$. By tensoring the top short exact sequence with $H$
in the above diagram, we obtain the exact sequence
$$0\lra X\o H \o H \lra C_u\o H \lra Z\o H\lra 0,$$
which is $A$-split. By commutativity of the second square in $(\star\star)$,
$u^\prime: X\o H \lra C_u\o H$ factors through
$$u^\prime:X\o H \xrightarrow{\lambda_{X\o H}} X\o H\o H \lra C_u\o H.$$
Now notice that the map $H \lra H\o H$ which sends $h\mapsto h\o \Lambda$ is an
$H$-module injection, whose quotient $H \o (H/\Bbbk \Lambda)\cong
H^{(\textrm{dim}(H)-1)}$ is an injective and free summand in $H\o H$ which we write as $H^\prime$.
This is true since $H$ is self-injective (see Lemma 1 of \cite{Kh} for an explicit splitting).
Modding out the submodule $X\o H\o \Lambda$ in $C_u\o H$, which is no other than $C_{\overline{v}}$, we get a
short exact sequence of $B$-modules.
$$0\lra X\o H^\prime \stackrel{\alpha}{\lra} C_{\overline{v}} \stackrel{\beta}{\lra} Z\o H \lra 0,$$
which is also $A$-split. The next lemma then shows that
$$C_{\overline{v}}\cong X\o H^\prime \oplus Z\o H,$$
and the result follows.
\end{proof}

\begin{lemma}\label{lemma-A-summand-as-B-summand}
Let $\beta:C\lra Z\o H$ be a surjective map of $B$-modules which admits a section
in $A\dmod$. Then $Z\o H$ is a direct summand of $C$ in $B\dmod$.
\end{lemma}

\begin{proof}
Let $\gamma^\prime:Z\o H \lra C$ be a section of $\beta$ as a map of $A$-modules,
so that $\beta\circ \gamma^\prime=\Id_{Z\o H}$. Define
$$\gamma:Z\o H \lra C, \ \ \ \ z\o h \mapsto h_{(2)}\gamma(S^{-1}(h_{(1)})z\o 1).$$
Then we claim that $\gamma$ is a section of $\beta$ in $B\dmod$.

To prove the claim, we first show that $\gamma$ is $A$-linear. For any $a\in A$ and $z\o h\in Z\o H$, we have
\[
\begin{array}{rl}
\gamma(az\o h)& \hspace{-0.1in} = h_{(2)}\gamma^\prime(S^{-1}(h_{(1)})(az) \o 1) = h_{(3)}\gamma^\prime((S^{-1}(h_{(2)})a)( S^{-1}(h_{(1)})z)\o 1)\vspace{0.1in}\\
&\hspace{-0.1in} =  h_{(3)}((S^{-1}(h_{(2)})a)\gamma^\prime(( S^{-1}(h_{(1)})c)\o 1))\vspace{0.1in}\\
&\hspace{-0.1in} =  h_{(3)}(S^{-1}(h_{(2)})a) h_{(4)}(\gamma^\prime( S^{-1}(h_{(1)})z\o 1))\vspace{0.1in}\\
&\hspace{-0.1in} =  (\epsilon(h_{(2)})a) h_{(3)}(\gamma^\prime(S^{-1}(h_{(1)})z\o 1))
 =  a h_{(2)}\gamma^\prime(S^{-1}(h_{(1)})z\o 1)
 =  a \gamma(c\o h),
\end{array}
\]
with the third equality holding because $\gamma^\prime$ is $A$-linear.

Then we show that it is $H$-linear as well. If $l\in H$, $z\o h \in Z\o H$, then
\[
\begin{array}{rl}
\gamma(l(z\o h))&\hspace{-0.1in} = \gamma(l_{(1)}z\o l_{(2)}h) = l_{(3)}h_{(2)}\gamma^\prime(S^{-1}(l_{(2)}h_{(1)})(l_{(1)}z)\o 1)\vspace{0.1in}\\
&\hspace{-0.1in} =l_{(3)}h_{(2)}\gamma^\prime(S^{-1}(h_{(1)})S^{-1}(l_{(2)})l_{(1)}z\o 1)= l_{(2)}h_{(2)}\gamma^\prime(S^{-1}(h_{(1)})\epsilon(l_{(1)})z\o 1)\vspace{0.1in}\\
&\hspace{-0.1in} = lh_{(2)}\gamma^\prime(S^{-1}(h_{(1)})z\o 1)=l\gamma(z\o h).
\end{array}
\]
Finally, we show that $\gamma$ is a $B$-module section of $\beta$. Take $z\o h\in Z\o H$, we have
\[
\begin{array}{rl}
\beta(\gamma(z\o h))&\hspace{-0.1in} =\beta(h_{(2)}\gamma^\prime(S^{-1}(h_{(1)})z\o 1))= h_{(2)}\beta\gamma^\prime(S^{-1}(h_{(1)})z\o 1)) =h_{(2)}(S^{-1}(h_{(1)}z\o 1)) \vspace{0.1in}\\
&\hspace{-0.1in} = h_{(2)}S^{-1}h_{(1)}z\o h_{(3)} = \epsilon(h_{(1)})z\o h_{(2)}= z\o h,
\end{array}
\]
where in the third equality, we used that $\beta$ is $H$-linear. The claim follows.
\end{proof}

Following Happel \cite[Section 2.7]{Ha}, we describe the class of distinguished triangles in the
derived category $\mc{D}(A,H)$.

After localization, any short exact sequence of $B$-modules, not necessarily $A$-split, will lead to a
distinguished triangle in $\mc{D}(A,H)$, as below. Let
$$\xymatrix{0\ar[r]&X\ar[r]^u&Y\ar[r]^v&Z\ar[r]&0}$$
be a short exact sequence of $B$-modules. Then, similar as in the proof of Lemma
\ref{lemma-split-ses-lead-to-dt}, there is a distinguished triangle in $\mc{C}(A,H)$,
$$X\lra Y \lra C_u \lra T(X),$$
coming from the diagram $(\star)$, and $C_u$ fits into a short exact sequence of $B$-modules
$$0\lra X\o H\lra C_u \lra Z\lra 0.$$
By Proposition \ref{prop-basic-property-H-mod} shows that $X\o H$, as an $H$-module, is projective
and injective. It follows that $\overline{v}:C_u\lra Z$ is a quasi-isomorphism which becomes invertible
in the derived category. Therefore we obtain a distinguished triangle
$$X\stackrel{\underline{u}}{\lra} Y \stackrel{\underline{v}}{\lra} Z\stackrel{\underline{w}}{\lra} T(X),$$
where $\underline{w}$ is taken to be the composition of $(\overline{v})^{-1}$ by $C_u\lra TX$.

\begin{lemma}\label{ses-lead-to-dt}In the same notation as in the above
discussion, given any short exact sequence of $B-$modules $0\lra X
\lra Y \lra Z \lra 0$,
$$\xymatrix{X \ar[r]^{\underline{u}} & Y \ar[r]^{\underline{v}} & Z
 \ar[r]^-{\underline{w}}& T(X)}$$
is a distinguished triangle in $\mc{D}(A,H)$. Conversely, any
distinguished triangle in $\mc{D}(A,H)$ is isomorphic to one that
arises in this way. \hfill$\square$
\end{lemma}

\subsection{Examples}
As an immediate application of the above construction, we calculate
the Grothendieck groups ($K_0$) of the stable categories $H\udmod$
($H\udgmod$) where $H$ is among the examples we gave in
\ref{eg-hopf-algebras}. Note that in our notation, $H\udmod\cong
\mc{C}(\Bbbk, H)\cong \mc{D}(\Bbbk, H)$. Recall that $K_0(H\udmod)$
($K_0(H\udgmod)$) is the abelian group generated by the symbols
$[X]$, where $[X]$'s are isomorphism classes of finite dimensional
objects in $H\udmod$ ($H\udgmod$), modulo the relations
$[Y]=[X]+[Z]$ whenever $X\lra Y\lra Z \lra T(X)$ is a distinguished
triangle in $H\udmod$ ($H\udgmod$). More general discussion about
the Grothendieck groups of $\mc{D}(A,H)$ will be given in Section 2.6.

As a matter of notation, for any graded module $X$ over some graded
ring, we will denote by $X\{r\}$ the same underlying module but with
its grading shifted up by $r$.

\begin{itemize}
\item Let $H$ be the exterior algebra $\mathbf{\Lambda}^*V$ on an
$(n+1)$-dimensional vector space $V$ over $\Bbbk$, where we set
non-zero elements of $V$ to be of degree one. Then $H$ is a graded
Hopf super algebra and we will calculate $K_0(H\udgmod)$. Since $H$
is local with the maximal ideal $\mathbf{\Lambda}^{>0}V$, there is
only one simple $H$-module
$\Bbbk_0:=(\mathbf{\Lambda}^{*}V)/(\mathbf{\Lambda}^{>0}V)$ up to a
grading shift. Therefore $K_0(H\udgmod)$ is generated as a
$\Z[q,q^{-1}]$ module by $[\Bbbk_0]$, where
$q[\Bbbk_0]:=[\Bbbk_0\{1\}]$. Again since $H$ is local and thus
indecomposable as a left module over itself, the only relation
imposed on $[\Bbbk_0]$ comes from $H$ being the iterated extension
of the shifted simple module $\Bbbk_0$:
$$0\subset \mathbf{\Lambda}^{n+1}V \subset \cdots \subset \mathbf{\Lambda}
^{\geq k}V \subset \mathbf{\Lambda}^{\geq k-1}V \subset \cdots
\subset \mathbf{\Lambda}^{\geq 0}V=H,$$ where
$\mathbf{\Lambda}^{\geq k}V/\mathbf{\Lambda}^{\geq k+1}V \cong
(\Bbbk_{0}\{k\})^{\oplus {n+1\choose k}}$. Hence using Lemma
\ref{lemma-split-ses-lead-to-dt} inductively, we get
$$0=[H]=\sum_{k=0}^{n+1}{n+1 \choose k}q^k[\Bbbk_0]
=(1+q)^{n+1}[\Bbbk_0].$$ Therefore it follows that
$$K_0(H\udgmod)\cong \Z[q]/((1+q)^{n+1}).$$
This ring is isomorphic to the cohomology ring of the projective
space $\mathbb{P}(V)$, and this is no coincidence. In fact there is
an equivalence of triangulated categories $H\udgmod \cong
\mc{D}^b(\textrm{Coh}(\mathbb{P}(V)))$, the bounded derived category
of coherent sheaves on $\mathbb{P}(V)$ (see \cite[Section IV.3]{GM}
for the details).
\item Consider the graded Hopf algebra
$H=\Bbbk[\partial]/(\partial^p)$, where $\Bbbk$ is of positive
characteristic $p$. As shown in \cite[Section 3]{Kh},
$K_0(H\udgmod)$ is again generated by the graded simple one
dimensional module $\Bbbk_0:=H/(\partial)$, subject to the only
relation
$$0=[H]=[\Bbbk_0]+q[\Bbbk_0]+\cdots+q^{p-1}[\Bbbk_0].$$
Therefore the Grothendieck group $K_0(H\udgmod)\cong
\Z[q,q^{-1}]/(1+q+\cdots+q^{p-1})\cong \Z[\zeta]$, the ring of $p$-th
cyclotomic integers ($\zeta$, being the image of $q$, is a primitive
$p$-th root of unity). If we forget about the grading, the same
reasoning above gives us $K_0(H\udmod)\cong \Z/p\Z$, the field of
$p$ elements. It was this observation that lead Khovanov to initiate
the program of categorification at certain roots of unity. See
\cite{Kh} for more details about the motivation.
\item Let $H=H_n$ be the Taft algebra as in Example \ref{eg-hopf-algebras}.
Inverting Majid's bosonization process \cite{Maj}, one can identify
the category of $H_n$-modules with the category whose objects are
$\Z/n\Z$-graded $\Q[\zeta]$-vector spaces $\oplus_{i=0}^{n-1}V_i$,
together with a map $d:V_i\lra V_{i+1}$ of degree $1$ such that
$d^n=0$, and morphisms are homogenous degree zero maps of graded
vector spaces commuting with $d$. Under this identification, it is
readily seen that the indecomposable projective modules are
precisely the shifts of the module
$$P_0:=(\Q[\zeta] \stackrel{\cdot 1}{\lra} \Q[\zeta] \stackrel{\cdot 1}{\lra}
\cdots \stackrel{\cdot 1}{\lra} \Q[\zeta]),$$ where there are $n$
terms of $\Q$ and the starting term sits in degree zero. The
simple modules are the grading shifts of the one dimensional module
$\Q[\zeta]_0:=\Q[\zeta]$, with $d$ acting as zero. Using the same
argument as above, we see that $K_0(H_n\udmod)$ is generated as an
$\Z[q, q^{-1}]/(q^n-1)$-module by $[\Q[\zeta]_0]$ subject to the
only relation
$$0=[P_0]=[\Q[\zeta]_0]+q[\Q[\zeta]_0]+\cdots +q^{n-1}[\Q[\zeta]_0],$$
and thus $K_0(H_n\udmod)\cong \Z[q]/(1+q+\cdots+q^{n-1})$. In
particular, when $n=p$, this gives rise to a characteristic zero
categorification of the rings of the $p$-th cyclotomic integers.
\end{itemize}


\section{Morphism spaces} In this section we further
analyze the $\Hom$-spaces introduced previously for the
categories $B\dmod$ and $\mc{C}(A,H)$. We will see that they are in
fact the spaces of $H$-invariants of some naturally enriched
$\Hom$-spaces that we will introduce in this section.

\subsection{The Hopf module Hom}
As before, we assume that $H$ is a finite dimensional (graded) Hopf
algebra over $\Bbbk$, or more generally, a finite dimensional
Hopf-algebra object in some $\Bbbk$-linear symmetric monoidal
category (for an example of such an object, take a graded super Hopf
algebra in the category of graded super vector spaces). Throughout
we will continue with the assumption that $A$ is a left $H$-module
algebra and the notation $B=A\#H$ (see the main example
\ref{the-main-example}).

\begin{defn}\label{def-hom} Let $M$, $N$ be $B$-modules. The
vector space $\Hom_{A}(M, N)$ becomes an $H$-module by defining for
any $f\in \Hom_{A}(M,N)$, $m\in M$, and $h\in H$
$$(h \cdot f)(m):=\sum h_{(2)}f(S^{-1}(h_{(1)})m).$$
When $A$ and $H$ are $\Z-$graded and $M$, $N$ are graded modules, we
define the enriched $\HOM_{A}$ space to be
$$\HOM_{A}(M,N)=\bigoplus_{r\in \Z}\Hom_A(M, N\{r\}),$$
where $N\{r\}$ denotes the same underlying $A$-module $N$ with
grading shifted up by $r$, and the $\Hom$ space on the right hand
side stands for the space of degree preserving maps of graded
$A-$modules. The graded $H-$module structure on $\HOM_A(M,N)$ is
given by the same formula for homogeneous elements in $H$ as that in
the ungraded case above.
\end{defn}

It is readily seen that when $M=A$, we have $\Hom_{A}(A,N)\cong N,$
and in the graded case, $\HOM_{A}(A,N)\cong N$, both as (graded)
$H$-modules.

\subsection{The space of chain maps}
The newly defined $H-$module $\Hom_{A}(M,N)$ (resp. graded
$H-$module $\HOM_{A}(M,N)$ in the graded case) for any hopfological
modules $M$, $N$ is closely related to the $\Hom$ spaces in the
abelian category $B\dmod$ and the homotopy category $\mc{C}(A,H)$.
We clarify this relation in this section. We will mostly consider
the ungraded case, as the graded case follows by similar arguments.

To avoid potential confusion, we will denote the abstract one
dimensional trivial $H$-module by $\Bbbk_0$, i.e. $\Bbbk_0 \cong
\Bbbk\cdot v_0$, where for any $h\in H$
$$h\cdot v_0= \epsilon(h)v_0.$$
When $H$ is graded, we let $v_0$ be homogeneous of degree zero.

\begin{lemma} \label{lemma-trivial-submodules}
Let $M$, $N$ be hopfological modules over $B$. Any $f \in
\Hom_{B}(M, N)$, regarded as an element in $\Hom_A(M,N)$, spans a
trivial submodule of $H$, i.e. for all $ h \in H$,
$$h\cdot f = \epsilon(h) f.$$ Conversely, any $f \in
\Hom_A(M,N)$ on which $H$ acts trivially extends to a $B$-module
homomorphism. In other words, we have a canonical isomorphism of
$\Bbbk$-vector spaces:
$$\Hom_B(M,N)=\Hom_{H}(\Bbbk_0, \Hom_A(M,N)).$$
\end{lemma}

\begin{proof}
Since $B$ contains $H$ as a subalgebra, $f$ is $H-$linear. Therefore
for any $h\in H$, $m\in M$, we have
$$
(h\cdot f)(m)  =  h_{(2)} f(S^{-1}(h_{(1)})\cdot m) =
h_{(2)}S^{-1}(h_{(1)})f(m) = \epsilon(h)f(m).
$$
For the converse, it suffices to see that $f$ is $H-$linear:
$$
\begin{array}{rcl}
f(h\cdot m) & = &\epsilon(h_{(2)})f(h_{(1)}\cdot m) = (h_{(2)}\cdot
f)(h_{(1)}\cdot m) =  h_{(3)}\cdot f
(S^{-1}(h_{(2)})\cdot h_{(1)}\cdot m)\\
& = &
 h_{(2)} \cdot f (\epsilon{(h_{(1)})}m)  =  h \cdot f(m).
\end{array}$$
This finishes the proof of the first part of the lemma. The last
claim is clear.
\end{proof}

The right hand side of the canonical identification in the lemma
involves taking $H$-invariants, of which we now recall the
definition.

\begin{defn}\label{def-z-zero} For any $H$-module $V$,
its \emph{space of $H$-invariants}, denoted $\mc{Z}(V)$, is
defined to be the $\Bbbk$-vector space (in fact an $H$-submodule):
$$\mc{Z}(V):=\Hom_{H}(\Bbbk_0,V)\cong\{v\in V| h\cdot v = \epsilon
(h)v, \forall h \in H\}\cong V^H.$$ Likewise, when $H$ and $V$ are
graded, we define the \emph{total space of homogeneous
$H$-invariants} $\mc{Z}^*(V)$ to be the graded $\Bbbk$-vector space
$$\mc{Z}^*(V):=\HOM_{H\dgmod}(\Bbbk_0,V)\cong V^H.$$
Moreover, in the graded case, the \emph{subspace of homogeneous
degree $n$ invariants} is defined to be the homogeneous degree $n$
part of $\mc{Z}^*(V)$.
$$\mc{Z}^n(V):=\{v \in V|~\textrm{deg}(v)=n,
~h\cdot v = \epsilon (h)v,~\forall~h \in H\},$$ so that
$\mc{Z}^*(V)=\oplus_{n\in \Z}\mc{Z}^n(V)$.
\end{defn}

In this notation, we can interpret the subspace of $H$-invariants in
$\Hom_A(M,N)$ as the analogous notion  of ``the space of chain
maps'' in the DG case between two hopfological modules $M$, $N$.
Indeed, the above lemma says that
$$\Hom_B(M,N)\cong \mc{Z}(\Hom_A(M,N))=\{f \in \Hom_A(M,N)|~
h\cdot f = \epsilon (h)f, \forall h \in H\},$$ and allows us to
realize the bifunctor $\Hom_{B}(-,-)$ as the composition of functors
$$
\begin{array}{ccccc}
B\dmod \times B\dmod & \lra & H\dmod & \lra & \vect\\
 (M, N) & \mapsto  & \Hom_A(M,N) & \mapsto & \mc{Z}(\Hom_A(M,N)),
\end{array}
$$
where $\vect$ stands for the category of
$\Bbbk$-vector spaces. From now on, we will refer to
$\mc{Z}(\Hom_A(M,N))=\Hom_{B}(M,N)$ as the \emph{space of chain
maps} between the two hopfological modules $M$ and $N$.

This immediately raises the related question: What's the analogue of
the \emph{space of chain maps up to homotopy}?

\subsection{The space of chain maps up to homotopy}
Our main goal in this section is to exhibit and explain the
following commutative diagram:
$$
\xymatrix{ B\dmod\times B\dmod\ar[rr]^-{\Hom_A(-,-)}\ar[rrdd]
\ar@{..>}@/_4.6pc/[ddrrrr]_(.25){\Hom_{\mc{C}(A,H)}(-,-)}&&
H\dmod \ar[dd]^{Q} \ddrruppertwocell^{\mc{Z}}{<3>{\pi}}  \\
&& && \\
&& H\udmod \ar[rr]^{\mc{H}} && \vect.\\
&& && &&}
$$
Here $Q$ is the natural localization (Verdier quotient) functor, the slanted
arrow on the left is the composition of $Q$ with $\Hom_A(-,-)$, and
$\mc{H}$ is the functor of taking ``stable invariants'' (see
Definition \ref{def-h-0}). We put $\pi$ on a double arrow to
indicate that it is a natural transformation of the two functors
$\pi:\mc{Z} \Rightarrow \mc{H}\circ Q$. As in the usual DG case,
$\pi$ will play the role of passing from the space of cocycles to
cohomology, and we will be more precise about its definition after
the next lemma. The composition of $\mc{Z}$ with $\Hom_A(-,-)$ gives
the bifunctor $\Hom_B(-,-)$, while the functor $\mc{H}\circ Q \circ
\Hom_A(-,-)$ (we will omit $Q$ when no confusion could arise) is
just the previously defined $\Hom_{\mc{C}(A,H)}(-,-)$ of the
homotopy category, which is labeled as the dotted arrow. Therefore,
we can roughly summarize the diagram as saying that, the functor
$\mc{Z}$ of taking $H$-invariants descends to a functor $\mc{H}$ on
the stable category $H\udmod$ (this explains the terminology we use
for $\mc{H}$), and the space of stable invariants
$\mc{H}(\Hom_A(M,N))$ computes the ``chain maps up to homotopy'',
which turns out to be the same as the hom space from $M$ to $N$ in
the homotopy category $\mc{C}(A,H)$.

To do this, we first need to take a closer look at the ideal of
null-homotopic morphisms in $B\dmod$. By the definition of
null-homotopy in $B\dmod$ (see Definition
\ref{def-null-homotopy-category-of-B} and Lemma
\ref{lemma-characterizing-null-homotopy}), to construct
$\Hom_{\mc{C}(A,H)}(M,N)$, we need to mod out $\Hom_B(M,N)$ by the
subspace of morphisms that factor through the natural inclusion map
$M \stackrel{\lambda_M}{\lra} M\o H$. Denote this subspace by
$I(M,N)$. Now let us look at its preimage in $\Hom_A(M,N)$ under the
isomorphism of Lemma \ref{lemma-trivial-submodules}.

\begin{lemma}\label{lemma-ideal-null-homotopy}
Under the canonical isomorphism of Lemma
\ref{lemma-trivial-submodules}, for any two hopfological modules $M$
and $N$, the space $I(M,N)$ of null-homotopic morphisms in
$\Hom_B(M,N)$ is naturally identified with
$$I(M,N)\cong \Lambda\cdot \Hom_A(M,N),$$ where the right hand
side is regarded as a $\Bbbk$-subspace of $\mc{Z}(\Hom_A(M,N))$.
A similar result holds in the graded case as well.
\end{lemma}
\begin{proof}That $\Lambda\cdot \Hom_A(M,N)$ is contained in
$\mc{Z}(\Hom_A(M,N))$ follows easily from the left integral property
$h\cdot \Lambda=\epsilon(h)\Lambda$. We need to show that, if $f\in
\mc{Z}(\Hom_A(M,N))\cong \Hom_B(M,N)$ satisfies $f=\Lambda\cdot g$
for some $g\in \Hom_A(M,N)$, then $f$ is null-homotopic as a
$B$-module map, i.e. it factors through as $f:
M\stackrel{\lambda_M}{\lra} M\o H \stackrel{\tilde g}{\lra}N$ for
some $B$-module map $\tilde g$, and vice versa. To do this, we
extend $g$ to be a $B$-module map $\tilde g :M\o H\lra N$, by
setting $\tilde g (m \o h):=(h \cdot g)(m)$. This map $\tilde g$ is
$H$-linear since for any $h,l \in H$ and $m\in M$
$$\begin{array}{rcl}
\tilde g (h \cdot (m \o l))& = & \tilde g (h_{(1)} \cdot m\o h_{(2)}l)
 =  (h_{(2)}l\cdot g)(h_{(1)}\cdot m)\\
& = &  h_{(3)}l_{(2)}g(S^{-1}(h_{(2)}l_{(1)})\cdot h_{(1)}\cdot m) =
h_{(3)}l_{(2)}g(S^{-1}(l_{(1)})S^{-1}(h_{(2)})h_{(1)}\cdot m)\\
& = & h_{(2)}l_{(2)}g(\epsilon(h_{(1)})S(l_{(1)})\cdot m) =
h(l_{(2)}g(S^{-1}(l_{(1)})\cdot m))\\
& = & h((l\cdot g)(m))  =  h \tilde g (m\o l).
\end{array}$$

 Conversely, given an $f \in \Hom_{B}(M,N)= \Hom_{A}(M,N)^H$
which is null-homotopic, we need to exhibit a $g\in \Hom_A(M,N)$ so
that $f=\Lambda\cdot g$. The hint is to reverse the above
equalities and define $g$ to be the composition $$g: M \cong M\o 1
\hookrightarrow M \o H \stackrel{\tilde g}{\lra} N.$$ This is only
an $A$-module map, since the first identification is only an
$A$-linear. Then, for any $h\in H$, $m \in M$, we have
$$\begin{array}{rcl}
\tilde g(m\o h) & = & \tilde g(\epsilon(h_{(1)}) m \o h_{(2)}) =
  \tilde g(h_{(2)}S^{-1}(h_{(1)})\cdot m \o h_{(3)})\\
& = & \tilde g(h_{(2)}\cdot(S^{-1}(h_{(1)})\cdot m \o 1))  =
  h_{(2)} \tilde g (S^{-1}(h_{(1)})\cdot m \o 1) \\
& = & h_{(2)} g (S^{-1}(h_{(1)})\cdot m) =  (h \cdot g)(m),
\end{array}$$
where the fourth equality uses that $\tilde g$ is $H$-linear by
assumption, and the fifth equality holds by definition of $g$. Now
the lemma follows since $f(m)={\tilde g}(\lambda_M(m))={\tilde g}
(m \o \Lambda)=(\Lambda\cdot g)(m)$.
\end{proof}

In particular, when $A=\Bbbk$, we obtain an explicit way of
computing morphism spaces in the category $\mc{C}(\Bbbk,H)=H\udmod$.

\begin{cor}\label{cor-morphism-space-in-stable-H-mod}Let $H$ be a
finite dimensional Hopf algebra over $\Bbbk$. The morphism space of
two $H$-modules $M$, $N$ in the stable category $H\udmod$ is
canonically isomorphic to the quotient space
$(\Hom_{\Bbbk}(M,N))^H/(\Lambda\cdot\Hom_{\Bbbk}(M,N))$. In other
words, we have a bifunctorial isomorphism:
\[\Hom_{H\udmod}(M,N)\cong \mc{Z}(\Hom_{\Bbbk}(M,N))/ (\Lambda\cdot
\Hom_{\Bbbk}(M,N)).\] Likewise, in the graded case,
\[\HOM_{H\udgmod}(M,N)\cong \mc{Z}^*(\HOM_{\Bbbk}(M,N))/ (\Lambda\cdot
\HOM_{\Bbbk}(M,N)).\] \vskip-\baselineskip\qed
\end{cor}

The right hand side of the above isomorphism is defined for any
$H$-module $V$ in place of $\Hom_A(M,N)$, which we formalize in the
following definition.

\begin{defn}\label{def-h-0}For any $H$-module $V$ we define its
\emph{space of stable invariants} to be the $\Bbbk$-vector space
$$\mc{H}(V):=\mc{Z}(V)/(\Lambda \cdot V)\cong V^H/(\Lambda\cdot V).$$
It's readily seen that $\mc{H}:H\dmod \lra \Bbbk-\mathrm{vect}$ is a functor.
Likewise, in the graded case, we define the \emph{total space of
graded stable invariants} to be
$$\mc{H}^*(V):=\mc{Z}^*(V)/(\Lambda\cdot V)\cong V^H/(\Lambda\cdot V),$$
while the \emph{space of degree $n$ stable invariants}, denoted
$\mc{H}^n(V)$, is defined to be the homogeneous degree $n$ part of
$\mc{H}^*(V)$, for any $n\in \Z$.
\end{defn}

\begin{cor} \label{cor-H-zero-functorial}The functor $\mc{H}:H\dmod\lra
\vect$ descends to a cohomological functor
$$\mc{H}:H\udmod \lra \vect.$$Here, by cohomological we mean that
$\mc{H}$ takes distinguished triangles in $H\udmod$ into long exact
sequences of $\Bbbk$-vector spaces. Likewise, in the graded case,
$$\mc{H}^*: H\udgmod\lra \Bbbk\!-\!\mathrm{gvect},$$
$$\mc{H}^n:H\udgmod \lra \vect$$ are cohomological
functors as well.
\end{cor}
\Pf{Taking $M$ to be the trivial module $\Bbbk_0$ in
Corollary \ref{cor-morphism-space-in-stable-H-mod}, we obtain
$$\mc{H}(N)\cong\Hom_{H\udmod}(\Bbbk_0,N).$$ Thus $\mc{H}$ descends to the
stable category, and it takes distinguished triangles into long
exact sequences. The graded case follows similarly.}

\begin{rmk}[An alternative proof of Corollary
\ref{cor-H-zero-functorial}] This corollary can be proven
independent of Lemma \ref{lemma-ideal-null-homotopy}, which we give
here.
\begin{itemize}
\item{Claim:} Let $V$
be any $H$-module and $v_0\in V$ a non-zero vector on which $H$ acts
trivially. Then the inclusion map $\Bbbk v_0\hookrightarrow V$
becomes $0$ in $H\udmod$ if and only if there exists an element $v
\in V$ such that
$$\Lambda\cdot v =v_0.$$ Thus we have a canonical
isomorphism of $\Bbbk$-vector spaces
$$\Hom_{H\udmod}(\Bbbk_0, V)\cong \mc{Z}(V)/(\Lambda\cdot V) \cong
\Hom_{H\dmod}(\Bbbk_0,V)/(\Lambda \cdot V),$$ which is functorial in
$V$.
\end{itemize}
\begin{proof}[Proof of claim.] The inclusion of the trivial submodule
$$\Bbbk_0\cong \Bbbk \Lambda \hookrightarrow H$$
implies that the injective envelope of the trivial submodule $\Bbbk
\Lambda$ is a direct summand of $H$, since $H$ is self-injective (part 2
of Proposition \ref{prop-basic-property-H-mod}). Denote the
injective envelope by $I$. There is a direct sum decomposition
$H=I\oplus I'$ of $H-$modules. Let $e: H \lra I$ be the projection.
Since $\Lambda e(1)=e(\Lambda)=\Lambda\in I$, $e(1)\in I$ is
non-zero.

Now let $V$ be as in the lemma and $\Bbbk v_0 \hookrightarrow V$ be
an inclusion of a trivial submodule which becomes stably zero. Then
the inclusion map must factor through an injective module, which we
may assume to be the injective envelope of $\Bbbk v_0$:
$$\Bbbk v_0 \cong \Bbbk_0 \lra I \stackrel{f}{\lra} V.$$
The image of $e(1)$ under $f$ is nonzero since $\Lambda
f(e(1))=f(\Lambda e(1))=f(\Lambda)=v_0$. The ``only if'' part
follows by taking $v=f(e(1))$.

Conversely if we have such a $v$ that $\Lambda \cdot v= v_0$, we
will show that $V$ contains an injective summand isomorphic to $I$
containing $\Bbbk v_0$, and this will finish the proof of the lemma.
Since an injective submodule of $V$ is always a direct summand,
without loss of generality, we may assume that $V=H\cdot v:=\{h\cdot
v| h \in H\}$. Consider the following commutative diagram:
$$\xymatrix{I \ar[r] \ar[rd]_f & H \ar[d] \\
\Bbbk v_0 \ar[r] \ar[d] & H\cdot v \ar[ld]^{g}\\
 I &,}$$
where $f$ is the composition of the inclusion of $I$ into $H$ and
the action map $H \lra H\cdot v$, and $g$ exists by injectivity of
$I$ and satisfies $g(v_0)=\Lambda $. Notice that $f \neq 0$ because
$\Lambda f(e(1))= f(\Lambda e(1))=f(\Lambda)=\Lambda v =v_0$ by our
assumption. Then the composition $g\circ f$ is an endomorphism of
$I$ satisfying $g\circ f (\Lambda)=g (v_0)= \Lambda$. Since $I$ is
indecomposable, $g\circ f$ is an automorphism. Therefore $f$ is an
injective homomorphism and maps $I$ isomorphically onto its image.
Again by the injectivity of $I$, the image is a direct summand of
$H\cdot v$, as claimed. The last statement is easy.
\end{proof}
\end{rmk}

\begin{rmk}\label{rmk-possible-confusion-hom-space} One possible
confusion about the definition of $\mc{H}(\HOM_A(M,N))$ is that,
although this space plays the role analogous as the total space of
chain maps up to homotopy of all different degrees in the DG case,
the latter in turn being the total cohomology group of the usual
$\RHom$ complex, it is in general different from the total
cohomology we defined earlier using the (``stablized'') restriction
functor $\underline{\mathrm{Res}}: \mc{C}(A,H)\lra H\udgmod$ for an
arbitrary $H$. In fact by Corollary \ref{cor-H-zero-functorial},
$\mc{H}$ is cohomological, and we lose information if we forget about
its derived terms. We will return to this point later when
discussing derived functors.
\end{rmk}

We summarize the previous results of this subsection in the next
proposition, which is just a reformulation of the commutative
diagram we exhibited at the beginning of this subsection.

\begin{prop}\label{prop-comparison-hom-spaces}Let $H$ be a finite
dimensional Hopf algebra over $\Bbbk$ and $A$ be a left
$H$-module-algebra. There are identifications of bifunctors:
$$\mc{Z}(\Hom_A(-,-))\cong \Hom_B(-,-):B\dmod\times B\dmod \lra \vect,$$
$$\mc{H}(\Hom_A(-,-))\cong \Hom_{\mc{C}(A,H)}(-,-):B\dmod\times B\dmod \lra \vect,$$
i.e. for any hopfological modules $M$, $N$ over $B=A\#H$, there are
isomorphism of $\Bbbk$-vector spaces
$$\mc{Z}(\Hom_A(M,N)) \cong \Hom_{B}(M,N),$$
$$\mc{H}(\Hom_A(M,N)) \cong \Hom_{\mc{C}(A,H)}(M,N),$$
bifunctorial in $M$ and $N$.
\end{prop}

\Pf{The first identification is Lemma
\ref{lemma-trivial-submodules}, while the second follows from Lemma
\ref{lemma-characterizing-null-homotopy} Lemma
\ref{lemma-ideal-null-homotopy}, and the definition of $\mc{H}$. }

The identifications in the proposition above also show that taking
$\mc{Z}$ or $\mc{H}$ commutes with direct sums of $\Hom$ spaces. The
following corollary will be needed later when dealing with compact
objects.

\begin{cor}\label{cor-taking-cohomology-commute-with-direct-sums}Let
$I$ be any index set and $M_i$, $N_i$, $i\in I$ be hopfological
modules. Then
$$\mc{Z}(\oplus_{i\in I}\Hom_A(M_i, N_i))\cong \oplus_{i\in I}\mc{Z}
(\Hom_A(M_i,N_i));$$
$$\mc{H}(\oplus_{i\in I}\Hom_A(M_i, N_i))\cong \oplus_{i\in I}\mc{H}
(\Hom_A(M_i,N_i)).$$
\end{cor}
\Pf{This follows readily from the proposition and the fact that
$\Hom_H(\Bbbk_0,-)$ commutes with arbitrary direct sums.}

\subsection{Examples}
We will give three examples on what homotopic morphisms look like
for some of the Hopf algebras we discussed in Section 3. By Lemma
\ref{lemma-ideal-null-homotopy} these are precisely the morphisms of
the form $f=\Lambda \cdot h$ for some $h\in \Hom_A(M,N)$. Recall
that
$$(\Lambda\cdot h)(-)=\sum \Lambda_{(2)}h(S^{-1}(\Lambda_{(1)})(-)).$$

\begin{itemize}
\item When $H$ is the super Hopf algebra $\Bbbk[d]/(d^2)$, (i.e. we are in the
usual DG algebra case), $\Lambda=d$ and for any homogeneous $h\in
\Hom_A(M,N)$ of degree $|h|$,
$$d\cdot h=dh+(-1)^{|h|+1}hd.$$
The minus signs come from switching $d$ and $f$ in the category of
super vector spaces and $S^{-1}(d)=-d$. We also recall the familiar
diagram depicting a null-homotopic morphism in the DG case, for
comparison with the next two examples.
$$\xymatrix
{\cdots \ar[r]^-{d_M} & M^{i-1} \ar[d]_{f} \ar[r]^-{d_M}
\ar[dl]|*+<1ex,1ex>{\scriptstyle h} &
M^i\ar[dl]|*+<1ex,1ex>{\scriptstyle h} \ar[d]_{f}\ar[r]^-{d_M}
&M^{i+1}\ar[dl]|*+<1ex,1ex>{\scriptstyle h} \ar[d]_{f}\ar[r]^-{d_M}&
\cdots \ar[dl]|*+<1ex,1ex>{\scriptstyle h}
\\
\cdots \ar[r]_-{d_N} & N^{i-1}\ar[r]_-{d_N} & N^i\ar[r]_-{d_N}
&N^{i+1}\ar[r]_-{d_N}& \cdots . }$$

\item Let $H=H_n$ be the Taft algebra. In the examples of Section 2,
we have seen that a left integral of $H$ is given by
$\Lambda=\frac{1}{n}(\sum_{i=0}^{n-1} K^i)d^{n-1}$. Notice that if
$g=\sum_{i=0}^{n-1}g_i\in\Hom_A(M,N)$ is a decomposition of $g$ into
its homogeneous components,
$$\Lambda\cdot g=1/n(\sum_{i=0}^{n-1}K^i)g_j=1/n(\sum_{i=0}^{n-1}\zeta^{ij})g_j,$$
which can be non-zero only when $j=0$, in which case it equals
$g_0$. i.e. $1/n(\sum_{i=0}^{n-1} K^i)$ projects any vector onto its
degree zero component. Thus the effect of applying $\Lambda$ to any
$h\in \Hom_A(M,N)$ will only be seen in its homogeneous of degree
$(1-n)$ part. Without loss of generality we will assume
$\textrm{deg}(h)=1-n$. Then using the commutator relations, we
obtain that, on such an $h$,
$$\begin{array}{rcl}
d^{n-1}\cdot h& = &
\sum_{j=0}^{n-1}(-1)^{n-1-j}\zeta^{(1-n)(n-1-j)}{n-1 \choose
j}_\zeta d^j\circ h \circ
d^{n-1-j}\\
&=&\sum_{j=0}(-1)^{n}\zeta^{-(j+1)(j+2)/2}d^j\circ h \circ d^{n-1-j}~,
\end{array}$$
where ${n \choose
k}_\zeta=\frac{(n-1)_\zeta!}{(k)_\zeta!(n-1-k)_{\zeta}!}$ and for
any $j\in \N$, $(j)_\zeta:=1+\zeta+\cdots +\zeta^{j-1}$ is the
un-symmetrized quantum integer $j$. In the last step, we used that ${n-1 \choose j}_\zeta$ equals
$$\frac{(1+\cdots +\zeta^{n-1})\cdots (1+\cdots
+\zeta^{n-j-1})}{(1+\cdots +\zeta^{j-1})\cdots
1}=(-\zeta^{-1})\cdots(-\zeta^{-j-1})=(-1)^{j+1}\zeta^{(j+1)(j+2)/2}.$$
Since each of the coefficient $(-1)^{n}\zeta^{-(j+1)(j+2)/2}$ is
non-zero, we may rescale $h$ componentwise by this scalar to obtain
the formula for a null-homotopic $f$ (c.f. \cite{Ka, Sar2}):
$$f=\sum_{j=0}^{n-1}d^j\circ h \circ d^{n-1-j}.$$

\item Consider the (graded) Hopf algebra
$H=\Bbbk[\partial]/(\partial^p)$, where $\Bbbk$ is of positive
characteristic $p$. Similar as above, if $h\in \Hom_A(M,N)$, we have
$$\partial^{p-1}(h)=\sum_{i=0}^{p-1}(-1)^{p-1-i}{p-1 \choose i}
\partial^i\circ h \circ \partial^{p-1-i}=\sum_{i=0}^{p-1}
\partial^{i}\circ h \circ \partial^{p-1-i}.$$
The last equality holds because $(-1)^i {p-1 \choose i}=1$ in
$\Bbbk$. We depict such a morphism in the following diagram, in
comparison with the previous cases.
$$\xymatrix@C=1.5em
{\cdots\ar[r]^-{\partial_M} & M^{i-p+1}\ar[r]^-{\partial_M}
 \ar[d]_{f} & M^{i-p+2}
\ar[d]_{f}\ar[r]^-{\partial_M} &\cdots\ar[r]^-{\partial_M} &
M^i\ar[dlll]|*+<1ex,1ex>{\scriptstyle h}
\ar[d]_{f}\ar[r]^-{\partial_M}
&M^{i+1}\ar[dlll]|*+<1ex,1ex>{\scriptstyle h}
\ar[d]_{f}\ar[r]^-{\partial_M}& \cdots \ar[r]^-{\partial_M} &
M^{i+p-1}\ar[dlll]|*+<1ex,1ex>{\scriptstyle h}
\ar[d]_{f}\ar[r]^-{\partial_M} & \cdots
\\
\cdots\ar[r]_-{\partial_N} & N^{i-p+1}\ar[r]_-{\partial_N} &
N^{i-p+2}\ar[r]_-{\partial_N} &\cdots\ar[r]_-{\partial_N}
&N^i\ar[r]_-{\partial_N} &N^{i+1}\ar[r]_-{\partial_N}& \cdots
\ar[r]_-{\partial_N} & N^{i+p-1}\ar[r]_-{\partial_N} & \cdots . }$$
\end{itemize}


\section{Cofibrant modules} Adapting the corresponding definition from
Keller \cite{Ke1, Ke2} on the DG case, we define the notion of
cofibrant hopfological modules and give a functorial cofibrant
resolution (i.e. quasi-isomorphism) $\mathbf{p}M \lra M$ for any
hopfological module $M$. This will be utilized later when discussing
compact hopfological modules, derived functors and derived
equivalences between hopfological module categories.

In this section $H$  will be assumed as before to be a finite
dimensional Hopf algebra over a base field $\Bbbk$, $A$ be an
$H$-module algebra, and we set $B= A\#H$.

\subsection{Cofibrant modules}
First we introduce the notion of ``cofibrant hopfological modules''
in analogy with the DG case.

\begin{defn}\label{def-cofibrant-module} A $B$-module $P$ is called
\emph{cofibrant} if for any surjective quasi-isomorphism
$M\twoheadrightarrow N$ of $B-$modules, the induced map of
$\Bbbk$-vector spaces
$$\mc{Z}(\Hom_A(P, M))\lra \mc{Z}(\Hom_{A}(P,N))$$
is surjective. In the graded case, we require instead that the
graded $H-$module map
$$\mc{Z}^*(\HOM_{A}(P,M))\lra \mc{Z}^*(\HOM_{A}(P,N))$$
be surjective in the category of graded $\Bbbk$-vector spaces.
Notice that this is equivalent to requiring the same condition on
$\mc{Z}^0$, as $M\{r\}\lra N\{r\}$ is a surjective
quasi-isomorphism, for any $r\in \Z$, whenever $M\lra N$ is.
\end{defn}

Recall from Lemma \ref{lemma-trivial-submodules} that,
$\mc{Z}(\Hom_A(P,M))=\Hom_B(P,M)$ consists of ``chain maps'' between
the hopfological modules $P$ and $M$. Therefore the definition just
says that any $B$-module map from $P$ to $N$ factors through a
$B$-module map from $P$ to $M$. It is rather straightforward to see
that being a ``cofibrant module'' in the case of DG modules implies
the usual sense of being ``K-projective module'', as described, for
instance, in Bernstein and Lunts \cite{BL}. It says that for any
acyclic DG-module $M$, the complex $\HOM_{A}(P, M)$ is acyclic as a
$\Bbbk[d]/(d^2)$-module, i.e the homology of this complex is $0$.
Indeed, it can be verified by applying the defining property to the
surjective quasi-isomorphism
$$\mathrm{Cone}(\Id_{M})\lra M,$$
and observing that $\HOM_{A}(P,
\mathrm{Cone}(\Id_{M}))=\mathrm{Cone}(\Id_{\HOM_{A}(P, M)})$ is
contractible. The following lemma is motivated by this discussion.

\begin{lemma}\label{lemma-characterizing-cofibrant-objects} Let $P$
be a cofibrant hopfological module. Then, for any acyclic module $M
\in B\dmod$ (resp. $B\dgmod$), the $H-$module $\Hom_A(P,M)$ (resp.
$\HOM_{A}(P,M)$) has trivial stable invariants:
$$\mc{H}(\Hom_A(P,M))=0~(\text{resp.}~\mc{H}^*(\HOM_A(M,N))=0),$$
and thus in the homotopy category, we have
$$\Hom_{\mc{C}(A,H)}(P,M)=0.$$
\end{lemma}
\begin{proof}The proof follows from the discussion before the lemma by
replacing the surjection $\mathrm{Cone}(\Id_M)\longrightarrow M$
with the cone in the hopfological case $M \o H\xrightarrow{\Id_M \o
\epsilon} M$. More precisely, let $P$ be a cofibrant hopfological
module. Apply $\Hom_A(P,-)$ to the $B-$module map $ M\o
H\xrightarrow{\Id_M \o \epsilon} M$, we obtain the induced map:
$$\mc{Z}(\Hom_A(P,
M \o H))\twoheadrightarrow \mc{Z}(\Hom_A(P,M)),$$
which is a
surjection by the cofibrance assumption. Therefore, for any
$\phi \in \mc{Z}(\Hom_A(P,M))$, we can find $\Phi \in
\mc{Z}(\Hom_A(P,M\o H))$ which when composed with $\Id \o \epsilon$
gives us $\phi$. Since $\Hom_A(P,M\o H)= \Hom_A(P,M)\o H$ is
contractible, $\Phi=\Lambda \cdot \Psi$ for some $\Psi \in
\Hom_A(P,M\o H)$ (Lemma \ref{lemma-ideal-null-homotopy}). Then for
any $x\in P$, we have
$$\begin{array}{rcl}
(\Lambda\cdot((\Id \o \epsilon)\circ\Psi))(x) & =
&\Lambda_{(2)}\cdot((\Id \o \epsilon) \circ
\Psi(S^{-1}(\Lambda_{(1)}\cdot x))\\
& = & (\Id \o
\epsilon)(\Lambda_{(2)}\cdot\Psi(S^{-1}(\Lambda_{(1)}\cdot
x)) = (\Id \o \epsilon)((\Lambda\cdot \Psi)(x))\\
& = & (\Id \o \epsilon)(\psi(x)) = \phi(x),
\end{array}
$$
where the second equality holds since $\Id \o \epsilon$ is
$H$-linear. Therefore by Corollary
\ref{cor-morphism-space-in-stable-H-mod}, $\phi=0$ when passing to
the stable category. The last claim follows from Proposition
\ref{prop-comparison-hom-spaces}.
\end{proof}

Notice that, when $H$ is a finite dimensional local Hopf algebra,
$\mc{H}(\Hom_A(P,M))=0$ actually implies that the
total cohomology $\Hom_A(P,M)$ is $0$ in the stable category
$H\udmod$. This follows from the observation that any indecomposable
module in the case contains a trivial submodule. Therefore for such
$H$'s, we know that the $H$-module $\Hom_A(P,M)$ is projective and
injective as an $H$-module (we will just call such $H$-modules
acyclic when no confusion could arise). In fact, this will turn out
to be true for any $H$ and any cofibrant module $P$. We will show
this after introducing some necessary tools.

Our main goal in this section is to construct, for each $A-$module
$M$, a functorial cofibrant replacement. We make the following
definition.

\begin{defn}\label{def-property-P} We say that a $B$-module
satisfies \emph{property $(P)$} if it is isomorphic to a module $P$ in
the category $\mc{C}(A,H)$ for which the following three conditions hold
(c.f.~\cite[section 3]{Ke1}):
\begin{enumerate}
\item[(P1)] There is a filtration
$$0\subset F_{0}\subset F_{1}\subset \cdots F_{r} \subset F_{r+1} \subset
\cdots \subset P, $$ and the filtration is exhaustive in the sense
that $$P=\cup_{r\in \N}F_{r};$$
\item[(P2)] The inclusion $F_r \subset F_{r+1}$ splits as left $A$-modules (resp. graded left
$A$-modules when they are graded) for all $r \in \N$;\\
\item[(P3)] $F_0$, as well as the quotients $F_{r+1}/F_r$ for all $r\in \N$,
is isomorphic to direct sums of $B$-modules of the form $A\o V$,
where $V$ is an indecomposable $H$-module (resp. $A\o V \in B\dgmod$
and $V\in H\dgmod$ in the graded case).
\end{enumerate}
Equivalently, in the last condition (P3), we may drop the direct sum
requirement for indecomposable $V$'s but instead allow $V$ to be any
$H$-module.
\end{defn}

We need to clarify the relation between modules with property (P) and
cofibrant modules. First of all we will show that modules with
property (P) are cofibrant.

\begin{lemma}\label{lemma-property-p-cofibrant}Let $P\in B\dmod$
(resp. $B\dgmod$) be a module satisfying property (P), and $K$ be an
acyclic $B$-module. Then the $H$-module $\Hom_A(P,K)$ is projective
and injective as an $H$-module.
\end{lemma}
\begin{proof}The proof is divided into three steps. First off, we check
that free modules of the form $A\o V$ have the claimed property of
the lemma. As $H$-modules, we have a canonical isomorphism:
 $$\Hom_A(A\o V,K)\cong \Hom_{\Bbbk}(V,K).$$
Thus the result for $A\o V$ follows from Lemma
\ref{lemma-infinite-projective-H-modules}.

Secondly, we use induction to prove that $\Hom_A(F_r,K)$ is
projective and injective (acyclic for short) for any $r \geq 0$. In
fact, assuming so for $F_r$, applying $\Hom_A(-,K)$ to the short
exact sequence of free $A-$modules:
$$0 \lra F_r \lra F_{r+1}\lra \bigoplus_{j\in J} A\o V_j \lra 0,$$
we obtain a short exact sequence of $H$-modules:
$$0 \lra \prod_{j \in J} \Hom_A(A\o V_j,K) \lra \Hom_A(F_{r+1}, K)
\lra \Hom_A(F_r, K)\lra 0.$$ By inductive hypothesis and the
previous step, $\Hom_A(F_r,K)$ and $\prod_{j\in J}\Hom_A(A\o V_j,K)$
are acyclic. Thus $\Hom_A(F_{r+1},K)$ is acyclic, since in $H\dmod$
it is isomorphic to the direct sum of these acyclic modules.

Finally, by definition, we have the following short exact sequence
of free $A-$modules:
$$0\lra \bigoplus_{r\in \N}F_r \stackrel{\Psi}{\lra}
\bigoplus_{s \in \N}F_s \lra P \lra 0,$$ where the map $\Psi$ is
given by the block upper triangular matrix:
$$\Psi=\left(
\begin{array}{ccccc}
\Id_{F_0} & -\iota_{01} & 0 & 0 & \dots\\
0 & \Id_{F_1} & -\iota_{12} & 0 & \dots\\
0 & 0 & \Id_{F_2} & -\iota_{23} & \dots\\
0 & 0 & 0 & \Id_{F_3} & \dots\\
\vdots & \vdots & \vdots & \vdots &\ddots
\end{array}
\right),
$$
where $\iota_{r,r+1}$ is the inclusion of $F_r$ into $F_{r+1}$.
Applying $\Hom_A(-,K)$ to the short exact sequence of free
$A-$modules, we obtain a short exact sequence of $H-$modules:
$$0 \lra \Hom_A(P,K) \lra \prod \Hom_A(F_s, K) \lra \prod \Hom_A(F_r,
 K)\lra 0.$$
By the second step, the two terms on the right are acyclic. Hence
$\Hom_A(P,K)$ is acyclic and the lemma follows.
\end{proof}

\begin{cor}\label{cor-property-P-implies-cofibrant} If $P$ is a
$B$-module with property (P), then it is cofibrant.
\end{cor}
\begin{proof}
Let $M \lra N$ be a surjective quasi-isomorphism in $B\dmod$. We
have a short exact sequence of $B$-modules:
$$0\lra K\lra M\lra N\lra 0,$$
where $K$ is acyclic by our assumption. Applying $\Hom_A(P,-)$ to
this short exact sequence, we obtain a short exact sequence of
$H$-modules:
$$0\lra \Hom_A(P,K)\lra\Hom_A(P,M)\lra \Hom_A(P,N)\lra 0,$$
since $P$ is projective as an $A$-module.  The above Lemma
\ref{lemma-property-p-cofibrant} says that $\Hom_A(P,K)$ considered
as an $H-$module is projective and injective, and thus the sequence
splits and we have a direct sum decomposition:
$$\Hom_A(P,M)\cong \Hom_A(P,K)\oplus \Hom_A(P,N).$$
Taking $H$-invariants on both sides (Proposition
\ref{prop-comparison-hom-spaces}) gives us:
$$\mc{Z}(\Hom_A(A\o V, M))\cong \mc{Z}(\Hom_A(A\o V, K))\oplus \mc{Z}(\Hom_A(A
\o V, N)),$$ whence the surjectivity  $\mc{Z}(\Hom_A(P,M))
\twoheadrightarrow \mc{Z}(\Hom_A(P,N))$ follows.
\end{proof}

\subsection{The bar resolution} Now we formulate the main result of
this section and its immediate consequences.

\begin{thm}\label{thm-bar-resolution}Let $H$ be a finite dimensional
(graded) Hopf algebra, $A$ be a left $H-$module algebra, and set
$B=A\#H$. For each module $M\in B\dmod$ (resp. $B\dgmod$), there is
a short exact sequence in $B\dmod$ (resp. $B\dgmod$) which is split
exact as a sequence of $A$-modules:
$$\xymatrix{0\ar[r]&M\ar[r] & \mathbf{a}M\ar[r] & \mathbf{\tilde{p}}M\ar[r]&0},$$
where $\mathbf{\tilde{p}}M$ satisfies property (P) and $\mathbf{a}M$
is an acyclic $B$-module. Moreover the construction of the short
exact sequence is functorial in $M$.
\end{thm}

We will refer to the construction of the theorem, as well as the
cofibrant replacement in the next corollary, as the ``bar
resolution'' of any hopfological module $M$, which is the functorial
cofibrant replacement we claimed at the beginning of this section.

\begin{cor}\label{cor-cofibrant-replacement}Under the same conditions
as in Theorem \ref{thm-bar-resolution}, let $M$ be any hopfological
module $M\in B\dmod$.
\begin{enumerate}
\item[(i).] There is an associated distinguished triangle, functorial in
$M$ inside $\mc{C}(A,H)$:
$$M\lra \mathbf{a}M \lra \tilde{\mathbf{p}}M \lra TM.$$
\item[(ii).] In the derived category $\mc{D}(A,H)$, there is a functorial
isomorphism
$$\mathbf{p}M\stackrel{\cong}{\lra} M,$$
where $\mathbf{p}M:=T^{-1}(\tilde{\mathbf{p}}M)$ is a module with
property (P).
\item[(iii).] The isomorphism in (ii) arises as the image of a surjective
quasi-isomorphism $\mathbf{p}M\twoheadrightarrow M$ in $B\dmod$.
\end{enumerate}
\end{cor}

\begin{proof}
By applying Lemma \ref{lemma-split-ses-lead-to-dt} to the short
exact sequence of the theorem, we obtain a distinguished triangle in
$\mc{C}(A,H)$
$$\xymatrix{M\ar[r] & \mathbf{a}M\ar[r] & \mathbf{\tilde{p}}M\ar[r]& T(M)},$$
which is functorial in $M$ by Theorem \ref{thm-bar-resolution}.
Since $\mathbf{a}M$ is acyclic, it is isomorphic to $0$ in the
derived category. By passing to the derived category $\mc{D}(A,H)$
we obtain a functorial isomorphism
$$\mathbf{\tilde{p}}M\stackrel{\cong}{\lra} T(M).$$
Then apply $T^{-1}$ to this isomorphism $\mathbf{\tilde{p}}M \lra
T(M)$, and we define
$$\mathbf{p}M := T^{-1}(\mathbf{\tilde{p}}M)=\tilde{\mathbf{p}}M\o
\textrm{ker}(\epsilon),$$
which satisfies property (P) since
$\tilde{\mathbf{p}}M$ does. This proves (i) and (ii). We will
postpone the proof of part (iii) until the end of this section,
where the explicit surjective quasi-isomorphism is constructed.
\end{proof}

We reap some other direct consequences of the bar construction, the
first of which is the promised relationship between cofibrant
modules and modules with property (P).

\begin{cor}\label{cor-cofibrant-direct-summand-of-property-p} Let
$M$ be a cofibrant hopfological module. Then $M$ is a direct summand
of a $B$-module with property (P). Conversely, any $B\dmod$ direct
summand of a module with property (P) is cofibrant. In other words,
the class of cofibrant modules is the idempotent completion of the
class of modules with property (P) in the abelian category $B\dmod$.
\end{cor}

\begin{proof}By (iii) of Corollary \ref{cor-cofibrant-replacement}, we
have a surjective quasi-isomorphism $\mathbf{p}M\twoheadrightarrow
M$. Applying the $\Hom_B(M,-)$ to this surjection and using the
cofibrance condition, we see immediately that $M$ is a direct
summand of $\mathbf{p}M$, which is a module with property (P) by the
same corollary.

Conversely, if $M$ is a direct summand of a property (P) module $N$,
say $N\cong M\oplus M^{\prime}$, then $\Hom_A(N,-)\cong
\Hom_A(M,-)\oplus \Hom_A(M^{\prime},-)$ as functors from $B\dmod$ to
$H\dmod$. Since a direct summand of a projective and injective
$H$-module is still projective and injective, $\Hom_A(M,K)$ is
acyclic for any acyclic module $K$, using Lemma
\ref{lemma-property-p-cofibrant}. The same proof as in Corollary
\ref{cor-property-P-implies-cofibrant} shows that $M$ is cofibrant.
The rest of the corollary is clear.
\end{proof}

The next result gives the promised characterization of cofibrant
modules as an analogue of ``K-projective modules'' due to Bernstein
and Lunts \cite{BL}.
\begin{cor}\label{cor-cofibrant-implies-K-projective} A hopfological
module $M$ is cofibrant if and only if $M$ is projective as an
$A$-module, and for any acyclic module $K$, the $H$-module
$\Hom_A(M,K)$ is projective and injective.
\end{cor}

\begin{proof} The ``if'' direction follows from the the same
argument we used in Corollary
\ref{cor-property-P-implies-cofibrant}. The ``only if'' part follows
from the above Corollary
\ref{cor-cofibrant-direct-summand-of-property-p}, the corresponding
result for property (P) modules \ref{lemma-property-p-cofibrant}, and
the fact that an injective submodule of any $H$-module is an $H$-direct
summand.
\end{proof}

The last immediate consequence of the theorem we record here is the
equivalence between $\mc{D}(A,H)$ and the homotopy category of
property (P) (resp. cofibrant) objects.

\begin{cor}\label{cor-equivalence-cofibrant-subcat}Let
$\mathcal{P}(A,H)$ (resp. $\mc{CF}(A,H)$) be the full triangulated
subcategory of $\mc{C}(A, H)$ whose objects consist of hopfological
modules satisfying property (P) (resp. cofibrant modules). Then:
 \begin{enumerate}\item The morphism space
between any two objects $P_1$, $P_2$ in $\mc{P}(A,H)$ (resp.
$\mc{CF}(A,H)$) coincides with the morphism space of these objects
in the derived category:
$$\Hom_{\mc{P}(A,H)}(P_1, P_2)\cong \Hom_{\mc{D}(A,H)}(P_1,P_2).$$ In
fact, for any $P$ with property (P) (resp. cofibrant), we
have:$$\Hom_{\mc{C}(A,H)}(P,-)\cong \Hom_{\mc{D}(A,H)}(P,-).$$
\item The composition of functors
$$\mc{P}(A,H) \subset \mc{C}(A,H)
\stackrel{Q}{\lra} \mc{D}(A,H)~~\left(\textrm{resp.}~\mc{CF}(A,H) \subset
\mc{C}(A,H) \stackrel{Q}{\lra} \mc{D}(A,H)\right),$$
 where $Q$ is the
localization functor, is an equivalence of triangulated categories.
\item The bar resolution is a functor $\mathbf{p}:\mc{D}(A,H)\lra \mc{P}(A,H)$
which is the left adjoint to the composition functor $\mc{P}(A,H)
\subset \mc{C}(A,H) \stackrel{Q}{\lra} \mc{D}(A,H)$.
\end{enumerate}
\end{cor}
\begin{proof}The first claim follows from standard homological algebra
arguments, using Lemma \ref{lemma-characterizing-cofibrant-objects}.
It goes as follows. By definition of morphisms in $\mc{D}(A,H)$, it
suffices to show that, for any quasi-isomorphism $s:X\lra P$ in
$\mc{C}(A,H)$, where $P$ is either with property (P) or cofibrant,
there exists a morphism
$$t: P\lra X$$ in $\mc{C}(A, H)$ such that $ts=\Id_{P}$. The
cone of $s$ is acyclic, giving a distinguished triangle in
$\mc{C}(A,H)$: $X\stackrel{s}{\lra} P\lra \textrm{Cone}(s) \lra TX$.
Applying $\Hom_{\mc{C}(A,H)}(P,-)$ produces the desired isomorphism:
$$\Hom_{\mc{C}(A,H)}(P, X)\cong \Hom_{\mc{C}(A,H)}(P,P).$$
The result follows. The second and third claims are easy, and we
leave them as exercises to the reader.
\end{proof}

\begin{rmk} To summarize the notions we introduced in this section, we
have an inclusion of diagrams inside the abelian category $B\dmod$:
$$
(\textrm{Modules with property (P)}) \subset (\textrm{Cofibrant
modules}) \subset (\textrm{Hopfological modules}).
$$
The previous corollary can be summarized as saying that these
inclusions in turn give equivalences of the homotopy categories
$\mc{P}(A,H)$ and $\mc{CF}(A,H)$ with the derived category
$\mc{D}(A,H)$.
\end{rmk}

\subsection{Proof of Theorem \ref{thm-bar-resolution}}
\paragraph{The simplicial bar resolution of an
algebra.}Recall that for an algebra $A$ over $\Bbbk$ (the
construction works more generally over $\Z$), the simplicial bar
resolution of $A$ is a projective resolution of $A$ as a module over
the envelope algebra $A\otimes A^{op}$, i.e. as an $(A,A)-$bimodule.
We review its construction briefly here. Standard details about bar
resolutions can be found in Loday's monograph~\cite[Chapter I]{Lo}.

Let $(C_\bullet,d_i,s_i)$ be a simplicial module over the base field
$\Bbbk$, where $d_i$ is the face map, and $s_i$ is the degeneration
map, satisfying the commutator relation:
$$d_i d_j=d_{j-1}d_i~~\textrm{if}~i<j, \quad
 d_i s_j=
\begin{cases}
s_{j-1} d_i & \textrm{if}~i<j,\\
\textrm{id} & \textrm{if}~i=j,j+1,\\
s_j d_{i-1} & \textrm{if}~i>j+1.
\end{cases}$$
One can naturally associate with such a simplicial module a complex
by defining the differential $\delta : C_{n} \lra C_{n+1}$ as the
alternating sum of the face maps $\delta=\sum_{i=0}^{n-1}(-1)^id_i$.
One then checks readily using the commutator relations in the
definition that $(C_\bullet,\delta)$ becomes a complex. Now we apply
this construction to the Hochschild complex:

\begin{defn}
The Hochschild simplicial module of a $\Bbbk$-algebra $A$ is the
simplicial module $(C(A),d_i,s_i)$, where for each $n \geq 0$,
$C_{-n}=A^{\o(n+1)}$, and $C_{n+1}=0$. The face and degeneration
maps are defined by:
$$d_i(a_0\o a_1 \o \cdots \o a_n)=
\begin{cases}
a_0 \o\cdots\o a_ia_{i+1}\o \cdots \o a_n & \textrm{if}~0\leq i\leq n-1,\\
a_na_0\o a_1\cdots \o a_{n-1} & \textrm{if}~i=n,
\end{cases}$$
and
$$s_i(a_0\o a_1 \o \cdots \o a_n)=a_0 \o \cdots \o a_i\o 1\o a_{i+1}\cdots\o a_n.$$
\end{defn}

We have the well-known:

\begin{lemma} \label{lemma-simplicial-bar-contractiblity} The associated simplicial
 bar complex $(C_{-n}=A^{\otimes (n+1)},
\delta_n)$ is a contractible complex, giving a resolution of $A$ as
an $(A,A)$-bimodule by free bimodules.
\end{lemma}
\Pf{A homotopy is given by the ``extra-degeneracy" $$s:A^{\o n}\lra
A^{\o(n+1)}\ , \ \ \  a_0\o\dots\o a_{n-1} \mapsto 1\o a_0\o \dots \o
a_{n-1},$$ for any $n\in \N$.}

\paragraph{Proof of Theorem \ref{thm-bar-resolution}: construction.}Now we begin
with the construction of the bar resolution. The first observation
to make is that, in the recap above, when $A$ is a left
$H-$module-algebra, all the face and degeneration maps are
$H-$module maps. For instance, the map $\delta_0:A\o A\lra A$,
$\delta_0(a_0\otimes a_1)=a_0a_1$ is the multiplication map, which
is an $H-$module map by definition. Now we apply the cone
construction (Definition \ref{def-cone-construction}) to this map
and obtain:
$$\textrm{Cone}(\delta_0)\cong A \o A\o (H/\Bbbk\Lambda)\oplus A,$$
the isomorphism viewed as a $(A,A)-$bimodule map ($A$ acts trivially
on the $(H/\Bbbk\Lambda)$ factor). However, this isomorphism is not
an $H-$module isomorphism. The $H-$ module structure on the cone is
defined in a slightly abstract way using the push-out property,
which is not preserved under this identification. We can give a more
explicit description as follows, but it's not necessary for the
construction below.

We complete $\Lambda$ to a basis $\{h_i|i=1,\dots, r, h_r=\Lambda,
r=\textrm{dim}(H)\}$ of $H$. We describe the left action of $H$ on
itself explicitly in this basis by setting:
$$h\cdot h_i=\sum_jc(h,i)_jh_j.$$
Now take a basis $\{a_k|k\in I\}$ of $A$, where $I$ is some index
set so that the cone has as a basis of elements:
$$\{a_k \o a_l \o h_i|i=1,\dots, r-1,k,l\in I\}\cup\{a_k|k\in I\}.$$
The $H-$action is given as follows:
$$h\cdot a_k=ha_k;$$
$$\begin{array}{lll}
h\cdot(a_k \o a_l \o h_i)&=&\sum_{(h)}\left( h_{(1)} a_k \o a_l \o
h_i+  a_k \o h_{(2)} a_l\o h_i+ a_k
\o  a_l \o h_{(3)}h_i\right)\\
& = & \sum_{(h)} \l \sum_{j=1}^{r-1} \left( c(h_{(3)},i)_j
a_k \o a_l \o h_j\right)+c(h_{(1)},i)_ra_k \o a_l \o h_r\\
& & + h_{(1)} a_k \o a_l \o h_i+ a_k \o
h_{(2)} a_l \o h_i\r\\
& = &\sum_{(h)} \l \sum_{j=1}^{r-1}\left( c(h_{(3)},i)_j
a_k \o a_l \o h_j\right)\\
& & +c(h_{(3)},i)_ra_k \o a_l \o \Lambda+c(h_{(3)},i)_r a_ka_l-c(h_{(3)},i)_r a_ka_l\\
& & + h_{(1)} a_k \o a_l \o h_i+ a_k \o
h_{(2)} a_l \o h_i \r\\
& \equiv & \sum_{(h)} \l \sum_{j=1}^{r-1}\left( c(h_{(3)},i)_j a_k
\o a_l \o h_j \right) -c(h_{(3)},i)_r
a_ka_l\\
& & +h_{(1)} a_k \o a_l \o h_i + a_k \otimes h_{(2)} a_l \o h_i\r,
\end{array}
$$
where in the last equality, we used that $ a_k\o a_l \o
\Lambda+a_ka_l\equiv 0$ in the cone. Notice that when $H$ is the
Hopf super algebra $\Bbbk[d]/(d^2)$, and if we take the basis of $H$
to be $\{1,d\}$, it is readily seen that the action of $d$ recovers
the usual ``connection map" from the cone to $T(A)$ in the standard
distinguished triangle associated with $\delta_0:A\otimes A\lra A$.

Next, we will lift the map $\delta_1:A\o A\o A\lra A\o A$ to a map
$\tilde{\delta_1}: A\o A\o A \o (H/\Bbbk \Lambda)\lra
\textrm{Cone}(\delta_0)$, as follows. First off we define a map:
$$\begin{array}{ccc}
 A\o A \o A \o H& \lra & A\o A \o H \oplus A\\
 a \o a^\prime \o a^{\prime\prime}\o h & \mapsto &
( \delta_1 (a \o a^{\prime}\o a^{\prime \prime}) \o h,0 )
\end{array}
$$
The submodule $A \o A \o A \o \Bbbk\Lambda$ of $A\o A\o A \o H$ is
mapped into the module $$\textrm{Im}({A \otimes A
\xrightarrow{ \lambda_{A\o A} \oplus \delta_0}  A\otimes A \o
H\oplus A }),$$ since
$\left(\delta_1(a \o a^\prime \o a^{\prime\prime})\o \Lambda, 0\right)=
(\left(\delta_1(a \o a^\prime \o a^{\prime\prime})\o \Lambda,
\delta_0\delta_1 (a \o a^\prime \o a^{\prime\prime})\right)),
$ where we used that $\delta_0\delta_1=0$. Therefore, this map descends
to the quotient and gives rise to $\tilde{\delta_1}$:
$$\tilde{\delta_1}:A\o A\o A \o (H/\Bbbk\Lambda)\lra \textrm{Cone}(\delta_0)$$
Also observe that $\tilde{\delta_1}$ kills elements in the submodule
$\textrm{Im}(\delta_2) \o (H/\Bbbk\Lambda)$.

Then we can construct the cone of $\tilde{\delta_1}$. Recall from
the definition of the cone construction that in
$\textrm{Cone}(\delta_0)$, $A$ is naturally an $H-$submodule, while
the quotient $\textrm{Cone}(\delta_0)/A$ is isomorphic to the
$H-$module $A\o A \o (H/\Bbbk\Lambda)$. Thus the cone of
$\tilde{\delta_1}$ has a filtration by $(A,A)-$bimodules:
$$0\subset A \subset \textrm{Cone}(\delta_0)\subset \textrm{Cone}(\tilde{\delta_1}),$$
whose subquotients are respectively $A$, $A^{\o 2}\o
(H/\Bbbk\Lambda)$, and $A^{\o 3} \o (H/\Bbbk\Lambda)^{\o 2}$. These
observations will allow us to construct the bar resolution
inductively.

Now assume we have inductively constructed:
\begin{enumerate}
\item $C_n=\textrm{Cone}(\tilde{\delta_n}:A^{\o(n+2)}\o (H/\Bbbk\Lambda)^{\o n}
\lra C_{n-1})\in B\dmod$;
\item For any $x\in (H/\Bbbk\Lambda)^{n}$, $a\in A^{\o(n+3)}$ we have
 $\tilde{\delta_n}(\delta_{n+1}(a)\o x)=0.$
\end{enumerate}
This assumption implies that $C_{n-1}$ is a submodule of $C_n$. Then
using another induction argument, we see that $C_n$ has an
exhaustive filtration
$$F^{\bullet}: 0=F^{-1}\subset F^0\subset \dots \subset F^{p-1}\subset F^{p}\subset \dots
\subset F^{n+1}=C_n,$$
whose subquotients $F^{n}/F^{n-1}$ are
isomorphic to $A^{\o (n+1)} \o (H/\Bbbk\Lambda)^{\o n}$. In
particular, this says that $C_n$ satisfies ``property (P)", and
therefore is a cofibrant $B$-module as defined earlier.

Now we construct the $B$-module map $\widetilde{\delta_{n+1}}$.
Tensoring with the identity map of $(H/\Bbbk\Lambda)^{\o n}$, we
have a map $A^{\o(n+3)}\o (H/\Bbbk\Lambda)^{\o n}\lra A^{\o(n+2)} \o
(H/\Bbbk\Lambda)^{\o n}$, which in turn gives rise to a map:
$$\begin{array}{ccc}
A^{\o(n+3)}\o(H/\Bbbk\Lambda)^{\o n}\o H & \lra &
 A^{\o(n+2)}\o (H/\Bbbk\Lambda)^{\o n}\o H \oplus C_{n-1}\\
 a \o x \o h & \mapsto & (\delta_{n+1}(a)\o x \o h,0)
\end{array},$$
where $h\in H$, $x\in (H/\Bbbk\Lambda)^{\o n}$, and $a \in
A^{\o(n+3)}$. This map descends to the desired
$$\widetilde{\delta_{n+1}}:A^{\o(n+3)} \o (H/\Bbbk\Lambda)^{\o(n+1)} \lra C_n$$
since elements of the form $a \o x \o \Lambda$ are sent to
$$a \o x \o \Lambda \mapsto ( \delta_{n+1}(a) \o x \o\Lambda,0)=
(\delta_{n+1}(a)\o x \o \Lambda,\tilde{\delta_{n}}( \delta_{n+1}(a)
\o x )),$$ and by our inductive hypothesis
$\tilde{\delta_{n}}(\delta_{n+1}(a)\o x)=0$. Finally, we verify the
inductive hypothesis 2 for $\widetilde{\delta_{n+1}}$, which
requires that it kills elements in the image of $\delta_{n+2}$:
$$\widetilde{\delta_{n+1}}( \delta_{n+2}(a)\o x\o \bar{h})=\delta_{n+1}
\delta_{n+2}(a)\o x \o\bar{h} =0,$$ where $\bar{h}\in
H/\Bbbk\Lambda$, $x \in (H/\Bbbk\Lambda)^{\o n}$, $a \in
A^{\o(n+4)}$, and we have used that $\delta_{n+1}\delta_{n+2}=0$.

In conclusion, we have constructed inductively a chain of
$(A,A)-$bimodules:
$$A=C_{-1}\subset C_{0}\subset C_1 \subset \dots \subset C_{n-1}
\subset C_{n} \subset \dots$$ whose subquotients are
$$C_{n}/C_{n-1}\cong A^{\o(n+2)} \o (H/\Bbbk\Lambda)^{\o (n+1)}.$$
We define
$$\mathbf{a}A := \bigcup_{n=-1}^{\infty} C_n,$$
which fits into a short exact sequence:
$$\xymatrix{0\ar[r]&A\ar[r] & \mathbf{a}A\ar[r] & \mathbf{\tilde{p}}A\ar[r]&0}.$$
We may regard any left $B$-module $M$ as an $A$-module by
restriction. Tensoring the above sequence by $M$ gives rise to the
short exact sequence
$$\xymatrix{0\ar[r]&M\ar[r] & \mathbf{a}M\ar[r] & \mathbf{\tilde{p}}M\ar[r]&0}$$
claimed in the theorem. Our next goal would then be to show that
$\mathbf{a}M$ in the above short exact sequence is contractible as
an $H$-module, for any hopfological module $M$.

\paragraph{Proof of Theorem \ref{thm-bar-resolution}:
contractibility.} Now we show that $\mathbf{a}M$ is acyclic, for any
$A-$module $M$. To do this we may safely forget about the $B-$module
structures involved and regard the modules as $H$-modules.  We will
show this for $\mathbf{a}A$; and the general case follows by the
same argument.

Observe that in the Lemma \ref{lemma-simplicial-bar-contractiblity},
the homotopy $s:A^{\o n}\lra A^{\o (n+1)}$ is an $H$-module map
since $A$ is an $H$-module algebra. Thus the homotopy allows us to
split the terms in the original bar complex of $A$ into $H$-modules
summands
$$A^{\o n}\cong A^{(n)}\oplus A^{(n-1)}$$
so that the boundary map $\delta:A^{\o n}\lra A^{\o(n-1)}$ (an
$H$-module map again) kills the $A^{(n)}$ factor and identifies the
$A^{(n-1)}$ factor with that in $A^{\o(n-1)}$. Now if we go back to
the definition of the cone $C_0$ as in the previous part, we see
that it was constructed as a pushout, and therefore, as $H$-modules,
we can identify it with:
$$\begin{array}{ccl}
C_0 & \cong &( A^{\o 2} \o H\oplus A)/(\{a \o a^\prime \o \Lambda,
aa^\prime|a,a^\prime \in A\})\\
&\cong &((A^{(2)}\oplus A)\o H \oplus A)/(\{(( a^{(2)}
,~ a) \o \Lambda, a)|a^{(2)}\in A^{(2)},a\in A\})\\
& \cong &( A^{(2)} \o H\oplus A\o H\oplus A)/(\{(
a^{(2)}\o \Lambda)|a^{(2)}\in A^{(2)}\} \oplus \{(a\o \Lambda , a)|a\in A\})\\
& \cong & A\o (H/\Bbbk\Lambda) \oplus A \o H.
\end{array}$$

Then at the second step, we constructed $C_1$ as the cone of
$\tilde\delta_1$, which was defined by first mapping $A^{\o 3} \o H$
onto $A^{\o 2} \o H\oplus A$ via $(\delta_1\o \Id_H, 0)$ and then
taking a quotient. With respect to the decompositions $A^{\o 3}\cong
A^{(3)}\oplus A^{(2)}$ and $A^{\o 2}\cong A^{(2)}\oplus A$, the map
is identified with the map $A^{(3)}\o H \oplus A^{(2)}\o H \lra
A^{(2)}\o H\oplus A \o H\oplus A$ which is the identity on the
$A^{(2)}\o H$ factor and zero on $A^{(3)}\o H$. Therefore,
$\tilde\delta_1$ written out in this componentwise form becomes:
$$\begin{array}{rcl}
\tilde\delta_1: A^{(3)}\o (H/\Bbbk\Lambda)\oplus A^{(2)}\o
(H/\Bbbk\Lambda) & \lra & A^{(2)} \o (H/\Bbbk\Lambda)
\oplus A \o H,\\
(a^{(3)}\o \overline{h^\prime},~a^{(2)}\o \overline{h} ) & \mapsto &
(a^{(2)}\o \overline{h},~0),
\end{array}$$
for any $a^{(3)}\in A^{(3)}$, $a^{(2)}\in A^{(2)}$, and
$\overline{h},~ \overline{h^\prime} \in H/\Bbbk\Lambda$. The cone of
$\tilde\delta_1$ is then identified as an $H$-module with
$$C_1\cong A^{(3)}\o(H/\Bbbk\Lambda)^{\o2}\oplus
A^{(2)}\o (H/\Bbbk\Lambda)\o H\oplus  A \o H.$$

Inductively, assume that as $H$-modules,
$$C_{n-1}\cong A^{(n+1)}\o (H/\Bbbk\Lambda)^{\o n}\oplus \bigoplus_{i=1}^{n}
(A^{(i)}\o (H/\Bbbk\Lambda)^{\o (i-1)}\o H)$$ and the $H$-module map
$\widetilde{\delta_n}$ in the construction procedure is given
componentwise by
$$\xymatrix{ A^{\o (n+2)} \o
(H/\Bbbk\Lambda)^{n} \ar[rrr]^{\widetilde{\delta_n}} &&&
C_{n-1} \\
\Updownarrow &&& \Updownarrow \\
 A^{(n+2)}\o(H/\Bbbk\Lambda)^{\o n}\ar[r]& 0 && A^{(n+1)} \o
 (H/\Bbbk\Lambda)^{\o n}\\
 \bigoplus &&& \bigoplus\\
 A^{(n+1)}\o (H/\Bbbk\Lambda)^{\o n} \ar[rrruu]^{=} &&&
 \oplus_{i=1}^{n}A^{(i)}\o (H/\Bbbk\Lambda)^{\o(i-1)} \o H.
}$$
Then as $H$-modules, the cone of $\widetilde{\delta_n}$ is
isomorphic to:
$$\begin{array}{rcl}
C_n=\textrm{Cone}(\widetilde\delta_n) & \cong &
A^{(n+2)}\o(H/\Bbbk\Lambda)^{\o (n+1)} \oplus
\textrm{Cone}(\Id_{A^{(n+1)}\o (H/\Bbbk\Lambda)^{\o n}})\\
& & \oplus  \bigoplus_{i=1}^{n}(A^{(i)}\o
(H/\Bbbk\Lambda)^{\o(i-1)} \o H)\\
& \cong & A^{(n+2)}\o(H/\Bbbk\Lambda)^{\o (n+1)} \oplus
(A^{(n+1)}\o (H/\Bbbk\Lambda)^{\o n}\o H)\\
& & \oplus  \bigoplus_{i=1}^{n}(A^{(i)}\o
(H/\Bbbk\Lambda)^{\o(i-1)} \o H)\\
 & \cong &A^{(n+2)}\o(H/\Bbbk\Lambda)^{\o (n+1)} \oplus
 \bigoplus_{i=1}^{n+1}(A^{(i)}\o (H/\Bbbk\Lambda)^{\o(i-1)}\o H).
\end{array}$$
Furthermore, $\widetilde{\delta_{n+1}}:A^{\o(n+3)}\o
(H/\Bbbk\Lambda)^{\o(n+1)}\lra C_n$, which is constructed as the
quotient of $(\delta_{n+1}\o \Id \o \Id, 0): A^{\o(n+3)}\o
(H/\Bbbk\Lambda)^{\o n}\o H \lra A^{\o(n+2)}\o (H/\Bbbk\Lambda)^{\o
n}\o H \oplus C_{n-1}$ by the submodule $A^{\o(n+3)}\o
(H/\Bbbk\Lambda)^{\o n}\o \Bbbk\Lambda$,
 decomposes as the
$H$-module map:
$$\xymatrix{
 A^{(n+3)}\o(H/\Bbbk\Lambda)^{\o (n+1)}\ar[r]& 0 && A^{(n+2)} \o
 (H/\Bbbk\Lambda)^{\o (n+1)}\\
 \bigoplus &&& \bigoplus\\
 A^{(n+2)}\o (H/\Bbbk\Lambda)^{\o (n+1)} \ar[rrruu]^{=} &&&
 A^{(n+1)} \o (H/\Bbbk\Lambda)^{\o n}\o H \oplus C_{n-1}.
}$$ This finishes the induction step, and establishes the $H$-module
isomorphism:
$$C_n=\textrm{Cone}(\widetilde{\delta_{n}})\cong A^{(n+2)}\o(H/\Bbbk\Lambda)
^{\o(n+1)}\oplus\bigoplus_{i=1}^{n+1} A^{(i)}\o (H/\Bbbk\Lambda)^{\o
(i-1)}\o H.$$ Taking the union of all $n$ gives us
$$
\mathbf{a}(A) \cong \bigoplus_{i=1}^{\infty}
A^{(i)}\o(H/\Bbbk\Lambda)^{\o (i-1)}\o H \cong
\left(\bigoplus_{i=1}^{\infty} A^{(i)} \o (H/\Bbbk\Lambda)^{\o
(i-1)}\right)\o H,$$ which is of the form $N\o H$ for some
$H-$module $N$, and the acyclicity follows. This finishes the proof
of Theorem \ref{thm-bar-resolution}. $\hfill\Box$

\begin{proof}[Proof of part (iii) of Corollary
\ref{cor-cofibrant-replacement}] Now we finish the proof of the
corollary. Notice that as $(A,A)$ bimodules,
$$\begin{array}{rcl}
\mathbf{p}(A)& = &\bigcup_{n=0}^{\infty}C_{n}\o
\textrm{ker}(\epsilon)\\
& \cong & A\o A \o H/(\Bbbk \Lambda) \o \textrm{ker}(\epsilon)
\oplus \cdots \oplus A^{\o (n+2)}\o (H/\Bbbk\Lambda)^{\o(n+1)}
\o \textrm{ker}(\epsilon) \oplus \cdots \\
& \cong & A\o A \o (\Bbbk \oplus Q) \oplus \cdots \oplus A^{\o
(n+2)}\o (H/\Bbbk\Lambda)^{\o(n)} \o (\Bbbk \oplus Q) \oplus \cdots,
\end{array}
$$
where $Q$ is a projective $H$-module (see Proposition 3 of
\cite{Kh}). It is then easily seen that the map $A\o A \o \Bbbk
\cong A \o A  \xrightarrow{\delta_0=m} A$ extends to $\mathbf{p}A
\twoheadrightarrow A$. The cone of this map, when ignoring the
contributions from factors containing tensor products with $Q$, is just $\mathbf{a}A$, which is
contractible. The corollary follows by inducing
$(\mathbf{p}A\twoheadrightarrow A)$ up to the resolution
$\mathbf{p}M\twoheadrightarrow M$.
\end{proof}

\begin{rmk}The more general notion of $H$-module algebra would be
``$H$-module category'', which is a graded category (including the
cyclic $\Z/(n)-$graded case as well) with a finite dimensional
(graded super) Hopf algebra action on the $\Hom$ spaces between
objects. A first example of such a category which is not an
$H$-module algebra (i.e. there are infinitely many objects) is the
category $H\dmod$. More generally, the graded module category over
$B=A\#H$ is another example of such a category. The algebra $A$
itself is an $H$-module category with a single object whose
endomorphism space is given by $A$, together with the defining $H$
action. Our treatment follows Keller's treatment of DG categories
\cite{Ke1} closely and the above story generalizes without much
difficulty to the categorical case.
\end{rmk}


\section{Compact modules} In this section, we follow Neeman's
original treatment in \cite{Nee} to discuss compact hopfological
modules. Thankfully, Neeman's original setup was general enough that
it can be applied here without essential modification. See also
Keller \cite[Section 5]{Ke1} for another account of Neeman's
treatment, where the notion of generators of a triangulated category
appears to be slightly different. However, it turns out that the two
notions are equivalent.

Throughout this section, we make the same assumption as in the
previous section that $H$ is a finite dimensional Hopf algebra over
the base field $\Bbbk$, and $A$ is an $H$-module algebra. We let
$\mc{D}$ denote a $\Bbbk$-linear triangulated category that admits
infinite direct sums.

\subsection{Generators}
We begin with a discussion of the notion of compact generators for
$\mc{D}(A,H)$.

\begin{defn}\label{def-compact-objects}An object $X\in \mc{D}$ is said to be
\emph{compact} if the functor $$\Hom_{\mc{D}}(X,-):\mc{D} \lra \vect$$ commutes
with arbitrary direct sums.
\end{defn}

The following lemma is obvious from the definition and the axioms of
triangulated categories.

\begin{lemma}\label{lemma-2-out-of-3-compact} In any distinguished
triangle in $\mc{D}$, if two out of the three objects in the
distinguished triangle are compact, so is the third. $\hfill\Box$
\end{lemma}

The next lemma gives us the easiest examples of compact objects in
$\mc{D}(A,H)$.

\begin{lemma}\label{lemma-A-tensor-V-is-compact}For any finite dimensional
$H$-module $V$, the hopfological module $A\o V$ is compact in
$\mc{D}(A, H)$.
\end{lemma}
\begin{proof} Since $A\o V$ is cofibrant, using
Lemma \ref{cor-equivalence-cofibrant-subcat}, we have:
$$\begin{array}{rl}
\Hom_{\mc{D}(A,H)}(A\o V, \oplus_{i\in I} M_i) & \cong
\Hom_{\mc{C}(A,H)}(A\o V, \oplus_{i \in I} M_i) \\
& \cong  \mc{H}(\Hom_A(A\o V, \oplus_{i \in I}M_i)) \cong
\mc{H}(\Hom_\Bbbk(V,\oplus_{i\in I}M_i))\\
& \cong  \oplus_{i\in I}\mc{H} (\Hom_\Bbbk(V,M_i)) \cong \oplus_{i
\in I} \mc{H} (\Hom_A(A\o V, M_i))\\
& \cong \oplus_{i \in I} \Hom_{\mc{D}(A,H)}(A\o V, M_i),
\end{array}$$
where, in the fourth equality, we used that $V$ is finite dimensional
(thus compact) and taking $\mc{H}$ commutes with direct sums
(Corollary \ref{cor-taking-cohomology-commute-with-direct-sums}).
The lemma follows.
\end{proof}

\begin{cor} \label{cor-compact-shifted-remains-compact} Let $A\o V$
be as in the previous lemma. Then $T^n(A\o V)$ is compact for any
$n\in \Z$.
\end{cor}
\Pf{Of course, this can be seen without the previous lemma since the
shift functors are automorphisms of $\mc{D}(A,H)$. Alternatively,
recall that the shifts $T$, $T^{-1}$ are given by right tensoring
$A\o V$ with the finite dimensional $H$-modules $H/\Bbbk\Lambda$,
$\textrm{Ker}(\epsilon)$ respectively (Proposition
\ref{prop-invertible-shift-functor}). The compactness of $T(A\o
V)=A\o V\o (H/\Bbbk\Lambda)$ etc.~then follows directly from the
previous lemma.}

\begin{defn}[Neeman]\label{def-compact-generators} Let $\mc{D}$ be
as above. We say that $\mc{D}$ is \emph{generated by a set of
objects} if there exists a set $\mc{G}=\{G_i\in \mc{D}|i\in I\}$ so
that for any $X \in \mc{D}$, $X\cong 0$ if and only if
$$\Hom_\mc{D}(T^{n}(G_i), X)=0$$
for all $n \in \Z$ and $G_i \in \mc{G}$. $\mc{D}$ is said to be
\emph{compactly generated} if $\mc{D}$ is generated by a set
$\mc{G}$ consisting of compact objects.
\end{defn}
As an example of this definition, we show that $\mc{D}(A,H)$ admits
a set of compact generators.

\begin{prop}\label{prop-compact-generators-for-derived-categories}The
derived category $\mc{D}(A,H)$ is compactly generated by the finite
set of objects $\mc{G}:=\{A\o V\}$, where $V$ ranges over a finite
set of representatives of isomorphism classes of simple $H$-modules.
\end{prop}
\begin{proof}It suffices to show that, if an object $X\in
\mc{D}(A,H)$ satisfies $\Hom_{\mc{D}(A,H)}(A\o V, X)=0$ for all $A\o
V \in \mc{G}$, then $X \cong 0$ in $\mc{D}(A,H)$.

Firstly, we show that the hypothesis implies that
$\Hom_{\mc{D}(A,H)}(A\o W , X)=0$ for any finite dimensional
$H-$module $W$. We prove this by induction on the length of $W$, the
length $1$ case following by the assumption. Inductively, take any
finite dimensional irreducible submodule $W^\prime$ of $W$ and form
the quotient $W^{\prime\prime} =W/W^\prime$. $W^{\prime\prime}$ has
shorter length by construction, and we have a short exact sequence
of cofibrant modules
$$0\lra A\o W^\prime \lra A\o W \lra A\o W^{\prime\prime} \lra 0.$$
This short exact sequence becomes a distinguished triangle of
cofibrant modules in $\mc{D}(A,H)$ and applying
$\Hom_{\mc{D}(A,H)}(-,X)$ to the triangle leads to a long exact
sequence
$$\begin{array}{lll}\cdots \lra \Hom_{\mc{D}(A,H)}(T^n(A\o W^{\prime\prime}), X) & \lra &
\Hom_{\mc{D}(A,H)}(T^n(A\o W), X)\\
& \lra & \Hom_{\mc{D}(A,H)}(T^n(A\o W^\prime), X) \lra \cdots
\end{array}$$ The two end terms vanish by assumption and inductive
hypothesis, therefore so does the middle term.

 Next, we show that $\Hom_{\mc{D}(A,H)}(T^n(A\o W), X)$
vanishes for any indecomposable $H$-module $W$ ($W$ could be
infinite dimensional). The strategy is to filter $W$ by finite
dimensional submodules, which we used in the proof of Lemma
\ref{lemma-infinite-projective-H-modules}. Tensoring the short exact
sequence there with $A$, we obtain a short exact sequence of
$B$-modules
$$0\lra \bigoplus_{i\in I}A \o W_i \stackrel{\Id_A\o \Psi}{\lra}
\bigoplus_{i\in I} A\o W_i \lra A\o W \lra 0,$$ where each $W_i$ is
finite dimensional. Applying $\Hom_{\mc{D}(A,H)}(-,X)$ to the
corresponding distinguished triangle and using the previous step
finishes this step.

Thirdly, we prove the vanishing of $\Hom_{\mc{D}(A,H)}(T^n(P),X)$
for all $P$ with property (P). Consider the following short exact
sequence of $B=A\#H$ modules used in Lemma
\ref{lemma-property-p-cofibrant}:
$$0\lra \bigoplus_{r\in \N}F_r \stackrel{\Psi}{\lra}
\bigoplus_{s \in \N}F_s \lra P \lra 0.$$ An induction argument on
$q$ using the previous step shows that
$\Hom_{\mc{D}(A,H)}(T^n(F_r),X)=0$ for all $r\in \N$, $n\in \Z$.
Then applying $\Hom_{\mc{D}(A,H)}(-,X)$ to the distinguished
triangle associated with the above short exact sequence gives us a
long exact sequence
$$\begin{array}{lll}\cdots \lra \prod_{s\in \N}\Hom_{\mc{D}(A,H)}(T^n(F_s), X)
& \lra & \Hom_{\mc{D}(A,H)}(T^n(P), X)\\
& \lra & \prod_{r\in \N}\Hom_{\mc{D}(A,H)}(T^{n+1}(F_r), X) \lra
\cdots.
\end{array}$$
Both ends vanish and the claim follows

Finally, for any object $X\in \mc{D}(A,H)$, take its bar resolution
$\mathbf{p}X\cong X$ (\ref{cor-cofibrant-replacement}), where
$\mathbf{p}X$ satisfies property (P). Then
$$\Hom_{\mc{D}(A,H)}(X,X)\cong \Hom_{\mc{D}(A,H)}(\mathbf{p}X,X),$$
and the right hand side vanishes by the previous step. It follows
that $\Id_X\cong 0$ and $X\cong 0$, finishing the proof of the
lemma.
\end{proof}

\begin{rmk}[On the notion of generators]In the above proposition,
we can equivalently take one compact generator $A\o W$ where $W$ is
a direct sum of simple $H$-modules, one from each isomorphism
classes. Notice that, when $H$ is a local Hopf algebra of
\emph{finite type}, we can replace condition P3 of property (P)
(Definition \ref{def-property-P}) with the equivalent requirement
that $F_r/F_{r+1} \cong A$ instead. Here by finite type we mean that
the set of isomorphism classes of indecomposable modules over $H$ is
finite. Indeed, in this case, the dimensions of indecomposable
modules are bounded, and thus any direct sum of indecomposable
$H$-modules $V$ admits a finite step filtration whose subquotients
are isomorphic to the trivial $H$-module. Therefore by refining the
original filtration of condition P3 by inducing this filtration of
$V$'s, we obtain a new filtration whose subquotients are just
isomorphic to the free module $A$ (with appropriate grading shifts
in the graded case). In particular, this allows us to see
immediately that $A$ generates $\mc{D}(A,H)$ in the stronger sense of
Keller \cite[Section 4.2]{Ke1}:
\begin{itemize}
\item \emph{``$\mc{D}(A,H)$ is the smallest strictly
\footnote{A subcategory $\mc{D}^\prime$ of $\mc{D}$ is called
strictly full if any object of $\mc{D}$ that is isomorphic to some
object in $\mc{D}^\prime$ must itself be in $\mc{D}^\prime$.} full
triangulated subcategory in itself which contains $A$ and is closed
under taking arbitrary direct sums and forming distinguished
triangles.''}
\end{itemize}
It is readily seen that this seemingly stronger version of
generators implies the notion we used in Definition
\ref{def-compact-generators}.

By contrast, for almost all finite dimensional Hopf algebras $H$,
the set of isomorphism classes of indecomposable $H$-modules may
well be infinite, and there is in general no good parametrization of
these isomorphism classes. Over such an $H$, it seems that the
definition of property (P) using all indecomposable modules is more
natural and fits the construction of the bar resolution we gave
previously. Moreover, using the bar resolution, Proposition
\ref{prop-compact-generators-for-derived-categories} shows that a
natural set of compact generators is given by $\{A \o V\}$,  where
$V$ ranges over the representatives of isomorphism classes of simple
$H$-modules. Thus one might wonder whether in the generic case of
$H$ there would still be a similar relation between the two notions
of generators. By a localization theorem of Thomason-Neeman, they
are always equivalent.
\end{rmk}

\begin{cor}\label{cor-equivalence-of-generators} $\mc{D}(A,H)$ is
the smallest strictly full triangulated subcategory in itself that
contains $\mc{G}=\{A\o V\}$ and is closed under taking arbitrary
direct sums and forming distinguished triangles.
\end{cor}

\Pf{The proof is just a corollary of the following theorem, where we
take $R=T^{\Z}(\mc{G}):=\{T^n(G)|G\in \mc{G}, n \in \Z\}$, and
$\mc{G}$ is the set of compact generators we exhibited in
Proposition \ref{prop-compact-generators-for-derived-categories}.}

\begin{thm}[Thomason-Neeman]\label{thm-thomason-neeman} Let $\mc{D}$
be a compactly generated triangulated category. Let $R$ be a set of
compact objects of $\mc{D}$ closed under the shift functor $T$ of
$\mc{D}$. Let $\mc{R}$ be the smallest full subcategory of $\mc{D}$
containing $R$ and closed with respect to taking coproducts and
forming triangles. Then:
\begin{enumerate}
\item The category $\mc{R}$ is compactly generated by the set of
generators $R$.
\item If $R$ is also a set of generators for $\mc{D}$, then
$\mc{R}=\mc{D}$.
\item The compact objects in $\mc{R}$ equals $\mc{R}^c=\mc{D}^c\cap \mc{R}$.
In particular, if $R$ is closed under forming triangles and taking direct summands,
it coincides with $\mc{R}^c$.
\end{enumerate}
\end{thm}
\Pf{This is part of Theorem 2.1 in \cite{Nee}.}

\subsection{Compact modules}
\paragraph{Brown representability theorem.} We recall the notion of
homotopy colimits in a triangulated category that admits infinite
direct sums. Homotopy colimits are used in the construction of
representable functors on the triangulated category (Brown's
representability theorem).

\begin{defn}\label{def-homotopy-colimits} Let $\mc{D}$ be as before.
Let $\{f_n:X_n\lra X_{n+1}|n\in \N\}$ be a sequence of morphisms in
$\mc{D}$. A homotopy colimit of this sequence is an object $X\in
\mc{D}$ that fits into a distinguished triangle as follows:
$$\bigoplus_{n\in \N}X_n\stackrel{\Psi}{\lra}\bigoplus_{n \in \N}X_n
\lra X \lra T\left(\bigoplus_{n\in \N}X_n\right),$$ where $\Psi$ is
given by the infinite matrix
$$\Psi=\left(
\begin{array}{ccccc}
\Id_{X_1} & -f_{1} & 0 & 0 & \dots\\
0 & \Id_{X_2} & -f_{2} & 0 & \dots\\
0 & 0 & \Id_{X_3} & -f_{3} & \dots\\
0 & 0 & 0 & \Id_{X_4} & \dots\\
\vdots & \vdots & \vdots & \vdots &\ddots
\end{array}
\right).
$$
Notice that such $X$ is unique up to isomorphisms in $\mc{D}$.
\end{defn}

\begin{thm}[Brown representability]\label{thm-brown-representability}
Let $\mc{D}$ be a triangulated category that admits infinite direct
sums. Suppose $\mc{D}$ is compactly generated by a set of generators
$\mc{G}$. A cohomological functor $F: \mc{D}\lra (\vect)^{op}$ is
representable if and only if it commutes with direct sums. When
representable, such an $F$ is represented by the homotopy colimit of
a sequence $\{f_r:X_r\lra X_{r+1}|r\in \N\}$ where $X_1$ as well as
the cone of any $f_n$ is represented by a possibly infinite direct
sum of objects of the form $T^n(G)$, with $G\in \mc{G}$ and $n\in
\Z$.
\end{thm}
\Pf{See \cite[Theorem 3.1]{Nee}.}

\paragraph{Characterizing compact modules.} The fact
that $\mc{D}(A,H)$ is compactly generated allows us to give an
alternative characterization of compact hopfological modules as
summands of iterated extensions of a finite number of free modules
of the form $T^{n}(A\o V)$ where $V$ belongs to the set of simple
$H$-modules. The original idea of the proof is due to
Ravenel~\cite{Ra} and Neeman~\cite{Nee2}, and a very readable
account of the proof is given by Keller~\cite[Section 5.3]{Ke1},
which we follow.

\begin{defn}\label{def-iterated-extension} Let $\mc{D}$ be a
triangulated category as above and $\mc{U}$, $\mc{V}$ be two classes
of objects of $\mc{D}$. Let $\mc{U}*\mc{V}$ be the class of objects
$X$ in $\mc{D}$ that fit into a distinguished triangle of the form
$$G_1\lra X \lra G_2\lra T(G_1),$$ where $G_1\in \mc{U}$ and $G_2\in
\mc{V}$. The lemma below says that the operation $*$ is associative,
and therefore we can define unambiguously the class of length $n$
objects generated by $\mc{W}$ to be the class of objects in
$$\mc{W}*\mc{W}*\cdots*\mc{W},$$
where there are $n$ copies of $\mc{W}$. We will refer to objects
belonging to $\mc{W}*\mc{W}*\cdots*\mc{W}$ for some $n\in \N$ as a
\emph{finite extension of objects in $\mc{W}$}.
\end{defn}

\begin{lemma}\label{lemma-star-associativity}The above operation $*$
is associative in the sense that the two classes of objects
$(\mc{U}*\mc{V})*\mc{W}$, $\mc{U}*(\mc{V}*\mc{W})$ coincide.
\end{lemma}
\begin{proof}The octahedral axiom for the morphisms $u$ and $v$
and their composition gives us a commutative diagram:
$$\xymatrix{ & U\ar[d]_u\ar @{=}[r] & U\ar[d]_{v\circ u} & \\
T^{-1}(W)\ar[r] \ar @{=}[d] & X\ar[r]_v\ar[d] & Z\ar[r]
\ar[d] & W\ar @{=}[d] \\
T^{-1}(W) \ar[r]& V \ar[r] \ar[d] & Y
\ar[r]\ar[d] & W\\
& T(U)\ar @{=}[r] & T(U) &,}$$ where we take $Y=C_{v\circ u}$,
$V=C_u$ and $W=C_v$. The horizontal and vertical sequences are
distinguished triangles. Read vertically, the diagram says that $Z$
belongs to $(\mc{U}*\mc{V})*\mc{W}$, while read horizontally, it
gives another realization of $Z$ as an object of
$\mc{U}*(\mc{V}*\mc{W})$.
\end{proof}

\begin{thm}[Ravenel-Neeman]\label{thm-characterizing-compact-objects}
Let $\mc{D}$ be a triangulated category compactly generated by a set
of generators $\mc{G}$. Any compact object of $\mc{D}$ is then a
direct summand of a finite extension of objects of the form
$T^{n}(G)$, where $G\in \mc{G}$ and $n \in \Z$.
\end{thm}
\begin{proof}[Sketch of proof] See \cite{Ra, Nee2, Ke1}. The formulation
given here is the same as that of \cite[Theorem 5.3]{Ke1}. The idea
of proof is to apply Brown's representability theorem to the
cohomological functor $\Hom_{\mc{D}}(-,M)$ for any compact object
$M\in \mc{D}$. Then compactness of $M$ allows us to factor the
identity morphism of $X$ through some $X_i$, a finite step of the
homotopy-limit-approximation of $X$ (in the notation of
\ref{def-homotopy-colimits}). It can be seen from the second part of
the Brown representability theorem that $X_i \in T^{\Z}(\mc{G})
* T^{\Z}(\mc{G})* \cdots *T^{\Z}(\mc{G})$ for $i$ copies of $T^{\Z}
(\mc{G})$. Finally the theorem follows from a ``d\'evissage'' type
of argument on the length of $X_i$, using the octahedral axiom.
\end{proof}

\begin{cor}\label{cor-characterizing-compact-objects} Let
$\mc{D}^c(A,H)$ denote the strictly full subcategory of compact
hopfological modules in $\mc{D}(A,H)$. It is triangulated and
idempotent complete. Any $X\in \mc{D}^c(A,H)$ is a direct summand of
an object which is a finite extension of modules in
$T^{\Z}(\mc{G})=\{T^n(A\o V)\}$, where $n\in \Z$ and $V$ ranges over
the set of representatives of isomorphism classes of simple
$H$-modules. Furthermore, $\mc{D}^c(A,H)$ is the smallest strictly full
triangulated subcategory of $\mc{D}(A,H)$ that contains $\mc{G}$ which
is closed under taking direct summands.
\end{cor}
\Pf{Combine the previous theorem with Proposition
\ref{prop-compact-generators-for-derived-categories}. The last statement follows
from Theorem \ref{thm-thomason-neeman}.}


\begin{defn}\label{def-Grothendieck-group} Let $A$ be an $H$-module
algebra over a finite dimensional Hopf algebra $H$ over the base
field $\Bbbk$. We define the Grothendieck group $K_0(\mc{D}^c(A,H))$
(or $K_0(A,H)$ for short) to be the abelian group generated by the
symbols of isomorphism classes of objects in $\mc{D}^c(A,H)$, modulo
the relations
$$[Y]=[X]+[Z],$$
whenever there is a distinguished triangle inside $\mc{D}^c(A,H)$ of
the form
$$X\lra Y \lra Z \lra T(X).$$
\end{defn}

\begin{rmk}\label{rmk-higher-k-theory}
Since $\mc{D}^c(A,H)$ is a (right) triangulated module-category over
$H\udmod$, on the Grothendieck group level, $K_0(\mc{D}^c(A,H))$ is
a (right) module over $K_0(H\udmod)$. When $H$ is cocommutative,
$K_0(H\udmod)$ is a commutative ring and there is no need to
distinguish right or left modules over it.

More generally, we can define higher $K$-groups of $A$ by applying
Waldhausen-Thomason-Trobaugh's construction to $\mc{D}^c(A,H)$. We
expect a large chunk of the K-theoretic results of Thomason-Trobaugh
\cite{TT} and Schlichting \cite{Sch} to generalize to our case.
\end{rmk}

\subsection{A useful criterion}
As another application of Thomason-Neeman's Theorem
\ref{thm-thomason-neeman} and the notion of compactly generated
categories \ref{def-compact-generators}, we give a useful criterion
concerning the fully-faithfulness of exact functors on a compactly
generated triangulated category and natural transformations between
these functors. Of course the main example of such categories we
have in mind are the derived categories of $H$-module algebras. The
criterion will be needed in the next section.

\begin{lemma}\label{lemma-criterion-functors-natural-transformations}
Let $\mc{D}_1$, $\mc{D}_2$ be triangulated categories,
$F,F^{\prime}:\mc{D}_1\lra \mc{D}_2$ be exact functors between them,
and $\mu: F\Rightarrow F^{\prime}$ be a natural transformation of
these functors. Suppose furthermore that $\mc{D}_1$ admits arbitrary
direct sums and is compactly
generated by a set of generators $\mc{G}$, $F$, $F^\prime$ commute
with direct sums\footnote{This
amounts to saying that $F(\oplus_{i \in I} X_i)$ is a direct sum
object for $F(X_i)$, $i\in I$ inside $\mc{D}_2$ although $\mc{D}_2$
may not admit arbitrary direct sums.}. Then:
\begin{enumerate}
\item $F$ is fully-faithful if $F$ restricted to the full subcategory
consisting of objects in $T^{\Z}(\mc{G}):={\cup_{n\in
\Z}T^n(\mc{G})}$ is fully faithful and $F(G)$ is compact for all
$G\in \mc{G}$. The converse holds if $F$ is essentially surjective
on objects\footnote{By ``essentially surjective'' we mean that any
object of $\mc{D}_2$ is isomorphic to an object in the image of
$F$.}.
\item $\mu$ is invertible if and only if $\mu(G):F(G)\lra
F^{\prime}(G)$ is invertible for all $G \in \mc{G}$.
\end{enumerate}
\end{lemma}

\begin{proof}To prove 1, notice that the full
subcategory consisting of objects $X$ on which the functor $F$
induces an isomorphism of vector spaces
$$\Hom_{\mc{D}_1}(T^n(G),X)\cong \Hom_{\mc{D}_2}(F(T^n(G)),F(X))$$
form a strictly full triangulated subcategory of $\mc{D}_1$. By the
compactness assumption on $F(G)$, this subcategory contains
arbitrary direct sums. Now Theorem \ref{thm-thomason-neeman} applies
since $\mc{D}_1$ is compactly generated. The converse is true since
if $F$ is essentially surjective on objects, $F(G)$ is then
automatically compact whenever $G$ is.

The second claim follows by considering instead the full subcategory
in which $\mu{X}:F(X)\lra F^\prime(X)$ is invertible. Similar
arguments as above show that this subcategory is a strictly full
triangulated subcategory, and it contains all the compact
generators. Therefore it coincides with the whole category.
\end{proof}

\begin{cor}\label{cor-criterion-functor-induce-equivalence}
Let $F:\mc{D}_1\lra\mc{D}_2$ be an exact functor between
$\Bbbk$-linear triangulated categories which are compactly generated
and admit arbitrary direct sums. Suppose $F$ also commutes with
direct sums. Let $\mc{G}=\{G\}$ be a set of compact generators for
$\mc{D}_1$. Then $F$ induces an equivalence of triangulated
categories if and only if when restricted to the full subcategory
consisting of objects $T^{\Z}(\mc{G}):=\cup_{n\in \Z}T^n(\mc{G})$ it
is fully-faithful, and $F(\mc{G}):=\{F(G)|G\in\mc{G}\}$ is a set of
compact generators for $\mc{D}_2$.
\end{cor}

\Pf{$F$ induces an equivalence of categories between $\mc{D}_1$ and
the image $F(\mc{D}_1)$. By Theorem \ref{thm-thomason-neeman}, the
image category coincides with $\mc{D}_2$.}


\section{Derived functors} In this section, we define the
derived functors associated with hopfological bimodules. Then we
proceed to prove a sufficient condition for two $H$-module algebras
to be derived Morita equivalent. As a corollary, we discuss when a
morphism of $H$-module algebras induces an equivalence of derived
categories. The arguments we use are modeled on the DG case, as in
Keller \cite[Section 6]{Ke1}.

Throughout this section, we will assume that
$H$ is also a \emph{(co)commutative} Hopf algebra. This condition is
needed when we define a left module-algebra structure on the
opposite algebra $A^{op}$ of a left module-algebra $A$, and when
dealing with derived functors and derived equivalences. We will make
some further remarks on this assumption later.

\subsection{The opposite algebra and tensor product}
By the construction of $B=A\#H$, it is readily seen that the opposite
algebra of $B$ is isomorphic to the smash product ring
$B^{op}=H^{op,cop}\#A^{op}$, where $H^{op,cop}$ denotes the Hopf
algebra $H$ with the opposite multiplication and opposite
comultiplication. Therefore, $A^{op}$ is naturally a right
$H^{op,cop}$-module algebra, or equivalently, a left
$H^{cop}$-module algebra ($H^{cop}$ becomes a Hopf algebra if we
equip with it the antipode map $S^{-1}$). By our assumption $H$ is
(co)commutative, and we can naturally identify $H^{cop}\cong H$
($S^{-1}=S$ in this case). Therefore, we have a left $H$-module
algebra structure on $A^{op}$.

\begin{defn} Let $H$ be a cocommutative Hopf algebra, and $A$ be an
$H$-module algebra as in the main example \ref{the-main-example}. We
define the \emph{opposite $H$-module algebra} $A^{op}$ to be the
same $H$-module as $A$ but with the opposite multiplication. An
analogous definition applies when $A$, $H$ are compatibly
$\Z$-graded.
\end{defn}

\begin{eg} We give an example showing the necessity of assuming
$H$ to be cocommutative. Consider an $n$-DG algebra $A$ equipped
with a differential $d$ of degree $1$ (see the second example of
Section 4). For any $a,~b\in A$, we have:
$$d(ab)=(da)b+\nu^{|a|}a(db),$$
where $\nu$ is an $n$-th root of unity and $|a|\in \Z$ denotes the
degree of $a$. As such an algebra can be regarded as a graded module algebra over
the Taft algebra $H_n$ at the $n$-th root of unity $\nu$ (see \cite{Bi} and
the second example of Section 3.4), which is
non-commutative and non-cocommutative. Now in $A^{op}$, whose
multiplication will be denoted by $\circ$, we have $a\circ
b=\xi^{|b||a|}ba$, where we allow $\xi$ to be some other $n$-th root
of unity. Then
$$ \begin{array}{rcl}
d(a\circ b)& = &d(\xi^{|b||a|}ba)=\xi^{|b||a|}((db)a+\nu^{|b|}b(da))\\
& = & \xi^{|b||a|}(\xi^{(|b|+1)|a|}a\circ
(db)+\nu^{|b|}\xi^{(|a|+1)|b|}(da)\circ b)\\
& = & \xi^{(2|b|+1)|a|} a \circ (db) +
\nu^{|b|}\xi^{|b|(2|a|+1)}(da)\circ b,
\end{array}$$
Compare with the relation we need to make $A^{op}$ differential
graded: $d(a\circ b) = (da)\circ b+\eta^{|a|}a\circ(db)$. Now assume
$A$ has non-zero terms in each degree, it is easy to see that in
order to make these expressions equal, we need $\xi=\pm 1$ and
$\nu=\xi^{-1}$. Thus it appears that the opposite algebra does not
carry a natural $n$-DG structure if $\nu \neq \pm 1$.
\end{eg}

\begin{defn}\label{def-tensor}Let $H$ be a cocommutative Hopf
algebra. Let $M$ be a left $A^{op}\#H$-module and $N$ be a left
$A\#H$-module. The tensor product space $M\o_{A}N$ is naturally an
$H$-module by setting, for any $m\in M$, $n\in N$ and $h\in H$,
$$h(m\o n):=\sum (h_{(1)}m)\o(h_{(2)}n).$$ The $H$-module
$M\o_{A}N$ is graded if $H$, $A$, $M$, $N$ are compatibly graded.
\end{defn}
We have, as $H$-modules, $M\otimes_{A}A\cong M$, and $A\otimes_{A}N
\cong N.$

One checks easily that, when $H$ is cocommutative, we have an
equivalence between the categories of right $A\#H$-modules and the
category of left $A^{op}\#H$-modules. Indeed, for any right
$A\#H$-module $M$, we define the corresponding left
$A^{op}\#H$-module to be the same underlying $H$-module with the
left $A^{op}$ action given by $a\circ m := ma$, for any element
$a\in A$ and $m\in M$. The compatibility of this left
$A^{op}$-structure with the $H$-module structure is guaranteed by
the cocommutativity of $H$.

Now, if $M$ is a $B$-module which is finitely presented as an
$A$-module (finitely generated if $A$ is noetherian), we have a
canonical isomorphism of $H$-modules
$$\Hom_{A}(M,N)\cong M^{\vee}\otimes_{A}N,$$
where $M^{\vee}$ denotes the $H-$module $\Hom_A(M,A)$, equipped with
the right $A$-module structure from that of the target $A$. A
similar identification holds in the graded case.

\subsection{Derived tensor}
Our first task is to define the derived tensor functor associated
with a hopfological bimodule and determine when it induces an
equivalence of derived categories. We will denote by
$A_1,~A_2$ two $H$-module algebras over a finite dimensional
(graded, super) cocommutative Hopf algebra $H$, and set
$B_1=A_1\#H$, $B_2=A_2\#H$.

\begin{defn}\label{def-tensor-module-algebra} Let $A_1$, $A_2$ be as
above, and define their \emph{tensor product $H$-module algebra}
$A_1\o A_2$ to be the usual tensor product of $A_1$, $A_2$ as a
$\Bbbk$-vector space and the algebra structure given by
$$(a_1\o a_2)\cdot(b_1\o b_2):=(a_1b_1)\o(a_2b_2),$$
for any $a_1,~b_1\in A_1$, $a_2~b_2\in A_2$.
We equip it with the $H$-action that, for any $h \in H$, $a_1\o a_2
\in A_1\o A_2$,
$$h\cdot(a_1\o a_2):=\sum h_{(1)}a_1\o h_{(2)}a_2.$$
It is readily checked that $A_1\o A_2$ indeed satisfies the axioms of
an $H$-module algebra under the assumption that $H$ is
cocommutative.
\end{defn}

Now let $A_1$, $A_2$ be as above and $_{A_1} X_{A_2}$ be an
$(A_1,A_2)$ hopfological bimodule, i.e. a module over the ring
$(A_1\o A_2^{op})\#H$. We define the associated tensor and hom
functors to be:
$$_{A_1} X_{A_2}\o_{A_2}(-): A_2\dmod \lra A_1\dmod,~ _{A_2}N \mapsto
~_{A_1}X\o_{A_2}N;$$
$$\Hom_{A_1}(_{A_1}X_{A_2},-): A_1\dmod \lra A_2 \dmod,~ _{A_1}M \mapsto
\Hom_{A_1}(X_{A_2}, M).$$ In the above definition and what follows,
we omit some of the subscripts whenever no confusion can arise. For
instance,
$\Hom_{A_1}(X_{A_2},M):=\Hom_{A_1}(_{A_1}X_{A_2},{_{A_1}M})$. The
natural left $A_2$-module structure on the right hand side is
compatible with the $H$-action under the assumption that $H$ is
cocommutative. Therefore $\Hom_{A_1}(X_{A_2},M)\in B_2\dmod$, and
more generally one easily checks that both maps above are compatible
with the $H$-actions on the algebras and modules, thus inducing
functors on the corresponding $B$-module categories. We leave the
analogous statements and their verification in the graded case to
the reader; their proofs are similar to the argument we use in the
next lemma.

\begin{lemma}\label{lemma-canonical-adjunction-preserves-H-module-structure}
The canonical adjunction between the tensor and hom functors in the
above definition associated with the bimodule $_{A_1} X_{A_2}$,
$$\Hom_{A_1}(X\o_{A_2}N, M)\cong\Hom_{A_2}(N,\Hom_{A_1}( X_{A_2}, M)),$$
is an isomorphism of $H$-modules, functorial in $M$ and $N$ for any
$M\in B_1\dmod$, $N \in B_2\dmod$. A similar statement holds in the
graded case.
\end{lemma}
\begin{proof} Recall that under the tensor-hom adjunction, we associate with any element
$f \in \Hom_{A_2}(N,\Hom_{A_1}( X_{A_2}, M))$ the element of
$\Hom_{A_1}(X\o_{A_2}N, M)$, still denoted $f$, which sends $x\o n$
to $f(n)(x)$. On one hand, for any $h\in H$, $h\cdot f \in
\Hom_{A_2}(N,\Hom_{A_1}( X_{A_2}, M))$ is given by
$$(h\cdot f)(-)=h_{(2)}f(S^{-1}(h_{(1)})\cdot-):N \lra
\Hom_{A_{1}}(X_{A_2},M).$$ Thus for any $n\in N$, $x\in X$, we have
from the above assignment
$$(h\cdot f)(x \o n)=(h_{(2)}\cdot f(S^{-1}(h_{(1)})\cdot n))(x)=h_{(3)}
f(S^{-1}(h_{(1)})\cdot n)(S^{-1}(h_{(2)})\cdot x).$$

On the other hand, when regarding $f$ as an element of
$\Hom_{A_1}(X\o_{A_2}N, M)$ using the adjunction, the $H$-action has
the effect
$$\begin{array}{rcl}
(h\cdot f)(x \o n) & = & h_{(2)}f(S^{-1}(h_{(1)})\cdot (x\o n))  =
h_{(3)}(f(S^{-1}(h_{(2)})\cdot x\o S^{-1}(h_{(1)})\cdot n))\\
& = & h_{(3)}(f(S^{-1}(h_{(1)})\cdot n)(S^{-1}(h_{(2)})\cdot x)).
\end{array}$$
This shows that the two expressions are equal and the lemma follows.
\end{proof}

Taking stable invariants $\mc{H}$ (Proposition
\ref{prop-comparison-hom-spaces}) of the above canonical isomorphism
gives us the corresponding adjunction in the homotopy categories.

\begin{cor}\label{cor-adjunction-homotopy-category} The functors
$_{A_1}X\o_{A_2}(-)$, $\Hom_{A_1}(X_{A_2},-)$ descend to adjoint
functors in the homotopy category:
$$\Hom_{\mc{C}(A_1,H)}(X\o_{A_2}N, M)\cong\Hom_{\mc{C}(A_2,H)}(N,\Hom_{A_1}( X_{A_2}, M))$$
functorially in $M\in B_1\dmod$, $N \in B_2\dmod$. $\hfill\Box$
\end{cor}

\begin{defn}\label{def-derived-bimodule-tensor} Let $_{A_1}X_{A_2}$
be as above. We define the (left) derived tensor functor
$_{A_1}X\Lo_{A_2}(-)$ to be the composition:
$$_{A_1}X\Lo_{A_2}(-):\mc{D}(A_2,H)\stackrel{\mathbf{p}}{\lra}
\mc{P}(A_2,H) \xrightarrow{X\o_{A_2}(-)} \mc{C}(A_1,H)
\stackrel{Q}{\lra} \mc{D}(A_1,H)$$
$$_{A_2}M \mapsto {_{A_1}X\o_{A_2}} \mathbf{p}M$$
where $\mathbf{p}$ is the functorial bar resolution of Corollary
\ref{cor-cofibrant-replacement} and $Q$ is the canonical
localization functor.
\end{defn}

\begin{prop}\label{prop-derived-functors-from-bimodule-maps} Let
$_{A_1}X_{A_2}$, $_{A_1}Y_{A_2}$ be $(A_1,A_2)$ hopfological
bimodules, and let
$$\mu: {_{A_1}X_{A_2}} \lra {_{A_1}Y_{A_2}}$$
be a map of hopfological bimodules. Then:
\begin{enumerate}
\item Suppose ${_{A_1}X_{A_2}}$ is cofibrant when regarded as
a $B_1$-module. The functor $$_{A_1}X\Lo_{A_2}(-): \mc{D}(A_2,H) \lra \mc{D}(A_1,H)$$
is an equivalence of categories if and only if $A_2\lra
\Hom_{A_1}(X_{A_2},X_{A_2})$ is a quasi-isomorphism, and
$\{{_{A_1}X\o V}\}$, when regarded as left $B_1$-modules, is a set of
compact cofibrant generators $\mc{D}(A_1,H)$. Here $V$ ranges over a
finite set of representatives of isomorphism classes of simple
$H$-modules.
\item The map of bimodules $\mu$ induces an invertible natural
transformation of functors $$\mu^{\mathbf{L}}: {_{A_1}X
\Lo_{A_2}(-)} \Rightarrow {_{A_1}Y \Lo_{A_2}(-)}$$ if and only if
$\mu$ is a quasi-isomorphism in $(A_1\o A_2^{op})\#H\dmod$.
\end{enumerate}
\end{prop}

\begin{proof}The first statement of the proposition is a consequence
of Corollary \ref{cor-criterion-functor-induce-equivalence},
provided we know that $\mc{D}(A_i,H)$ is compactly generated by the
set of generators $\mc{G}=\{A_i\o V\}$, $i=1,2$ (Proposition
\ref{prop-compact-generators-for-derived-categories}). We check that
under our assumption, the conditions of the corollary are satisfied.
Since $T^n(M)\cong M\o W$ for some finite dimensional $H$-module $W$
(see \ref{def-shift-functor} and
\ref{prop-invertible-shift-functor}), we have for any $A_2\o V,
A_2\o V^\prime \in \mc{G}$, which are property (P) modules:
$$\Hom_{A_2}(T^n(A_2\o V), T^m(A_2\o V^\prime
))\cong \Hom_{A_2}(A_2\o V\o W, A_2\o V^\prime \o W^\prime)$$
$$\cong A_2\o \Hom_\Bbbk(V\o W,V^\prime\o W^\prime)\stackrel{\alpha}{\lra}
\Hom_{A_1}(X,X)\o \Hom_\Bbbk(V\o W, V^\prime\o W^\prime)$$
$$\cong\Hom_{A_1}(X\o V\o W, X\o V^\prime\o W^\prime)\cong \Hom_{A_1}
(X\o_{A_2}(A_2\o V\o W),X\o_{A_2}(A_2\o V^\prime\o W^\prime)),$$
where $\alpha$ is a quasi-isomorphism by assumption. Since $V$,
$V^\prime$, $W$ and $W^\prime$ are finite dimensional, we can pull
$\Hom(V\o W,V^\prime\o W^\prime)$ in and out of the $A_1$-hom
spaces. Taking stable invariants of the first and last hom-spaces
shows that the morphism spaces in the derived categories are
isomorphic as well (here we use that $_{A_1}X$ is cofibrant),
thereby establishing the fully-faithfulness of
the tensor functor when restricted to $T^\Z(\mc{G})$. Furthermore,
the hypothesis says that the modules $_{A_1}X\Lo_{A_2}(A_2\o V)\cong
{_{A_1}X\o_{A_2}(A_2\o V) \cong {_{A_1}X \o V}}$ for the $V$ as in the
assumption constitute a set of compact cofibrant generators of
$\mc{D}(A_1,H)$. Finally, the functor commutes with direct sums
since tensor product does so.

For the second part, note that $X\o_{A_{2}}(A_{2}\o V) \cong X\o V$
is quasi-isomorphic to $Y\o_{A_{2}}(A_{2}\o V) \cong Y\o V$ for all
simple $H$-modules $V$ if and only if $X$ is quasi-isomorphic to
$Y$. Now use part 2 of Lemma
\ref{lemma-criterion-functors-natural-transformations}.
\end{proof}

\begin{cor}\label{cor-two-tensor-equivalent} Let $_{A_1}X_{A_2}$ be
a hopfological bimodule and $_{A_1}(\mathbf{p}X)_{A_2}$ be its bar
resolution in $(A_{1}\o A_{2}^{op})\#H\dmod$. Then $\mathbf{p}X\lra
X$ induces a canonical isomorphism of functors
$$_{A_1}X\Lo_{A_2}(-)\cong {_{A_1}(\mathbf{p}X)\o_{A_2}(-)}: \mc{D}(A_2,H)\lra \mc{D}(A_1,H).$$
\end{cor}

\begin{proof} By the previous result, we have an isomorphism of
functors
$$_{A_1}X\Lo_{A_2}(-)\cong {_{A_1}(\mathbf{p}X)\Lo_{A_2}(-)}:
\mc{D}(A_2,H)\lra \mc{D}(A_1,H).$$ To this end, it suffices to show
that, if a bimodule $_{A_1}P_{A_2}$ has property (P), then
$P_{A_2}\o_{A_{2}} M$ is quasi-isomorphic to $ P_{A_2}\o_{A_{2}}
\mathbf{p}M$ for any $M\in B_2\dmod$.

 Tensoring the short exact sequence of free $A_1\o A_2^{op}$-modules
$$0\lra \bigoplus_{r\in \N}F_r \lra \bigoplus_{s \in \N}F_s \lra P
\lra 0$$
we used in Lemma \ref{lemma-property-p-cofibrant} with the
bar resolution $\mathbf{p}M \lra M$ and passing to the homotopy
category $\mc{C}(A_1,H)$ (Lemma \ref{lemma-split-ses-lead-to-dt}),
we have a morphism of distinguished triangles
$$\xymatrix{\bigoplus_{r\in \N}(F_r  \o_{A_2} \mathbf{p}M) \ar[d]\ar[r] &
\bigoplus_{s \in \N}(F_s \o_{A_2} \mathbf{p}M) \ar[d] \ar[r] & P
\o_{A_2} \mathbf{p}M \ar[d]\ar[r] & T(\bigoplus_{r\in \N}(F_r\o_{A_2} \mathbf{p}M)) \ar[d]\\
\bigoplus_{r\in \N}(F_r  \o_{A_2} M) \ar[r] & \bigoplus_{s \in
\N}(F_s \o_{A_2} M) \ar[r] & P \o_{A_2} M \ar[r] & T(\bigoplus_{r\in
\N} (F_r\o_{A_2} M)).}
$$ Taking cohomology (passing to $H\udmod$ via $\underline{\mathrm{Res}}$) and using the
``two-out-of-three'' property of triangulated categories (see, for
instance, \cite[Corollary 4, Section IV.1]{GM}), we are
reduced to exhibiting the claimed property for each $F_r$, $r\in
\N$. An induction argument on $r$ further reduces us to the special
case when $P=A_1\o A_2 \o V$, which is easily seen to be true:
$$(A_1\o A_2 \o V) \o_{A_2}\mathbf{p}M \cong A_1\o V \o \mathbf{p}M \cong A_1\o V \o M \cong
(A_1\o A_2 \o V)\o_{A_2} M,$$
where the first and last isomorphisms
are that of modules, while the middle one is only a quasi-isomorphism.
\end{proof}

\begin{cor}\label{cor-composition-tensor-functors} Let $A_1$, $A_2$,
$A_3$ be $H$-module algebras, and $_{A_1}X_{A_2}$, $_{A_2}Y_{A_3}$
be hopfological bimodules. Then there is an isomorphism of functors
$$_{A_1}X_{A_2}\Lo_{A_2}(_{A_2}Y_{A_3}\Lo_{A_3}(-))\cong (_{A_1}Z_{A_3}\Lo_{A_3}(-))
:\mc{D}(A_3,H)\lra \mc{D}(A_1,H),$$ where
$_{A_1}Z_{A_3}={_{A_1}(\mathbf{p}X)}\o_{A_2} Y_{A_3}$ and
${_{A_1}(\mathbf{p}X)_{A_2}}$ stands for the bar resolution of $X$
as an $(A_1, A_2)$-bimodule.
\end{cor}
\Pf{Easy by Corollary \ref{cor-two-tensor-equivalent}.}

\subsection{Derived hom}
We next focus on the derived hom functor and exhibit a derived
version of the adjunctions
\ref{lemma-canonical-adjunction-preserves-H-module-structure},
\ref{cor-adjunction-homotopy-category}.

\begin{defn}\label{def-derived-bimodule-hom} Let $_{A_1}X_{A_2}$ be
a hopfological bimodule are before. Let $\mathbf{p}X$ be the bar
resolution of $X$ as a left $B_1$-module. By our construction,
$\mathbf{p}X=\mathbf{p}A_1\o_{A_1}X$ is also a right $B_2$-module.
We define the derived hom functor $\RHom_{A_1}(X_{A_2},-)$ to be the
composition:
$$\RHom_{A_1}(X_{A_2},-): \mc{D}(A_1,H) \xrightarrow{\Hom_{A_1}(\mathbf{p}X,-)}
\mc{C}(A_2,H) \stackrel{Q}{\lra} \mc{D}(A_2,H)$$
$$_{A_1}M \mapsto \Hom_{A_1}((\mathbf{p}X)_{A_2}, M).$$
The next lemma guarantees that $\Hom_{A_1}(\mathbf{p}X,-)$ is well
defined on the derived category $\mc{D}(A_1,H)$.
\end{defn}

\begin{lemma}\label{lemma-derived-hom-well-defined} If $_{A_1}{\tilde
X}_{A_2}$ has property (P) as a left $B_1$-module, then
$\Hom_{A_1}({\tilde X}_{A_2},K)$ is an acyclic $B_2$-module whenever
$K\in B_1\dmod$ is acyclic. Consequently, $\RHom_{A_1}({\tilde
X}_{A_2},-)$ descends to a functor:
$$\RHom_{A_1}({\tilde
X}_{A_2},-):\mc{D}(A_1,H)\lra \mc{D}(A_2,H).$$ Likewise, the result
holds when ``property (P)'' is replaced by ``cofibrant'' in the
statement.
\end{lemma}

\begin{proof}The proof is similar to that of Corollary
\ref{cor-two-tensor-equivalent}. Consider the short exact sequence
of $B_1$-modules $0\lra \bigoplus_{r\in \N}F_r \lra \bigoplus_{s \in
\N}F_s \lra {\tilde X} \lra 0$ associated with ${\tilde X}$. Since
each $F_r,~r\in \N$, and ${\tilde X}$ are free as $A_1$-modules,
applying $\Hom_{A_1}(-,K)$ yields a short exact sequence of
$B_2$-modules:
$$0\lra \Hom_{A_1}({\tilde
X}, K) \lra \prod_{s \in \N}\Hom_{A_1}(F_s,K)\lra  \prod_{r\in
\N}\Hom_{A_1}(F_r,K)\lra 0.$$ Thus it suffices to show that
$\Hom_{A_1}(F_r,K)$ is acyclic for each $r\in \N$. An induction on
$r$ further reduces us to the case of free modules of the form
$A_1\o N$ where $N$ is some indecomposable $H$-module. This case now
follows from Lemma \ref{lemma-infinite-projective-H-modules} since
$\Hom_{A_1}(A_1\o N, K)\cong \Hom_\Bbbk(N,K)$.

The last claim follows readily from the first part of the lemma and
Corollary \ref{cor-cofibrant-direct-summand-of-property-p}.
\end{proof}

\begin{rmk}More generally, it's easy to see that $\RHom_{A_1}(-,-)$ is a
bifunctor
$$\RHom_{A_1}(-,-):\mc{D}(A_1\o A_2^{op},H)^{op}\times \mc{D}(A_1,H)
\lra \mc{D}(A_2,H).$$ In particular, when $A_2\cong \Bbbk$, we have
a bifunctor
$$\RHom_{A_1}(-,-):\mc{D}(A_1,H)^{op}\times \mc{D}(A_1,H)
\lra H\udmod.$$

There is another derived $\Hom$-space one can associate with any two
hopfological modules $M$ and $N$, namely the space of chain maps up
to homotopy
$$\mc{H}(\Hom_A(M,N))=\Hom_{\mc{C}(A,H)}(M,N).$$
By
Proposition \ref{prop-comparison-hom-spaces} and the remark that
follows it, this is the space of (stable) invariants in
$\Hom_A(\mathbf{p}M, N)$, and thus it usually contains less
information than the $\RHom$ above. Another reason that we use the
definition above is that it satisfies the right adjunction property
with the derived tensor product functor as shown in the next lemma.
Notice that in the DG case, i.e. $H=\Bbbk[d]/(d^2)$, the natural map
of $\HOM$-spaces
$\RHOM_A(M,N)\stackrel{\mc{H}}{\lra}\HOM_{\mc{C}(A,H)}(M,N)$ is an
isomorphism since the only stably non-zero modules are the graded
shifts of the trivial module $\Bbbk_0$.
\end{rmk}

\begin{lemma}\label{lemma-Rhom-adjoint-to-Ltensor} $\RHom(X,-)$ is right adjoint
to $X\Lo_{A_1}(-)$ as functors between $\mc{D}(A_i,H)$, $i=1,2$.
\end{lemma}

\begin{proof}Notice that $\mathbf{p}X\o_{A_2} N$ has property (P) as a
$B_1$-module whenever $N\in B_2\dmod$ does (check for $N=A_2\o V$).
Therefore if $M\in B_1\dmod$ and $N\in B_2\dmod$, we have:
$$\begin{array}{rcl}
\Hom_{\mc{D}(A_1,H)}(X\Lo_{A_2}N,M) & \cong
&\Hom_{\mc{D}(A_1,H)}(\mathbf{p}X\o_{A_2}\mathbf{p}N,M) \\
& \cong & \Hom_{\mc{C}(A_1,H)}(\mathbf{p} X\o_{A_2}\mathbf{p}N,M)\\
& \cong & \Hom_{\mc{C}(A_2,H)}(\mathbf{p}N, \Hom_{A_1}(\mathbf{p}X,M))\\
& \cong & \Hom_{\mc{D}(A_2,H)}(N, \RHom_{A_1}(X,M)).
\end{array}$$
Here the first isomorphism holds by Corollary
\ref{cor-two-tensor-equivalent}; the second holds since $\mathbf{p}
X\o_{A_1}\mathbf{p}N$ has property (P) so that we can use Corollary
\ref{cor-equivalence-cofibrant-subcat}; the third holds by
adjunction \ref{cor-adjunction-homotopy-category} in the homotopy category,
while the fourth holds by
Corollary \ref{cor-equivalence-cofibrant-subcat} and the definition
of $\RHom$.
\end{proof}

\begin{defn}\label{def-dual-bimodule}Let $_{A_1}X_{A_2}$ be a
hopfological bimodule as before. We define its \emph{$A_1$-dual} to
be
$$_{A_2}\check{X}_{A_1}:=\Hom_{A_1}(\mathbf{p}X_{A_2},A_1),$$
where its left $A_2$ structure is inherited from the right
$A_2$-module structure of $\mathbf{p}X$, while the right $A_1$
structure comes from that of $A_1$.
\end{defn}

Notice that there is a canonical map
$$_{A_2}\check{X}\Lo_{A_1}M\cong \Hom_{A_1}(\mathbf{p}X,A_1)\o_{A_1} M\lra
 \Hom_{A_1}(\mathbf{p}X, M)\cong \RHom_{A_1}(X,M),$$
 which is an isomorphism whenever $M$ is of the form $A_1\o V$ for
 any finite dimensional $H$-module $V$.

\begin{prop}\label{prop-quasi-inverse-tensor}If $X\Lo_{A_1}(-):\mc{D}(A_1,H)
\lra\mc{D}(A_2,H)$ is an equivalence, its quasi-inverse is given by
$_{A_2}\check{X}_{A_1}\Lo(-):\mc{D}(A_2,H) \lra\mc{D}(A_1,H)$.
\end{prop}
\Pf{By the adjunction \ref{lemma-Rhom-adjoint-to-Ltensor}, if
$X\Lo_{A_1}(-)$ is an equivalence, its quasi-inverse is given by
$\RHom_{A_1}(X_{A_2},-)$. Therefore $\RHom_{A_1}(X_{A_2},-)$
commutes with direct sums, and the corollary now follows from part
two of Lemma \ref{lemma-criterion-functors-natural-transformations}
and the observation we made before this proposition.}

\subsection{A special case} We specialize the previous results to
the case of $H$-module algebras $\phi:A_2\lra A_1$, and the bimodule
$_{A_1}X_{A_2}:={_{A_1}{A_{1}}_{A_2}}$. Here the right $A_2$-module
structure on $A_1$ is realized via the morphism $\phi$, i.e.
$a_1\cdot a_2:=a_1\phi(a_2)$ where $a_i\in A_i$, $i=1,2$.

\begin{defn}\label{def-induction-restriction-functors}We define the
\emph{induction functor}
$$\phi^*:\mc{D}(A_2,H)\lra \mc{D}(A_1,H),~\phi^*(M):= A_1\Lo_{A_2}M$$
and the \emph{restriction functor}
$$\phi_*:\mc{D}(A_1,H)\lra \mc{D}(A_2,H),~\phi_*(N):= \RHom_{A_1}(A_2, N)\cong
{_{A_2}N}.$$ Note that
$\RHom_{A_1}({A_1}_{A_2},N)\cong\Hom_{A_1}({A_1}_{A_2},N)\cong
{_{A_2}N}$ where $A_2$ acts on $N$ via the morphism $\phi$. The
first isomorphism holds since $A_1$ has property (P) as a left
$B_1$-module.
\end{defn}

The derived adjunction (Lemma \ref{lemma-Rhom-adjoint-to-Ltensor})
gives us:
$$\Hom_{\mc{D}(A_2,H)}(\phi^*(N),M)\cong \Hom_{\mc{D}(A_1,H)}(N,\phi_*(M)),$$

We have the following immediate corollary, concerning when a
morphism of $H$-module algebras induces a derived equivalence of
their module categories. The result in the DG case is already proven
in \cite[Theorem 10.12.5.1]{BL}.

\begin{cor} \label{cor-morita-equivalent-hopfological-algebra}
Let $\phi:A_2 \lra A_1$ be a quasi-isomorphism of $H$-module
algebras. Then the induction and restriction functors
$$\phi^*: \mc{D}(A_2, H) \lra \mc{D}(A_1, H),$$
$$\phi_*: \mc{D}(A_1, H) \lra \mc{D}(A_2, H),$$
are mutually-inverse equivalences of triangulated categories.
\end{cor}
\begin{proof}This is a direct consequence of part 1 of Proposition
\ref{prop-derived-functors-from-bimodule-maps}. We give a second
direct proof of this important special case following \cite[Theorem
10.12.5.1]{BL}.

 We will show directly that under our assumption, there are quasi-
isomorphisms of functors:
$$\alpha : \Id_{\mc{D}(A_2,H)} \Rightarrow \phi_*\circ\phi^*,$$
$$\beta : \phi^*\circ\phi_* \Rightarrow \Id_{\mc{D}(A_1,H)}.$$

For this purpose, let $N$ be an object of $\mc{D}(A_2,H)$, and let
$\mathbf{p}N\xrightarrow{\mathbf{p}}N$ be its bar resolution in
$\mc{D}(A_2,H)$. Then set $\alpha:=\mathbf{p}^{-1}\circ \gamma$,
where $\gamma$ is the morphism:
\begin{align*}
\gamma: \mathbf{p}N&\lra A_1\o_{A_2}\mathbf{p}N\\
 n & \mapsto 1\o n.
\end{align*}
Now, $\gamma$ is a quasi-isomorphism since it can be rewritten as
$$\gamma=\phi\o \Id_{\mathbf{p}N}:A_2\o_{A_2}\mathbf{p}N
 \lra A_1\o_{A_2}\mathbf{p}N,$$
and since  $A_1$ and $A_2$ are isomorphic in $H\udmod$.

To define $\beta$, let $M$ be in $\mc{D}(A_1,H)$. $M$ can be
regarded as an object in $\mc{D}(A_2,H)$ via restriction, and we let
$\mathbf{p}M\xrightarrow{\mathbf{p}}M$ be its bar resolution in
$\mc{D}(A_2,H)$. Then $\phi^*\phi_*(M)\cong A_1\o_{A_2}\mathbf{p}M$.
Define $\beta$ to be
\begin{align*}
\beta: A_1\o_{A_2}\mathbf{p}M&\lra M\\
a_1\o m &\mapsto a_1\cdot\mathbf{p}(m).
\end{align*}
To check that it is an isomorphism, consider the commutative diagram
below:
$$\xymatrix{A_2\o_{A_2}\mathbf{p}M\ar[rrd]^{\mathbf{p}}
\ar[d]_{\phi\o \Id_{\mathbf{p}M}}&&\\
A_1\o_{A_2}\mathbf{p}M \ar[rr]^{\beta}& & M.\\}$$ Both $\phi\otimes
\Id_{\mathbf{p}M}$ and $\mathbf{p}$ become isomorphisms under
restriction to $H\udmod$. Therefore $\beta$ is a quasi-isomorphism
of $B_1$-modules, hence an isomorphism in the derived category, as
claimed. The corollary follows.
\end{proof}

\begin{cor} \label{cor-Hopfo-contractibility} Let $A$ be a left
$H-$module algebra. Then $\mc{D}(A,H)\cong 0$ if and only if there
exists an element $x\in A$ such that
$$\Lambda\cdot x  =1.$$
Furthermore, if $x$ is central in $A$, $\mc{C}(A,H)\cong 0$.
\end{cor}
\begin{proof} We will show that, under the assumption, the
$H-$module map $A\stackrel {\lambda_A}{\lra} A \o H$ admits an
$H-$module retract, defined as $$A \o H\lra A,~a\o h\mapsto (h\cdot
r_x)(a),$$ where the $r_x:A\lra A$ is the right multiplication on
$A$ by $x$. Then as shown in the proof of Lemma
\ref{lemma-ideal-null-homotopy}, this is an $H$-module map and we
have
$$\begin{array}{lll}a \o \Lambda &\mapsto &
\Lambda_{(2)}\cdot(r_{x}(S^{-1}(\Lambda_{(1)})\cdot a))\\
& = & \Lambda_{(2)}\cdot(S^{-1}(\Lambda_{(1)})\cdot a\cdot x)\\
& = & (\Lambda_{(2)}\cdot(S^{-1}(\Lambda_{(1)})\cdot a))(\Lambda_{(3)}\cdot x)\\
& = & (\epsilon(\Lambda_{(1)})a)(\Lambda_{(2)}\cdot x) \\
& = & a(\Lambda \cdot x)\\
& = & a~.
\end{array}$$
Therefore, $A$ is contractible as an $H$-module and Corollary
\ref{cor-morita-equivalent-hopfological-algebra} implies that
$\mc{D}(A,H)$ is trivial.

The converse follows by applying Lemma
\ref{lemma-ideal-null-homotopy}, since $A$ itself considered as a
hopfological module is acyclic in this case. The last claim follows
by observing that, if $x$ is central, left multiplication by $x$ on
any $B$-module $M$ is an $A$-module homotopy from $\Id_M$ to zero.
\end{proof}


\section{Special examples} In this section, we apply the
previous results to a very special class of $H$-module algebras on
which $H$ acts trivially. As a consequence we deduce that the
Grothendieck group for the ground field $K_0(\mc{D}^c(\Bbbk,H))$
coincides with $K_0(H\udmod)$.

\subsection{Variants of derived categories}
First off, we introduce the analogue of the usual notion of the
bounded derived category in the hopfological case. For simplicity,
we will only do this when the $H$-module algebra $A$ is noetherian.
Since $H$ is finite dimensional, $B=A\o H$ is a finite $A$-module,
and therefore the noetherian condition on $A$ is equivalent to that
on $B$.

\begin{defn}\label{def-bounded-derived-category} Let $A$ be a noetherian
$H$-module algebra. The bounded derived category $\mc{D}^b(A,H)$ is
the strictly full subcategory of $\mc{D}(A,H)$ consisting of objects
which are isomorphic to some finitely generated $A$-module.

Likewise, define the finite derived category $\mc{D}^f(A,H)$ to be
the strictly full subcategory of $\mc{D}(A,H)$ consisting of objects
which are isomorphic to some finite length $A$-module.
\end{defn}

Notice that if $A$ is finite dimensional, the two notions
$\mc{D}^b(A,H)$ and $\mc{D}^f(A,H)$ coincide with each other. In any
case, it is readily seen that there is an embedding
$\mc{D}^f(A,H)\subset \mc{D}^b(A,H)$, and there is always a
bifunctorial pairing
$$\mc{D}^c(A,H)\times \mc{D}^f(A,H)\lra \mc{D}^f(\Bbbk,H),~~(P, M) \mapsto
\RHom_{A}(P,M),$$ where the category $\mc{D}^f(\Bbbk,H)\subset
H\udmod$ is just the bounded (also finite) derived category of
$\Bbbk$.

\begin{defn}\label{def-Grothendieck-group-bounded-version}Let $A$ be a
noetherian $H$-module algebra. We define the bounded Grothendieck
group of $A$, denoted $G_0(\mc{D}^b(A,H))$ (or $G_0(A,H)$ for short)
to be the abelian group generated by the symbols of isomorphism
classes of objects in $\mc{D}^b(A,H)$, modulo the relations
$$[Y]=[X]+[Z]$$
whenever there is a distinguished triangle inside $\mc{D}^b(A,H)$ of
the form
$$X\lra Y \lra Z \lra T(X).$$
Likewise, we define the finite Grothendieck group
$G_0^f(A,H):=G_0(\mc{D}^f(A,H))$ in an analogous fashion.
\end{defn}

\subsection{Smooth basic algebras}
Now we exhibit a class of examples where the Grothendieck groups
$K_0(A,H)$ can be recovered from the usual Grothendieck group
$K_0(A)$.

\begin{defn} \label{def-basic-algebra} Let $A$ be an (graded) artinian
algebra over a ground field $\Bbbk$. We say that $A$ is \emph{basic
in its Morita equivalence class} if all simple modules over $A$ are
one-dimensional over $\Bbbk$.
\end{defn}

Equivalently, $A$ is basic in its Morita equivalence class if and
only if $A/J(A)\cong \Bbbk\times \cdots \times \Bbbk$, where $J(A)$
is the (graded) Jacobson radical. Here the number of copies of
$\Bbbk$ equals the number of isomorphism classes of simple
$A$-modules, or equivalently, that of indecomposable projective $A$
modules.

\begin{defn} \label{def-smooth-algebra} A $\Bbbk$-algebra $A$
is called \emph{smooth} if it has a finite projective resolution as
an $(A,A)$-bimodule,
\end{defn}

In this section we mainly focus on the class of (graded finite
dimensional) smooth, basic artinian algebras. Some examples of such
algebras are provided by the path algebras over oriented quivers
without oriented cycles. In fact, such path algebras are hereditary
and have length one (i.e. two-term) projective resolutions as
bimodules over themselves. In what follows, we will abbreviate the
above hypothesis on our algebra $A$ by simply saying that
\begin{center}
  $A$ is a \emph{smooth basic algebra},
\end{center}
meaning that it's artinian (or graded finite dimensional), smooth,
and basic in its Morita equivalence class. We will regard such an
$A$ as an $H$-module algebra by letting $H$ act trivially on it.
Notice that a $B$-module may carry some non-trivial $H$-action.

\begin{lemma} \label{lemma-projective-trivially-cofibrant}
Let $A$ be an $H$-module algebra with $H$ acting trivially
on it, and let $P$ be a finitely generated projective $A$-module
with trivial $H$ action. Then given any finite dimensional
$H$-module $V$, $P\o V$ is cofibrant in $B\dmod$.
\end{lemma}
\begin{proof} It suffices to show that $A\o V$ is cofibrant since in
this situation $P$ is a direct summand of $A^n$ (with trivial
$H$-module structure) for some $n\in \N$. The cofibrance of $A\o V$
is clear.
\end{proof}

By the characterization of compact modules in $\mc{D}(A,H)$
(Corollary \ref{cor-characterizing-compact-objects}), compact
cofibrant modules are direct summands of free modules in the derived
category. When $A$ is artinian, the direct summand can be taken in
the abelian category $B\dmod$, as shown in the next result. Note that
here we do not assume the $H$-action on $A$ is trivial.

\begin{lemma}\label{lemma-artinian-idempotent-lifting}
Let $A$ be an artinian $H$-module algebra and $M \in \mc{D}^c(A,H)$ be a compact
object. Then $M$ is isomorphic to a finite projective $A$-module in
the derived category.
\end{lemma}
\begin{proof} A direct summand of a finitely generated free
$A$-module $P$ in the derived category is given by an endomorphism
$e:P\lra P$ such that $e^2=e$ in
$\Hom_{\mc{D}(A,H)}(P,P)=\Hom_{\mc{C}(A,H)}(P,P)$. Therefore, by
Lemma \ref{lemma-ideal-null-homotopy}, $e^2-e=\Lambda \cdot f$ for
some $f\in \Hom_{A}(P,P)$. By the artinian assumption, the
endomorphism algebra of a free module is finite dimensional. Using the
classical Fitting's lemma \footnote{See, for instance Benson
\cite[Lemma 1.4.4]{Be} for the form of the lemma that is used
here.}, we can decompose $P$ into a direct sum of $B$-modules (since
$\Lambda\cdot f$ is a map of $B$-modules),
$$P\cong
\textrm{Im}(\Lambda\cdot f)^{N}\oplus \textrm{Ker}(\Lambda \cdot
f)^{N},$$ for $N$ sufficiently large. Here $\Lambda\cdot f$ acts as
an automorphism on $\textrm{Im}(\Lambda\cdot f)^{N}$, and it acts on
$\textrm{Ker}(\Lambda\cdot f)^{N}$ nilpotently. We may remove the
summand $\textrm{Im}(\Lambda\cdot f)^{N}$ since it is contractible
by Corollary \ref{cor-Hopfo-contractibility}.
$\textrm{Ker}(\Lambda\cdot f)^{N}$ is still a projective $A$-module.
Now $\Lambda\cdot f$ is nilpotent on $\textrm{Ker}(\Lambda\cdot
f)^{N}$ and we may lift the idempotent $e$ easily using Newton's
method, which we leave to the reader as an exercise (see
\cite[Theorem 1.7.3]{Be}).
\end{proof}

Therefore, $\mc{D}^c(A,H)$ consists of modules which are images of
finitely generated, projective $A$-modules under the localization
map. We now look at these modules more closely.

\begin{lemma} \label{lemma-finite-cofibrant-replacement-basic-algebra}
Let $A$ be a smooth basic algebra, and $M$ be a finitely dimensional
$B$-module. Then $M$ is quasi-isomorphic to some finite dimensional
projective $A$-module.
\end{lemma}

Before giving the proof, we recall that the simplicial bar
resolution of $A$ as an $(A,A)$-bimodule results in an infinite
cofibrant hopfological replacement (\ref{thm-bar-resolution}), even
for finitely generated modules over a finite dimensional algebra
$A$. However, the lemma says that if $A$ is smooth, there is instead
a much smaller cofibrant replacement, i.e. a finite dimensional
projective $A$-module. This is made possible since the finite
dimension and smoothness of $A$ provides us with a finite
dimensional projective $(A,A)$-bimodule resolution of $A$ as opposed
to the infinite simplicial bar complex we used before. Moreover, the
proof also shows that this cofibrant replacement is functorial,
in the same way as the bar resolution.

\begin{proof} Since $A$ is smooth, it has a finite projective
$(A,A)$-bimodule resolution $P_\bullet \lra A \lra 0$. Now as in the
bar construction \ref{thm-bar-resolution}, we can lift this
resolution to a hopfological resolution $\widetilde{P}_\bullet \lra
A$, since the differentials in the chain complex are (trivially)
$H$-module maps. Now for each finite dimensional $B$-module $M$, we
tensor this complex with $M$ to obtain $\widetilde{P}_\bullet\o_A
M\lra M \lra 0$. $\widetilde{P}_\bullet \o_A M$ is finite
dimensional since $P_\bullet$, $A$ and $M$ are. It is also cofibrant
by Lemma \ref{lemma-projective-trivially-cofibrant}. The claim
follows.
\end{proof}

\begin{prop}\label{prop-basic-algebra-equivalence-two-derived-categories}
If $A$ is smooth basic, then there is an equivalence of triangulated
categories $$\mc{D}^c(A,H)\cong \mc{D}^f(A,H).$$
\end{prop}
\begin{proof}Lemma \ref{lemma-artinian-idempotent-lifting} shows that
any compact module is isomorphic to a finite dimensional projective
$A$-module. Since $\mc{D}^f(A,H)$ is by definition strictly full,
there is an inclusion functor $\mc{D}^c(A,H)\subset \mc{D}^b(A,H)$.
On the other hand, any object in $\mc{D}^f(A,H)$, being isomorphic
to some finite dimensional module, has a finite cofibrant
replacement by the previous Lemma
\ref{lemma-finite-cofibrant-replacement-basic-algebra}. Hence the
inclusion functor is essentially surjective. The proposition
follows.
\end{proof}

The following corollary is immediate by taking $A=\Bbbk$ in the
above proposition.

\begin{cor}Under the canonical isomorphism $\mc{D}(\Bbbk) \cong H\udmod$,
$\mc{D}^c(\Bbbk)$ is isomorphic to the strictly full subcategory of
$H\udmod$ which consists of objects that are quasi-isomorphic to
finite dimensional $H$-modules. \hfill $\square$
\end{cor}

When $A$ is artinian, the $\RHom$-pairing
$$\RHom(-,-):\mc{D}^c(A,H)\times \mc{D}^f(A,H)\lra H\udmod$$
descends to the Grothendieck groups
$$[\RHom_A(-,-)]:K_0(A,H)\times G_0(A,H)\lra K_0(H\udmod).$$
Denote $R:=K_0(H\udmod)$ for the moment. Notice that if $V$ is a
finite dimensional $H$-module algebra, and $P$, $M$ are $B$-modules,
there is a canonical isomorphism of $H$-modules
$$\Hom_A(P \o V, M)\cong \Hom_A(P,M \o V^*) \cong \Hom_A(P,M)\o V^*.$$
On the Grothendieck group level, this says that the pairing above is
sesquilinear in the sense that it is linear in the second argument,
and $*\ $-linear in the first argument, where
$$*:R\lra R, \ \ \ \ \ [V]\mapsto[V^*]$$
is an involution of the ring $R$.

\begin{prop}\label{prop-basic-algebra-K-zero} Let $A$ be a smooth basic
algebra. Then there is an isomorphism of Grothendieck groups:
$$K_0(\mc{D}^c(A,H))\cong K_0(A)\o_{\Z} K_0(H\udmod),$$
where $K_0(A)$ denotes the usual Grothendieck group of the algebra
$A$. Likewise, when $A$ is graded finite dimensional,
$$K_0(\mc{D}^c(A,H))\cong K_0(A)\o_{\Z[q,q^{-1}]} K_0(H\udmod).$$
\end{prop}
\begin{proof} Let $\{P_i,i=1,\cdots,n\}$ and $\{S_j,j=1,\cdots,n\}$
be a complete list of isomorphism classes of indecomposable
projective and simple $A$-modules respectively, and
$R=K_0(H\udmod)$. Lemma \ref{lemma-artinian-idempotent-lifting} says
that $K_0(A,H)$ as an $R$-module is generated by the symbols
$[P_i]$, $i=1,\cdots,n$. In the usual $K_0(A)$,
$\{[P_i]|i=0,\cdots,n\}$ forms a basis. Thus it suffices to show
that the symbols $[P_i]$ are linearly independent over $R$ in
$K_0(A,H)$. To do this, we use the above sesquilinear pairing
$$[\RHom_A(-,-)]:K_0(A,H)\times K_0(A,H)\lra R.$$
Here we identify $K_0$ with $G_0$ using the previous Proposition
\ref{prop-basic-algebra-equivalence-two-derived-categories}. Since
$A$ is basic, we have
$$\Hom_A(P_i,S_j)=\left\{
\begin{array}{ll}
\Bbbk & i=j; \\
0 & \textrm{otherwise.}
\end{array}
\right.
$$
Since $P_i$ is cofibrant (Lemma
\ref{lemma-projective-trivially-cofibrant}),
$\Hom_A(P_i,S_j) \cong \RHom_A(P_i,S_j)$ (Lemma
\ref{lemma-derived-hom-well-defined}). Hence the sesquilinear
pairing is perfect and $\{[P_i]|i=1,\cdots,n\}$ forms an $R$-basis
of $K_0(A,H)$. The graded analogue is proved in a similar way using
the pairing $\RHOM_A$, and the proposition follows.
\end{proof}

In the special case when $A=\Bbbk$, the proposition says that
$K_0(H\udmod)$ is the Grothendieck ring of the ground field.

\begin{cor}\label{cor-two-notions-Grothendieck-group-coincide-ground-
field} We have an isomorphism of abelian groups: \[K_0(\Bbbk,H)\cong
K_0(H\udmod).\]\vskip-\baselineskip\qed
\end{cor}

\begin{rmk}\label{rmk-two-notions-Grothendieck-group-coincide-ground-field}
When the ring $A$ is a commutative algebra, the usual tensor product
of $A$-modules descends to an internal tensor product on
$\mc{D}^c(A,H)$. On the Grothendieck group level, it turns
$K_0(A,H)$ into a ring (not necessarily commutative). The above
corollary can then be strengthened into an isomorphism of rings. We
leave the details to the reader.
\end{rmk}

\vspace{0.1in}

\noindent { \sl \small You Qi, Department of Mathematics, Columbia
University, New York, NY 10027}

\noindent  {\tt \small email: yq2121@math.columbia.edu}

\end{document}